\documentclass[a4paper,10pt]{amsart}

\usepackage{a4wide}

\usepackage[mathscr]{euscript} 
\usepackage{graphicx}
\usepackage{stackengine}
\usepackage{soul} 

\usepackage[centertags]{amsmath}
\usepackage{amsfonts}
\usepackage{amssymb}
\usepackage{amsthm}
\usepackage{amscd}
\usepackage{hyperref}
\usepackage{algorithmic, algorithm}
\usepackage[all]{xy}

\usepackage{color}
\usepackage[dvipsnames]{xcolor}

\usepackage{mathrsfs} 

\DeclareFontFamily{OT1}{rsfs}{}
\DeclareFontShape{OT1}{rsfs}{n}{it}{<-> rsfs10}{}
\DeclareMathAlphabet{\mathscr}{OT1}{rsfs}{n}{it}

\newcommand{\comment}[1]{}

\renewcommand{\P}{\mathbf{P}}

\newcommand{\Z}{\mathbf{Z}}
\newcommand{\F}{\mathbf{F}}

\newcommand{\A}{\mathbf{A}}

\DeclareFontEncoding{OT2}{}{} 

\DeclareMathOperator{\tr}{tr}

\DeclareMathOperator{\rank}{rank}

\DeclareMathOperator{\PGL}{\mathbf{PGL}}

\DeclareMathOperator{\Gal}{Gal}

\DeclareMathOperator{\id}{id}

\DeclareMathOperator{\Div}{Div}

\DeclareMathOperator{\Cl}{Cl}

\theoremstyle{plain} 
\newtheorem{thm}{Theorem}[section] 
\newtheorem{prop}[thm]{Proposition}
\newtheorem{cor}[thm]{Corollary}
\newtheorem{lem}[thm]{Lemma}

\theoremstyle{definition} 
\newtheorem{defn}[thm]{Definition} 
\theoremstyle{remark} 
\newtheorem{rem}{Remark}

\newcounter{tasknumber}
\newcommand{\task}[2][]{%
  \addtocounter{tasknumber}{1}%
  \begin{center}%
  \framebox[1.1\width]{\begin{minipage}{0.9\textwidth}%
  \textbf{Task \arabic{tasknumber}} \textit{\if!#1(unassigned)!\else (#1)\fi}: {#2}%
  \end{minipage}}%
  \end{center}%
}

\newcounter{assumptionnumber}
\newcommand{\assumption}[2][]{%
  \addtocounter{assumptionnumber}{1}%
  \begin{center}%
  \framebox[1.1\width]{\begin{minipage}{0.9\textwidth}%
  \textbf{Assumption \arabic{assumptionnumber}} \textit{\if!#1!\else (#1)\fi}: {#2}%
  \end{minipage}}%
  \end{center}%
}

\newcommand{\authnote}[2][]{\noindent {\if!#1!  {\bf TODO} \else {\small \bf #1} \fi: #2}}

\renewcommand{\PGL}{{\mathrm{PGL}}}

\renewcommand{\P}{{\mathbf{P}}}
\renewcommand{\div}{{\mathrm{div}}}
\newcommand{\transpose}{{t}}
\newcommand{\Log}{{\mathrm{Log}}}

\newcommand{\algclosure}{{\overline{\F}_q}}
\newcommand{\mfd}{{\mathfrak d}}
\newcommand{\derivation}{{d}}

\begin{document}

\title[Discrete logarithms
in quasi-polynomial time]{Discrete logarithms
in quasi-polynomial time\\
in finite fields of fixed characteristic}
\author{Thorsten Kleinjung}
\address{EPFL IC LACAL, Station 14, CH-1015 Lausanne, Switzerland}
\author{Benjamin Wesolowski}
\address{Cryptology Group, CWI, Amsterdam, The Netherlands}

\maketitle

\begin{abstract}
We prove that the discrete logarithm problem can be solved in quasi-polynomial expected time in the multiplicative group of finite fields of fixed characteristic.
More generally, we prove that it can be solved in the field of cardinality $p^n$ in expected time $(pn)^{2\log_2(n) + O(1)}$. 
\end{abstract}


\section{Introduction}\label{sec:intro}

\noindent
In this article we prove the following theorem.

\begin{thm}\label{thm:main}
Given any prime number $p$ and any positive integer $n$, the
discrete logarithm problem in the group $\F_{p^n}^\times$ can be solved in expected time $(pn)^{2\log_2(n) + O(1)}$.
\end{thm}

Fixing the characteristic $p$, the complexity of solving the discrete logarithm problem in the family of groups $\F_{p^n}^\times$ is then $n^{2\log_2(n) + O\left(1\right)}$. Therefore the discrete logarithm problem in finite fields of fixed characteristic can be solved in quasi-polynomial expected time.
This result significantly improves upon the complexity $L_{p^n}(1/2)$
proved by Pomerance in 1987~\cite{Pom87} --- using the $L$-notation $L_q(\alpha) = \exp\left(O\left((\log q)^\alpha(\log\log q)^{1-\alpha}\right)\right)$. The quasi-polynomial complexity has been conjectured to be reachable since~\cite{BGJT14}, where a first heuristic algorithm was proposed.
More generally, Theorem~\ref{thm:main} implies that for any parameter $\alpha \in (0,1/2)$, discrete logarithms can be computed in expected time $L_{p^n}(\alpha + o(1))$ in any family of fields where $p = L_{p^n}(\alpha)$.

Following the first heuristic algorithm of~\cite{BGJT14}, a new one was proposed in~\cite{GKZ}. The latter algorithm is proven to terminate in  quasi-polynomial expected time for finite fields of fixed characteristic that admit a suitable model. Heuristically, it seems to be easy to compute such a model for any given field, but attempts to prove that it always exists have failed~\cite{Mic19}. Nevertheless, the approach of~\cite{GKZ} has been perceived as the most promising way towards a fully rigorous algorithm. 
Our approach in the present article is similar, and we take advantage of the geometric insights developed in~\cite{KW18}. The main difference with all previous work is that we rely on a different model for the field: one that can be proven to exist, eliminating the need for heuristics. This model is introduced in Section~\ref{sec:model}. The main difficulty is then to construct an algorithm that provably works in this model. The general strategy is similar to that of~\cite{GKZ}, yet their algorithm does not immediately translate to the new model. An overview of the new algorithm is presented in Section~\ref{sec:overview}. The remainder of the article is dedicated to the proof.

\section{A suitable model for the finite field}\label{sec:model}

\noindent
The recent algorithms to compute discrete logarithms in small characteristic all exploit properties of a very particular model for the field. It is assumed that the field is of the form $\F_{q^{d\ell}}$, for a prime power $q$ and integers $d$ and $\ell$, and there exist two polynomials $h_0$ and $h_1$ in $\F_{q^d}[x]$ of degree at most $2$, and an irreducible factor $I$ of $h_1x^q - h_0$ of degree $\ell$. The field is then represented as $\F_{q^{d\ell}} \cong \F_{q^d}[x]/(I)$, and the relation
\begin{equation}\label{eq:keyrelation}
x^q \equiv \frac{h_0}{h_1} \mod I
\end{equation}
is the key ingredient leading to heuristic quasi-polynomial algorithms, assuming that such a model of $\F_{q^{d\ell}}$ can be found where $q$ and $d$ are small enough. A proof that such a model can always be found seems out of reach, therefore we propose to use another one. All we need is a property similar to Equation~\eqref{eq:keyrelation}: applying Frobenius is equivalent to a small degree rational function.

\begin{defn}[Elliptic curve model]
Consider a prime power $q$ and an integer $n > 1$.
Suppose there is
an ordinary elliptic curve $E$ defined over $\F_q$, a rational point $Q \in E(\F_q)$ and an irreducible divisor $\mathscr I$ of degree $n$ over $\F_q$ such that for any $f \in \overline \F_q(E)$ with no pole at $\mathscr I$, one has $f\circ \phi_q \equiv f \circ \tau_Q \mod \mathscr I$, where $\phi_q$ is the $q$-Frobenius and $\tau_Q$ is the translation by $Q$. Then, $\F_{q}[\mathscr I] \cong \F_{q^n}$, and we call $(E,Q,\mathscr I)$ a \emph{$(q,n)$-elliptic curve model} of the field $\F_{q^n}$.
\end{defn}

\begin{rem}
Given two functions $f$ and $g$ with no pole at $\mathscr I$, the congruence $f \equiv g \mod \mathscr I$ means that $\div^+(f-g) \geq \mathscr I$, where $\div^+(h)$ is the positive part of the divisor of $h$ (equivalently, $f$ and $g$ have the same value at the geometric points of $\mathscr I$). The field $\F_{q}[\mathscr I]$ is the residue field at $\mathscr I$, i.e., the quotient $\mathscr O_{\mathscr I}/\mathfrak m_{\mathscr I}$ where $\mathscr O_{\mathscr I} \subset \F_q(E)$ is the subring of functions with no pole at $\mathscr I$, and $\mathfrak m_{\mathscr I} \subset \mathscr O_{\mathscr I}$ is the (maximal) ideal of functions that are zero modulo $\mathscr I$.
\end{rem}

We now show how to construct such a model.
Consider a prime power $q$ and an integer $n > 1$.
Let $E$ be an elliptic curve defined over the finite field $\F_q$, and let $\phi_{q}$ be its $q$-Frobenius. Suppose that $E(\F_q)$ contains a point $Q$ of order $n$. Let 
$$\mathscr Q = \{P\in E(\overline \F_q) \mid \phi_q(P) = P + Q\}.$$
The kernel of the isogeny $\phi_q - \id_E$ is $E(\F_q)$, and $\mathscr Q = (\phi_q - \id_E)^{-1}(Q)$ is a translation of $E(\F_q)$. In particular, $|\mathscr Q| = |E(\F_q)|$. Let $P \in \mathscr Q$ and $i$ any positive integer. Since $\phi_q(P) = P + Q$ and $$\phi_{q^{i}}(P) = \phi_{q^{i-1}}(P + Q) = \phi_{q^{i-1}}(P) + Q,$$ a simple induction yields $\phi_{q^{i}}(P) = P + iQ$. Also, since $Q$ is of order $n$, the isogeny $\phi_{q^{n}}$ is the first Frobenius fixing $P$. The orbit of $P$ under the action of $\phi_q$ is a place of degree $n$ over $\F_q$. Therefore $\mathscr Q$ consists of $|E(\F_q)|/n$ irreducible components of degree $n$ over $\F_q$. If $\mathscr I$ is one of these components, then $\F_{q}[\mathscr I] \cong \F_{q^{n}}$. 
Therefore, a $(q,n)$-elliptic curve model can be constructed from an elliptic curve $E$ containing an $\F_q$-rational point $Q$ of order $n$.\\

Given a finite field of the form $\F_{p^n}$, for a prime number $p$ and an integer $n$, there does not necessarily exist an elliptic curve model for $\F_{p^n}$, but we show in the following that one can find an extension of that field of degree logarithmic in $n$ which does admit an elliptic curve model. 
The construction relies on the following theorem.
\begin{thm}[\protect{\cite[Theorem 4.1, condition (I)]{Waterhouse1969}}]\label{thmWaterhouse}For any integer $t$ coprime to $q$ such that $|t| \leq 2q^{1/2}$, there is an ordinary elliptic curve $E$ defined over $\F_q$ such that $|E(\F_q)| = q+1-t$.
\end{thm}
We deduce the following proposition.

\begin{prop}\label{prop:pointofordern}
Let $n \leq \sqrt{2} q^{1/4}$ be a non-negative integer. There exists an ordinary elliptic curve defined over $\F_q$ containing an $\F_q$-rational point of order $n$.
\end{prop}
\begin{proof}
We first prove that there is an elliptic curve $E$ over $\F_q$ such that $n^2$ divides $|E(\F_q)|$. Since $n^2 \leq 2 q^{1/2}$, there exists an integer $m$ such that $|q+1-mn^2|\leq 2q^{1/2}$ and $|q+1-(m+1)n^2|\leq 2q^{1/2}$.
Either $q+1-mn^2$ or $q+1-(m+1)n^2$ is coprime to $p$, so by Theorem \ref{thmWaterhouse}, there is an ordinary elliptic curve over $\F_q$ with either $mn^2$ or $(m+1)n^2$ rational points.

We have shown that there is an elliptic curve $E$ defined over $\F_q$ such that $n^2$ divides $|E(\F_q)|$. From~\cite[Corollary 6.4]{Silverman-Arithmetic}, there are two integers $a$ and $b$ such that the group of rational points $E(\F_q)$ is isomorphic to $\Z/a\Z \oplus \Z/ab\Z$. Then, $n^2$ divides $a^2b$, so $n$ divides $ab$. Therefore $E(\F_q)$ contains a point of order $n$.
\end{proof}

\begin{thm}\label{thm:good-rep}
For any prime number $p$ and integer $n$, one can find in deterministic polynomial time in $p$ and $n$ an integer $r = {O(\log(n))}$ and a $(p^r,n)$-elliptic curve model of the finite field $\F_{p^{rn}}$.
\end{thm}

\begin{proof}
Let $r$ be a positive integer and $q = p^r$.
From Proposition~\ref{prop:pointofordern}, the existence of an elliptic curve model is ensured whenever $n\leq\sqrt{2} p^{r/4}$, which holds whenever $r \geq (4\log (n) - \log (4))/\log (p)$.
Therefore, the construction of the elliptic curve model is as follows: let
$$r= \left\lceil \frac{4\log(n) - \log (4)}{\log (p)} \right\rceil,$$
and $q = p^r$. Find an elliptic curve $E$ defined over $\F_q$ and a point $Q \in E(\F_q)$ of order $n$. As $q$ is polynomial in $p$ and $n$, these can be found in deterministic polynomial time by an exhaustive search. Finally, let $\mathscr I$ be any irreducible component of $\mathscr Q = \{P\in E(\overline \F_q) : \phi_q(P) = P + Q\}$.
\end{proof}

In the rest of this article, we suppose that the elliptic curve $E$ is in (generalised) Weierstrass form, so that we naturally have coordinates $x$ and $y$ such that $x$ is of degree $2$ and $y$ of degree $3$, and for any $P \in E$, we have $x(P) = x(-P)$.

\section{Overview}\label{sec:overview}

\noindent
The following theorem, summarising a series of refinements~\cite{EG02,diem_2011,GKZ}, shows that to obtain an algorithm to compute discrete logarithms, it is sufficient to have a descent procedure.
\begin{thm}[\protect{\cite[Theorem~1.4]{WesThesis}}]\label{thm:descentsufficient}
Consider a finite cyclic group $G$ of order $n$. 
Assume we are given a set $\mathfrak F = \{f_1,\dots,f_m\} \subset G$ (called the \emph{factor base}), for some integer~$m$, and an algorithm $\textsc{Descent}$ that on input $f \in G$ outputs a sequence $(e_{j})_{j=1}^m$ such that $f = \prod_{j=1}^mf_j^{e_{j}}$.
Then, there is a probabilistic algorithm  
that computes discrete logarithms in $G$ at the expected cost of $O(m\log \log n)$ calls to the descent procedure $\textsc{Descent}$, and an additional $O(m^3\log \log n)$ operations in $\Z/n\Z$.
\end{thm}
Therefore, to prove Theorem~\ref{thm:main}, it is sufficient to devise an efficient descent algorithm. Fix a $(q,n)$-elliptic curve model $(E,Q,\mathscr I)$ for the finite field $\F_{q^n}$. 

\subsection{Logarithms of divisors}
The notion of logarithm can be extended from field elements to divisors of the elliptic curve as follows.
Let $N = |E(\F_{q})|$. For any field extension $k/\F_{q}$, let $\Div^0_{k}(E)$ be the group of degree zero divisors of $E$ defined over $k$, and let $\Div^0_{k}(E,\mathscr I)$ be the subgroup of divisors which do not intersect~$\mathscr I$. Given a point $P\in E$, the corresponding divisor is written $[P]$. 
Let $\ell$ be the largest divisor of $q^n - 1$ coprime to $N$.
We can focus on the problem of computing discrete logarithms modulo $\ell$. Indeed, since $N = O(q)$, any prime divisor of $N$ can be handled by the baby-step giant-step method in polynomial time in $q$, and we can apply the Pohlig-Hellman method to compute the `full' discrete logarithms.
We denote by $\log$ the logarithm function modulo $\ell$, with respect to an arbitrary generator of the multiplicative group of the finite field.
We have the following commutative diagram where each line is exact 
\begin{equation*}
\xymatrix{
1 \ar@{->}[r] & \F_{q}^\times \ar@{->}[r] \ar@{->}[d] & \F_{q}(E)_{\mathscr I}^\times \ar@{->}[r]^{\mathrm{div}} \ar@{->}[d]_{\log} & \Div^0_{\F_{q}}(E,\mathscr I) \ar@{->}[r]^{\sigma}  \ar@{..>}[d]^{\mathrm{Log}} & E(\F_{q}) \ar@{->}[r] &  0\\
0 \ar@{->}[r] & 0 \ar@{->}[r] & \Z/\ell\Z  \ar@{->}[r]^{\id} &  \Z/\ell\Z\ar@{->}[r] & 0,&
}
\end{equation*}
where $\F_{q}(E)_{\mathscr I}^\times$ is the multiplicative group of rational functions on $E$ defined over $\F_{q}$ whose divisors do not intersect $\mathscr I$.
The function $\mathrm{Log}$ sends any divisor $D \in \Div^0_{\F_{q}}(E,\mathscr I)$ to the element $\log(f)/N$, where $f$ is any function with divisor $ND$ (which is principal). Given an effective divisor $D$ not intersecting $\mathscr I$, we also define $\mathrm{Log}(D) = \mathrm{Log}(D  - \deg(D)[0_E])$.

Let $\mathscr {D}_i = E^i/\mathfrak S_i$ be the variety of degree $i$ effective divisors on $E$, where $\mathfrak S_i$ is the $i$-th symmetric group. 
Let $\mathscr P_i \subset \mathscr {D}_i$ be the subvariety of divisors $\sum_{i}[P_i]$ such that $\sum_{i}P_i = 0_E$.
Given two subvarieties $\mathscr A \subset \mathscr D_n$ and $\mathscr B \subset \mathscr D_m$, we write $\mathscr A + \mathscr B = \{A+B \mid A \in \mathscr A, B \in \mathscr B\} \subset \mathscr D_{m+n}$. Given a point $P\in E$, we define $\mathscr P_2(P) = \{[P_0] + [P_1] \mid P_0 + P_1 = P\} \subset \mathscr D_2$.

\subsection{Elimination and zigzag}\label{subsec:elimandzigzag}
Consider a field extension $k/\F_{q}$ and a divisor $D \in \mathscr {D}_n(k)$. 
A \emph{degree $n$--to--$m$ elimination} is an algorithm that finds a list $(D_i)_{i = 1}^t$ of divisors over $E$ of degrees at most $m$ and integers $(\alpha_i)_{i = 1}^t$ such that
$$\mathrm{Log}(N_{k/{\F_q}}(D)) = \sum_{i=1}^t \alpha_i \cdot \mathrm{Log}(N_{k/{\F_q}}(D_i)).$$
The integer $t$ is called the \emph{expansion factor} of the elimination.
To build a descent algorithm, we first construct degree $4$--to--$3$ and $3$--to--$2$ elimination algorithms (in Propositions~\ref{prop:4to3} and~\ref{prop:3to2} respectively) with expansion factors at most some value $C$.
Combining these two eliminations, we obtain a degree $4$--to--$2$ elimination algorithm with expansion factor at most $C^2$. A descent can then be constructed following the \emph{zigzag} approach developed in~\cite{GKZ}, as done in Proposition~\ref{prop:zigzag}.
The idea is the following. The logarithm of the finite field element that we wish to descend is first represented as the logarithm of an irreducible divisor $D$ over $\F_q$ of degree a power of two, say $2^{e+2}$. Over the field $\F_{q^{2^e}}$, the divisor $D$ splits as $2^e$ irreducible divisors of degree $4$. If $D'$ is any of these, then $D = N_{\F_{q^{2^e}}/{\F_q}}(D')$. Applying the degree $4$--to--$2$ elimination to $D'$, the value $\Log(D)$ can be rewritten as a linear combination of logarithms $\mathrm{Log}(N_{\F_{q^{2^e}}/{\F_q}}(D_i))$ where each $D_i$ has degree $2$. Now, taking the norm of each $D_i$ to the subfield $\F_{q^{2^{e-1}}}$, we obtain divisors of degree $4$ again, but over a smaller field. One can apply the degree $4$--to--$2$ elimination recursively, until all the divisors involved are of small degree, over a small field $\F_{q^{2^c}}$ (with $c = O(1)$). These small divisors form the set
$$\widetilde{\mathfrak F} = \{N_{\F_{q^{2^c}}/\F_q}(D) \mid D \in \Div_{\F_{q^{2^c}}}(E,\mathscr I), D > 0, \deg(D) \leq 2\}.$$
We can finally rewrite our logarithm as a combination of logarithms of elements of the factor base
$$\mathfrak F = \{f \in \F_q[E] \mid \exists D \in \widetilde{\mathfrak F} \text{ such that } \mathrm{div}(f) = ND - \deg(f)[0_E]\}.$$

One difficulty in this approach is that the elimination algorithms might fail for certain divisors, which we call \emph{traps}. We show that traps are rare, in the sense that they form a proper sub-variety of $\mathscr {D}_4$ or $\mathscr {D}_3$ of bounded degree. The descent must then carefully avoid traps. In particular, we show that given a divisor that is not a trap, an elimination allows to rewrite it in terms of smaller degree divisors that are themselves not traps --- otherwise the descent could reach a dead end.

\begin{rem}
Note that along the descent, we encounter divisors of degree at most $4$ over extensions of $\F_q$ of degrees powers of $2$. We are therefore not worried of encountering any divisor intersecting~$\mathscr I$ --- at which $\Log$ would not be defined --- so long as the order of $Q$ is not divisible only by $2$ and $3$ (which can be enforced, by replacing if necessary the extension degree $n$ by $5n$).
\end{rem}

\subsection{Degree 3--to--2 elimination}\label{subsec3to2}
Consider an extension $k/\F_{q}$ and a divisor $D \in \mathscr {D}_3(k)$. 
Let $V = \mathrm{span}(x^{q+1},x^q,x,1)$.
We define the morphisms $\varphi_P$ for any $P \in E$ as
$$\varphi_P : V \longrightarrow \algclosure(E) :
\begin{cases}
x^{q+1} &\longmapsto (x\circ\tau_{Q+P^{(q)}}) \cdot (x\circ\tau_{P}),\\
x^q &\longmapsto x\circ\tau_{Q+P^{(q)}},\\
x &\longmapsto x\circ\tau_{P},\\
1 &\longmapsto 1.
\end{cases}$$
These linear morphisms are chosen so that for any vector $f \in V$ and point $P\in E$, we have the relation $\varphi_P(f) \equiv f \circ \tau_P \mod \mathscr I$.
Now, define the algebraic variety
\begin{alignat*}{1}
X_0 &= \{(f,P) \mid \varphi_P(f) \equiv 0 \mod D\} \subset  \P(V) \times E.
\end{alignat*}
We will see that it is a curve. Let $(f,P) \in X_0(k)$ be one of its $k$-rational points. We will prove that there are many such rational points where the polynomial $f$ splits into linear factors over $k$, i.e., $f = \prod_{i = 1}^{q+1}L_i$ with $L_i$ linear over $k$. Assuming this is the case, then we have a 3--to--2 elimination. Indeed, on one hand,
$$\log(\varphi_P(f)) = \log(f \circ \tau_P) =  \sum_{i = 1}^{q+1}\log(L_i \circ \tau_P).$$
On the other hand, from the definition of $X_0$ and the fact that $\varphi_P(f)$ has degree $4$, we have $\mathrm{div}(\varphi_P(f)) = D + [P'] - 2[-P] - 2[-Q-P^{(q)}],$
where $P'$ is a point of $E(k)$. We deduce
\begin{alignat*}{1}
\Log(N_{k/\F_q}(D)) = &\ \log(N_{k/\F_q}(\varphi_P(f))) - \Log(N_{k/\F_q}([P'] - 2[-P] - 2[-Q-P^{(q)}]))\\
 = &\ \sum_{i = 1}^{q+1}\log(N_{k/\F_q}(L_i \circ \tau_P)) - \Log(N_{k/\F_q}([P']))\\
&\  + 2\cdot\Log(N_{k/\F_q}([-P])) + 2\cdot\Log(N_{k/\F_q}([-Q-P^{(q)}])).
\end{alignat*}
The right-hand side is a sum of logarithms of divisors of degree $1$ or $2$ over $k$. Therefore, the 3--to--2 elimination algorithm simply consists in constructing $X_0$, and pick uniformly at random rational points $(f,P) \in X_0(k)$ until $f$ splits as a product of linear terms. It remains to prove that this happens with good probability. This is formalised in Proposition~\ref{prop:3to2}.

\subsection{On the action of $\PGL_n$ and splitting probabilities}
For the 3--to--2 elimination sketched above to work, we rely on the idea that for $(f,P) \in X_0(k)$, the polynomial $f$ splits into linear factors over $k$ with good probability. This polynomial $f$ has degree $q+1$, so at first glance it seems it should split with very small probability $1/(q+1)!$. However, and this is the key of previous (heuristic) quasi-polynomial algorithms, the polynomials in $V$ have a very particular structure and a fraction $1/O(q^3)$ of them split over $k$. This high splitting probability can be understood from the action of $\PGL_2$ on $\P(V)$.
We denote by $\star$ the action of invertible $2 \times 2$ matrices on univariate polynomials defined as follows:
$$\left(\begin{matrix}a&b \\ c & d\end{matrix}\right)\star f(x) = (cx+d)^{\deg f}f\left(\frac{ax+b}{cx+d}\right).$$
It induces an action of $\PGL_2$ on $\P(V)$, also written $\star$. The space $\P(V)$ is the closure of the orbit $\PGL_2 \star (x^q - x)$, and if $m \in \PGL_2(k)$, then $m \star (x^q-x)$ splits as a product of linear polynomials over $k$, which allows to deduce that a significant portion of the polynomials in $\P(V)(k)$ split over~$k$. These observations are developed and exploited in~\cite{KW18}.

This idea is sufficient for the previous heuristic algorithms and for our 3--to--2 elimination, but to obtain a rigorous 4--to--3 elimination algorithm, we need to work with higher dimensional objects. Let $V_n = \mathrm{span}(x_i^qx_j \mid i,j \in \{ 0,1,\dots,n-1 \})$. Then, $n\times n$ matrices naturally act on these polynomials by substituting $x_i$ with the scalar product of the $i$-th row with $(x_0,x_1,\dots,x_{n-1})^\transpose$. This induces an action of $\PGL_n$ on $\P(V_n)$, written $\star$.
Let $\mathfrak d_n = x_0^qx_1 - x_1^qx_0 \in \P(V_n)$.
The orbit $\PGL_n\star \mathfrak d$ is a subvariety of $\P(V_n)$, but as soon as $n>2$, this orbit is not dense anymore.
However, as illustrated in the following lemma, it remains the relevant subvariety to consider as we wish to find polynomials that split into linear factors.

\begin{lem}\label{lem:lineardecomposition}
The only polynomials in $V_n$ with 3 distinct linear factors are in the orbit $\PGL_n\star \mathfrak d_n$. The only polynomials with a double linear divisor are in the orbits $\PGL_n\star(x_0^qx_1)$ and $\PGL_n\star x_0^{q+1}$.
\end{lem}

\begin{proof}
For the first part, suppose that the three factors are not collinear, and apply the action of a matrix sending them to $x_0,x_1$ and $x_2$. The resulting polynomial is divisible by $x_0x_1x_2$, a contradiction.
So the three factors must be collinear, and send them to $x_0$, $x_1$ and $x_0 + x_1$. For the second part, send the double divisor to $x_0^2$.
\end{proof}

Consider the vector space $\Lambda_n = \mathrm{span}(x_i \mid i \in \{0,\dots, n-1\})$ of linear polynomials.
As in the $\PGL_2$ case, we have that for any $m \in \PGL_n(k)$, the polynomial $m\star \mathfrak d_n$ splits into linear factors over $k$. 
Before sketching how to use these observations to build a 4--to--3 elimination algorithm, we note that the closure of $\PGL_n\star \mathfrak d_n$ is well understood: it consists of $\PGL_n\star \mathfrak d_n$ itself and the closure of $\PGL_n\star(x_0^qx_1)$ (one way to see this is to show that the subvariety of $\P(V_n) \times \P(\Lambda_n)^3$ of points $(f,\ell_1,\ell_2,\ell_3)$ where $\ell_1\ell_2\ell_3$ divides $f$ is closed, and apply Lemma~\ref{lem:lineardecomposition}).
Now, the closure of $\PGL_n\star(x_0^qx_1)$ is the image of the morphism
\begin{alignat*}{1}
\Xi: \P(\Lambda_n) \times \P(\Lambda_n) &\longrightarrow \P(V_n)
: (v,u) \longmapsto v^qu.
\end{alignat*}
This closure therefore coincides with the image of the Segre embedding of $\P(\Lambda_n) \times \P(\Lambda_n) \cong \P^{n-1} \times \P^{n-1}$ into $\P(V_n) \cong \P^{n^2-1}$.

The points in $\PGL_n\star(x_0^qx_1)$ are called the \emph{exceptional points} of the closure of $\PGL_n\star \mathfrak d_n$, and they play a crucial role in our analysis of the descents. In particular, a divisor being a trap or not is closely  related to the properties of the exceptional points that appear in $X_0$.

\subsection{Degree 4--to--3 elimination algorithm}\label{subsec:overview-4-3}
Consider an extension $k/\F_{q}$ and a divisor $D \in \mathscr {D}_4(k)$.
Consider the vector space $V = V_3$ as defined above, the element $\mathfrak d = x_0^qx_1 - x_0x_1^q \in V$, and its orbit $\PGL_3 \star \mathfrak d \subset \P(V)$.
Define the morphism $\psi : V \rightarrow \algclosure[E]$ which substitutes $x_0,x_1$ and $x_2$ with $1$, $x$, and $y$ respectively. 
Now, define the morphism $\varphi: V \rightarrow \algclosure(E)$ with $\varphi(x_i^qx_j) = (\psi(x_i)\circ \tau_Q) \cdot \psi(x_j)$.
Observe that $\varphi(f) \equiv \psi(f)\mod \mathscr I$ for any $f\in V$.
Define
$$X_0 = \{f \in \overline{\PGL_3 \star \mathfrak d} \mid \varphi(f) \equiv 0 \mod D\}.$$
Let $f \in X_0(k)$. 
As briefly justified in the previous paragraph, when $f$ is in the orbit $\PGL_3 \star \mathfrak d$, we can expect it to split into linear factors over $k$ with good probability, i.e., $f = \prod_{i = 1}^{q+1}L_i$ with $L_i \in \Lambda_3 = \mathrm{span}(x_j \mid j \in \{0,1,2\})$. When this happens, we have a 4--to--3 elimination. Indeed, on one hand,
$$\log(\varphi(f)) = \log(\psi(f)) =  \sum_{i = 1}^{q+1}\log(\psi(L_i)).$$
On the other hand,
$$\mathrm{div}(\varphi(f)) = D + D' - 3[0_E] - 3[-Q],$$
where $D'$ is an effective divisor of degree $2$ defined over the field $k$. We deduce
\begin{alignat*}{1}
\Log(N_{k/\F_q}(D)) &= \log(N_{k/\F_q}(\varphi(f))) - \Log(N_{k/\F_q}(D')) + 3\cdot\Log(N_{k/\F_q}([Q]))\\
& = \sum_{i = 1}^{q+1}\log(N_{k/\F_q}(\psi(L_i))) - \Log(N_{k/\F_q}(D')) + 3\cdot[k:\F_q]\cdot\Log([Q]).
\end{alignat*}
The right-hand side is a sum of logarithms of divisors of degree $1$, $2$ of $3$ over $k$. Therefore, the 4--to--3 elimination algorithm consists in constructing $X_0$ and pick uniformly at random rational points $f \in X_0(k)$ until $f$ splits as a product of linear terms. We need to prove that this happens with good probability. This is formalised in Proposition~\ref{prop:4to3}.

\subsection{Traps}
The two types of elimination sketched above work for `most' degree $3$ and degree $4$ divisors. There are however certain divisors for which we cannot guarantee that the elimination succeeds: these \emph{trap divisors} form subvarieties $\mathscr T_3 \subset \mathscr D_3$ and $\mathscr T_4 \subset \mathscr D_4$.
When $D$ is not a trap divisor, we can prove that the elimination succeeds, but another problem might arise: it could be that all possible eliminations of this divisor involve traps, so the descent cannot be iteratively applied.
We deal with this issue in Section~\ref{sec:traps}.

\subsection{A general approach for elimination}
The elimination algorithms above are special
cases of a more general construction which will be sketched in the
following.
Although only the two special cases are used in the rest of the article, the general perspective may provide the reader with some useful insight.
Let $n \ge 2$ be an integer,
$\phi_E=\id \times \phi_{q} \times \cdots \times \phi_{q^{n-1}}:
E \longrightarrow E^n$ and
$\tau=\id \times \tau_Q \times \cdots \times \tau_{(n-1)Q}:
E \longrightarrow E^n$.
If $f: E \longrightarrow X$ is a morphism to a variety $X$, we have the
commutative diagram
$$
\xymatrix{
E \ar[d]^\tau \\
E^n \ar[r]^{f_{\phi}} & X^n \\
E \ar[u]^{\phi_E} \ar[r]^f & X \ar[u]^{\phi_X}
}
$$
with $f_{\phi}=f \times f^{(q)} \times \cdots \times f^{(q^{n-1})}$
and $\phi_X=\id \times \phi_{q} \times \cdots \times \phi_{q^{n-1}}$.
For any rational function $g$ of the variety $X^n$ we get
\begin{equation}\label{relation}
\Log((f_{\phi} \circ \tau)^*(\div(g))) =
\Log((f_{\phi} \circ \phi_E)^*(\div(g))) =
\Log((\phi_X \circ f)^*(\div(g))).
\end{equation}
In order to obtain an elimination algorithm one constucts sufficiently
many morphisms $f$ satisfying the following two restrictions:
\begin{itemize}
\item
The divisor $(f_{\phi} \circ \tau)^*(\div(g))$ is of low
degree.
For this we choose the morphism $f$ to be sufficiently simple
(of ``low degree'' in some sense).
\item
The divisor $\phi_X^*(\div(g))$ decomposes into ``simple''
(``low degree'' in some sense) divisors.
If $f$ is chosen as above, the pullback to $E$ also decomposes
into low degree divisors.
\end{itemize}
One can construct many morphisms $f$ satisfying the first restrictions by choosing
them as a composition of several morphisms such that some of the morphisms
are automorphisms of varieties with large automorphisms groups (e.g.,
$\P^m$ or genus one curves). To eliminate a divisor $D$, one then has to find one such morphism $f$ such that $D$ appears in the decomposition of $(f_{\phi} \circ \tau)^*(\div(g))$.

The degree $3$--to--$2$ and $4$--to--$3$ eliminations are two special cases.
Let $X=\P^{n-1}$ with coordinates $(x_0: \ldots : x_{n-1})$,
let $2 \le m \le n$ and, denoting the coordinates of
the $i$-th factor of $X^n$ by $(x_{0,i}: \ldots : x_{n-1,i})$, let
$g=\frac{\det(D_m)}{\prod_{i = 0}^{m-1}x_{0,i}}$ with
$D_m=(x_{i,j})_{0 \le i <m, 0 \le j <m}$.
Then it follows that $\phi_X^*(g)$ has a pole of order $1+\ldots+q^{m-1}$
at $x_0=0$ and simple zeros at $\sum_{i=0}^{m-1} a_ix_i=0$ for
$(a_0: \ldots : a_{m-1}) \in \P^m(\F_q)$.
\begin{enumerate}

\item
With $n = 1$, let $X=\P^1$ and $g$ be as above.
Choose $f$ to be the composition
$$\xymatrix{E \ar[r]^{f_1} & E \ar[r]^{x} & \P^1 \ar[r]^{f_2} & \P^1}$$
with $f_1$ being a translation and $f_2$ being an element of $\PGL_2$.
Then the positive part of the divisor on the left hand side
of \eqref{relation} has degree $4$ and the positive part on the
right hand side decomposes into $q+1$ divisors of degree $2$, and we
obtain the $3$--to--$2$ elimination described above.

\item
Let $f$ be the composition
$$\xymatrix{E \ar[r]^{f_1} & \P^{n-1} \ar[r]^{f_2} & \P^{n-1}}$$
where $f_2$ is an element of the automorphism group $\PGL_n$.
For $n=3$ one may choose $m=2$ and $f_1$ to be the usual embedding
so that the positive part of the divisor on the left hand side
of \eqref{relation} has degree $6$ and the positive part on the
right hand side decomposes into $q+1$ divisors of degree $3$;
this leads to the $4$--to--$3$ elimination described above.
One may also choose $m=3$ and $f_1 = \imath \circ x$
where $\imath$ is any embedding of $\P^1$ into $\P^2$,
implying that the relevant degrees are $6$ and $2$ respectively,
which will give a $4$--to--$2$ elimination; this is not considered in this
paper.
\end{enumerate}

\section{Degree $3$--to--$2$ elimination}
\noindent
In this section, we consider a degree $3$ divisor $D$ on $E$, defined over $k$.
Note however that the main ideas, and notably the roadmap presented in Section~\ref{subsec:roadmap}, also apply to the degree $4$--to--$3$ elimination.
We suppose $D$ does not belong to a set of exceptional divisors, the \emph{traps} $\mathscr T_3 \subset \mathscr D_3$, defined in Section~\ref{subsec:summary-traps-3to2}.
Consider the vector space $V = \mathrm{span}(x^{q+1},x^q,x,1)$ in the algebra $\algclosure[x]$.
As explained in~\ref{subsec3to2}, we can associate to the divisor $D$ a variety
$$X_0 = \{(f,P) \mid \varphi_P(f) \equiv 0 \mod D\} \subset  \P(V) \times E,$$
and our goal in this section is to prove that for a significant proportion of the pairs $(f,P) \in X_0(k)$, the polynomial $f$ splits into linear terms over $k$.
The general strategy is similar to that of~\cite{KW18}: we define a curve $C$ and a morphism $C \rightarrow X_0$ such that the image of any $k$-rational point of $C$ is a pair $(f,P)$ such that $f$ splits into linear terms over $k$. Such a curve $C$ can be defined in $\P(V) \times E \times \P^1 \times \P^1 \times \P^1$ as 
$$C = \{(f,P,r_1,r_2,r_3) \mid (f,P) \in X_0, \text{ and the }r_i\text{-values are three distinct roots of }f\}.$$
Similarly to~\cite[Proposition~4.1]{KW18}, Lemma~\ref{lem:lineardecomposition} implies that if $(f,P,r_1,r_2,r_3) \in C(k)$, then $f$ splits into linear factors over $k$ (and therefore leads to an elimination, as explained in Section~\ref{subsec3to2}).\\
Before proceeding, we note that the definition of $X_0$ above works only on the open set of pairs $(f,P)$ where $P \not\in \{-D_i \mid i = 1,...,4\}$ and $P^{(q)} \not\in\{-D_i-Q \mid i = 1,...,4\}$, otherwise $\varphi_P(f)$ might have a pole on $D$. The closure can actually be defined as
\begin{alignat*}{1}
X_0 &= \{(f,P) \mid \varphi_P(f) \in L(2[-P] + 2[-Q-P^{(q)}] - D) \},
\end{alignat*}
where $L(Z)$ is the Riemann-Roch space of rational functions $f$ such that $\div(f) \geq -Z$.

\subsection{Roadmap}\label{subsec:roadmap} We need to show that $C$ has a lot of $k$-rational points.
It is sufficient to prove that $C$ has at least one absolutely irreducible component defined over $k$, then apply Hasse-Weil bounds.
There again, our strategy draws inspiration from~\cite{KW18}.
Instead of considering directly $C$, whose points encode triples of roots, we start with the following variety which considers a single root at a time:
$$X_1 = \{(f,P,r) \mid (f,P) \in X_0, \text{ and }f(r) = 0\} \subset \P(V) \times E \times \P^1.$$
We can then increase the number of roots by considering fibre products over the projection ${\theta : X_1 \rightarrow X_0}$. Indeed, we have
\begin{alignat*}{1}
X_1 \times_{X_0} X_1 &= \{((f_1,P_1,r_1),(f_2,P_2,r_2)) \mid (f_1,P_1) = (f_2,P_2) \in X_0, \text{ and }f_1(r_1) = f_2(r_2) = 0\}\\
&\cong \{(f,P,r_1,r_2) \mid (f,P) \in X_0, \text{ and }f(r_1) = f(r_2) = 0\}.
\end{alignat*}
This product contains a \emph{diagonal} component $\Delta_{X_1}$ isomorphic to $X_1$, which corresponds to quadruples $(f,P,r,r)$.
The other components $X_2 = X_1 \times_{X_0} X_1 \setminus \Delta_{X_1}$ encode pairs of distinct roots (points of the form $(f,P,r,r)$ can still appear in $X_2$, but they imply that $r$ is a double root of $f$). We can iterate this construction, and consider the product $X_2 \times_{X_1} X_2$ over the projection $X_1 \times_{X_0} X_1 \rightarrow X_1$ to the first factor. This product encodes triples of roots, and the curve $C$ embeds into the non-diagonal part $X_3 = X_2 \times_{X_1} X_2\setminus \Delta_{X_2}$.
In the rest of this section, we prove sequentially that $X_0$, $X_1$, $X_2$, $X_3$ and $C$ contain absolutely irreducible components defined over $k$.\\

The following lemma allows us to prove irreducibility results through fibre products.
\begin{lem}\label{lem:updated-lemma-KW18}
Let $Y$ and $Z$ be two absolutely irreducible, complete curves over $k$, and consider a cover $\eta : Z \rightarrow Y$.
Suppose there is a point $s \in Y$ and two distinct points $a,b \in Z$ such that $\eta^{-1}(s) = \{a,b\}$.
If $s, a$ and $b$ are analytically irreducible, and the normalisation of $\eta$ is unramified at $a$, then $Z \times_Y Z \setminus \Delta_Z$ is absolutely irreducible, where $\Delta_Z$ is the diagonal component.
\end{lem}

\begin{proof}
The same proof as~\cite[Lemma~4.2]{KW18} implies this result, where smoothness is replaced by analytic irreducibility (both imply that a point belongs to a single irreducible component).
\end{proof}

\begin{rem}
The term \emph{analytically} refers to properties of the completion of the local ring. A point is \emph{analytically irreducible} if the completion of the corresponding local ring has no zero divisors (equivalently, a single branch passes through this point: it desingularises as a single point).
\end{rem}

The following proposition defines our strategy: the rest of our analysis of the $3$--to--$2$ elimination consists in showing that our cover $\theta : X_1 \rightarrow X_0$ satisfies the necessary conditions to apply Proposition~\ref{prop:structure-of-proof}.

\begin{prop}\label{prop:structure-of-proof}
Let $X_0$ and $X_1$ be complete curves over $k$, and suppose $X_0$ is absolutely irreducible. Let  $\theta : X_1 \rightarrow X_0$ be a cover of degree at least $3$.
Let $X_2 = (X_1 \times_{X_0} X_1) \setminus \Delta_{X_1}$, and $X_3 = (X_2 \times_{X_1} X_2) \setminus \Delta_{X_2}$ (for the projection $X_2\rightarrow X_1$ to the first factor).
Suppose that
\begin{enumerate}
\item there is a point $s \in X_0$ and two distinct points $a,b \in X_1$ such that $\theta^{-1}(s) = \{a,b\}$,
\item the points $a,b \in X_1$ and $(b,b) \in X_2$ are analytically irreducible,
\item\label{item:3:prop:structure-of-proof} the normalisation of the cover $\theta$ is unramified at $a$.
\end{enumerate}
Then, either
\begin{enumerate}
\item\label{prop:point-x1irred} the curve $X_1$ is absolutely irreducible, and so is $X_3$, or
\item\label{prop:point-x1notirred} the curve $X_1$ is the union of two absolutely irreducible components $A$ and $B$, with $a \in A$ and $b\in B$, and
$$(B\times_{X_0} A) \times_B ((B\times_{X_0} B) \setminus \Delta_B)$$
is an absolutely irreducible component of $X_3$ defined over $k$.
\end{enumerate}
\end{prop}

\begin{rem}Applied to the cover $\theta : X_1 \rightarrow X_0$ defined above, we choose $s \in X_0$ to be one of the exceptional points in $X_0 \cap S$, which have the form $((x-\alpha)(x-\beta)^q, P)$. Then, $$a = ((x-\alpha)(x-\beta)^q, P, \alpha), \text{ and } b = ((x-\alpha)(x-\beta)^q, P, \beta).$$
We then prove that all the conditions of the proposition are satisfied, which implies that $X_3$ contains an absolutely irreducible component defined over $k$, which by Hasse-Weil bounds implies that $X_3$ has a lot of rational points.
\end{rem}

\begin{proof}
First, we prove that either the curve $X_1$ is absolutely irreducible, or it splits into two absolutely irreducible components $A$ and $B$ defined over $k$.
Since $X_1$ is complete and $X_0$ is absolutely irreducible, each of the absolutely irreducible components of $X_1$ surjects to $X_0$ through~$\theta$.
The points $a$ and $b$ are the only two preimages of $s$,
and since they are analytically irreducible, each belongs to exactly one absolutely irreducible component of $X_1$.
Therefore, $X_1$ has at most two components.
Assuming $X_1$ is not absolutely irreducible,
let $A$ be the component containing $a$, and
$B$ the component containing $b$.
Since the normalisation of $\theta$ is not ramified at $a$, the cover $\theta$ restricts to a birational morphism $A \rightarrow X_0$.
Yet, $\theta$ is of degree at least $3$, so it \emph{does not} restrict to a birational morphism $B \rightarrow X_0$.
Since $\theta$ is defined over $k$, the components $A$ and $B$ are not $\Gal(\overline k / k)$-conjugate, so they are each defined over $k$.

Next, we prove Point~\ref{prop:point-x1irred}. If $X_1$ is absolutely irreducible, then $X_2$ is absolutely irreducible from Lemma~\ref{lem:updated-lemma-KW18}, and we deduce that $X_3$ is absolutely irreducible from Lemma~\ref{lem:updated-lemma-KW18} again. Note that we need here that $X_2\rightarrow X_1$ is unramified at $(b,a) \in X_2$, a consequence of the fact that $\theta$ is unramified at $a$.

We now prove Point~\ref{prop:point-x1notirred}.
Assume that $X_1$ decomposes as $A \cup B$. We first show that $X_2$ is the union of the absolutely irreducible components $A\times_{X_0}B$, $B\times_{X_0}A$, and $(B\times_{X_0}B) \setminus \Delta_B$, each defined over~$k$.
We have
\begin{align*}
(X_1 \times_{X_0} X_1) \setminus \Delta_{X_1} &= ((A \times_{X_0} A) \cup (A \times_{X_0} B) \cup (B \times_{X_0} A) \cup (B \times_{X_0} B)) \setminus \Delta_{X_1}\\
& = (A \times_{X_0} B) \cup (B \times_{X_0} A) \cup ((B \times_{X_0} B) \setminus \Delta_{B}).
\end{align*}
Both $A \times_{X_0} B$ and $B \times_{X_0} A$ are birational to $B$, so they are absolutely irreducible.
The point $(b,b) \in (B \times_{X_0} B) \setminus \Delta_{B}$ is analytically irreducible (so it can only be in one component), and is the only preimage of the point $b \in B$ through the projection to the first factor. Therefore $B \times_{X_0} B$ is absolutely irreducible.
Finally, we prove that $X_3$ contains an absolutely irreducible component defined over $k$. Consider the component of $X_3$ of the form
$$Y = (B\times_{X_0}A) \times_{B} ((B\times_{X_0}B)\setminus \Delta_B),$$
with respect to the projections to the first factor $B\times_{X_0}A \rightarrow B$ and $B\times_{X_0}B \rightarrow B$. The projection $Y \rightarrow (B\times_{X_0}B)\setminus \Delta_B$ is an isomorphism, so $Y$ is absolutely irreducible.
\end{proof}

\subsection{Traps}\label{subsec:summary-traps-3to2}
Recall that we need $D$ not to be a trap: we now define what this means. Let
\begin{alignat*}{1}
\mathscr T_3^0 &= \{[D_1] + [D_1] + [D_2] \mid D_1, D_2 \in E\} \subset \mathscr D_3,\\
\mathscr T_3^1 &= \{[D_1] + [D_2] + [D_3] \mid (D_1+D_2)^{(q)} = (D_1+D_i) + 2Q\text{ for some }i \neq 1\} \subset \mathscr D_3,\\
\mathscr T_3^2 &= \{[D_1] + [D_2] + [D_3] \mid D_1^{(q)} = D_i + Q\text{ for some }i\} \subset \mathscr D_3.
\end{alignat*}
In addition, a fourth kind of traps $\mathscr T_3^3$ is defined in Proposition~\ref{prop:trapsT33}.
The set of traps is $\mathscr T_3 = \bigcup_{i=0}^3 \mathscr T_3^i$, and we suppose that $D \not\in \mathscr T_3$.
The following lemma allows us to prove that traps can always be avoided in the descent algorithm (in particular, not every divisor is a trap).

\begin{lem}\label{lem:notalltraps02to3} Let $P_0 \in E$. If $P_0^{(q)} \not \in \{P_0-Q, P_0+2Q\}$, then $\mathscr P_2(P_0) + [-P_0] \not\subset \mathscr T_3$. 
\end{lem}
\begin{proof}
This easily follows from the above definitions, and Proposition~\ref{prop:trapsT33}.
\end{proof}

\subsection{Exceptional points of $X_0$}\label{subsec:x0interS}

Let $S = \overline{\PGL_2 \star x^q} \subset \P(V)$ be the subvariety of exceptional points.
In this section we give an explicit list of the $24$ points in $X_0 \cap (S \times E)$ (or only $12$ points in characteristic $2$).
Let $f = (x-\beta)^q(x-\alpha) \in S$, and suppose $(f,P) \in X_0$. Then, $D$ divides the positive part of
$$\div(\varphi_P(f)) = [P_\alpha - P] + [-P_\alpha - P] + [P_{\beta^q} - Q - P^{(q)}] + [-P_{\beta^q} - Q - P^{(q)}] - 2[-P] - 2[-Q-P^{(q)}],$$
where $P_\gamma$ is any of the two points such that $x(P_\gamma) = \gamma$.
There are three ways to split $D = D' + [D_3]$ where $D'$ is of degree $2$ and $D_3$ is a point. Each such splitting induces two possible ways for $D$ to divide $\div(\varphi_P(f))$: either $D'$ divides $\div(\varphi_P(x-\alpha))$ and $[D_3]$ divides $\div(\varphi_P((x-\beta)^q))$, or the reverse.
Let $c$ be the number of two-torsion points on $E$; it is $2$ in characteristic $2$ (recall that $E$ is ordinary) and $4$ otherwise.
Each of these $6$ configurations gives rise to $c$ possible values of $f$, and we find that $X_0 \cap S \times E$ contains a total of $6c$ points (we observe below that they project to $6c$ distinct points in $E$). Indeed, if $D' = [D_1] + [D_2]$, $D'$ divides $\div(\varphi_P(x-\alpha))$ and $[D_3]$ divides $\div(\varphi_P((x-\beta)^q))$, we get that $P$ is any of the $c$ points such that $2P = -(D_1 + D_2)$. Then, $\alpha = x(P+D_1)$ and $\beta^q = x(P^{(q)} + Q + D_3)$.
Similarly, if $[D_3]$ divides $\div(\varphi_P(x-\alpha))$ and $D'$ divides $\div(\varphi_P((x-\beta)^q))$, we get that $P$ is any of the $c$ points such that $2P^{(q)} = -(D_1 + D_2+2Q)$. Then, $\alpha = x(P+D_3)$ and $\beta^q = x(P^{(q)} + Q + D_1)$.

Assuming that $D_i\neq D_j$ for $i\neq j$ and $(D_i+D_j)^{(q)} \neq (D_i+D_k) + 2Q$ for any $i,j,k$ (with $i \neq j$ and $i \neq k$), then the $6c$ points are distinct. Therefore, since $D$ is not in $\mathscr T_3^0 \cup \mathscr T_3^1$, the intersection $X_0 \cap (S \times E)$ projects to $6c$ distinct points in $E$.

\subsection{Irreducibility of $X_0$}
Let $P \in E$.
From the Riemann-Roch theorem, $\dim(L(2[-P] + 2[-Q-P])) = 4$ and $\dim(L(2[-P] + 2[-Q-P] - D)) = 1$.
Since $\dim(V) = 4$, we generically expect $\varphi_P : V \rightarrow \dim(L(2[-P] + 2[-Q-P]))$ to be a bijection and thereby $\varphi_P^{-1}(L(2[-P] + 2[-Q-P] - D))$ to be of dimension $1$,
in which case there is exactly one $f \in \P(V)$ such that $(f,P) \in X_0$.
If this is indeed the case for all $P \in E$, then the projection $X_0 \rightarrow E$ is a bijection. Let us prove that it is always the case.

By contradiction, 
suppose there is a point $P\in E$ such that $L = \P(\varphi_P^{-1}(L(2[-P] + 2[-Q-P] - D)))$ has dimension at least $1$.
From Section~\ref{subsec:x0interS}, the variety $L$ intersects $S$ only at one point, so it must be a line tangent to $S$ at that point.
Write $D = \sum_{i=1}^3[D_i]$, $u_i = x(D_i + P)$ and $v_i = x(D_i + Q + P^{(q)})$.
Without loss of generality, the intersection point is $(x^q - v_3)(x - u_1)$. There are two cases to distinguish: either $u_1 = u_2$, or $v_3 = v_2$ (corresponding to the two cases exhibited in Section~\ref{subsec:x0interS}).
The points $a_{q+1}x^{q+1} + a_qx^q + a_1x + a_0 \in L$ satisfy (by construction) the three equations
$$a_{q+1}u_iv_i + a_q v_i + a_1 u_i + a_0 = 0, i \in \{1,2,3\}.$$
Also, since $L$ is tangent to $S$ at $(x^q - v_3)(x - u_1)$, they also satisfy the equation
$$a_{q+1}u_1v_3 + a_q v_3 + a_1 u_1 + a_0 = 0.$$
These linear equations can be represented by the matrix
\begin{equation*}
M = \left(
\begin{matrix}
v_1 u_1 & v_1 & u_1 & 1\\	
v_2 u_2 & v_2 & u_2 & 1\\	
v_3 u_3 & v_3 & u_3 & 1\\	
v_3 u_1 & v_3 & u_1 & 1\\		
\end{matrix}
\right).
\end{equation*}
Since $L$ has dimension at least $1$, the rank of this matrix is at most $2$. We now show that when $D$ is not a trap, $M$ is of rank $3$, a contradiction (implying that $X_0\rightarrow E$ is a bijection).
In the case where $u_1 = u_2$, we have $u_1 \neq u_3$ (because $D \not \in \mathscr T_3^0$), so
\begin{equation*}
\rank(M) = 1 + \rank\left(
\begin{matrix}
v_1 u_1 & v_1 & 1\\	
v_2 u_1 & v_2 & 1\\	
v_3 u_1 & v_3 & 1\\		
\end{matrix}
\right) = 1 + \rank\left(
\begin{matrix}
v_1 & 1\\	
v_2 & 1\\	
v_3 & 1\\		
\end{matrix}
\right).
\end{equation*}
Since $D \not \in \mathscr T_3^0$, the values $v_i$ are not all equal, so the rank of $M$ is $3$. The case $v_3 = v_2$ is similar.

We have proved that the projection $X_0 \rightarrow E$ is a bijection. Now, $X_0$ contains at least one smooth point: for instance, consider one of the points of the form $((x-\beta)^q(x-\alpha), P) \in X_0$ where $2P = -(D_1 + D_2)$.
The Jacobian matrix is
\begin{equation*}
\left(
\begin{matrix}
0				& 0				& 0				& \frac{\partial e}{\partial x_P}	& \frac{\partial e}{\partial y_P}			\\
v_1				& u_1			& 1				& \frac{\partial u_1}{\partial x_P}(v_1 + a_1) & \frac{\partial u_1}{\partial y_P}(v_1 + a_1) \\
v_2				& u_2			& 1				& \frac{\partial u_2}{\partial x_P}(v_2 + a_1) & \frac{\partial u_2}{\partial y_P}(v_2 + a_1) \\
v_3				& u_3			& 1				& \frac{\partial u_3}{\partial x_P}(v_3 + a_1) & \frac{\partial u_3}{\partial y_P}(v_3 + a_1) 		
\end{matrix}
\right),
\end{equation*}
and has rank $4$ since $(\frac{\partial e}{\partial x_P}, \frac{\partial e}{\partial y_P}) \neq(0,0)$ (because $E$ is smooth), $u_1 = u_2$ (by construction of the point), $u_2 \neq u_3$ (because $D \not \in \mathscr T_3^0$), and $v_1 \neq v_2$ (because $D \not \in \mathscr T_3^1$).
We deduce that $X_0$ is a smooth and absolutely irreducible curve.

\subsection{Local analysis of $X_1$}
Let us compute some equations for $X_1$. We see it as a subvariety of $\P^3 \times E \times \P^1$, parameterized by the (affine) variables $a_{q}, a_1, a_0, x_E, y_E, r$ (where the corresponding polynomial is $x^{q+1} + a_{q}x^q + a_1x + a_0 \in \P(V)$, the elliptic curve point is $P = (x_P,y_P) \in E$, and the root is $r \in \P^1$). As above, let $D = \sum_{i=1}^3[D_i]$, $u_i = x(P+D_i)$ and $v_i = x(P^{(q)} + D_i+Q)$. The defining polynomials of $X_1$ are the equation $e \in \F_q[x_P,y_P]$ of the elliptic curve $E$ and the four polynomials\label{eq:equationF3}
\begin{alignat*}{1}
F_1 &= v_1u_1 + a_{q}v_1 + a_1u_1 + a_0,\\
F_2 &= v_2u_2 + a_{q}v_2 + a_1u_2 + a_0,\\
F_3 &= v_3u_3 + a_{q}v_3 + a_1u_3 + a_0,\\
G &= r^{q+1} + a_{q}r^q + a_1r + a_0.
\end{alignat*}
Recall that for any $P_0 \in E$, we have $\mathscr P_2(P_0) = \{[R] + [T] \mid R+T = P_0\} \subset \mathscr D_2$.
\begin{prop}\label{prop:trapsT33}
There is a point  $s \in X_0 \cap (S \times E)$ of which both preimages through $\theta$ in $X_1$ are smooth, unless $D$ belongs to a strict closed subvariety $\mathscr T_3^3$ of $\mathscr D_3$. 
For any $P_0 \in E$, we have $\mathscr P_2(P_0) + [-P_0] \not\subset \mathscr T_3^3$. 
\end{prop}

\begin{proof}
The Jacobian matrix associated to the given defining polynomials of $X_1$ is
\begin{equation*}
\left(
\begin{matrix}
0				& 0				& 0				& \frac{\partial e}{\partial x_P}	& \frac{\partial e}{\partial y_P}			& 0\\
v_1				& u_1			& 1				& \frac{\partial u_1}{\partial x_P}(v_1 + a_1) & \frac{\partial u_1}{\partial y_P}(v_1 + a_1) & 0\\
v_2				& u_2			& 1				& \frac{\partial u_2}{\partial x_P}(v_2 + a_1) & \frac{\partial u_2}{\partial y_P}(v_2 + a_1) & 0\\
v_3				& u_3			& 1				& \frac{\partial u_3}{\partial x_P}(v_3 + a_1) & \frac{\partial u_3}{\partial y_P}(v_3 + a_1) & 0\\
r^q				& r				& 1				& 0				& 0				& r^q + a_1
\end{matrix}
\right)
\end{equation*}
Since $X_0$ is smooth, the top-left $5\times 4$ submatrix has rank $4$.
Therefore the above matrix has rank $5$ at any point where $r^q + a_1 \neq 0$.
Therefore, the only points that could be singular on $X_1$ correspond to polynomials of the form $(x-\beta)^q(x-\alpha) = x^{q+1} - \alpha x^q - \beta^qx + \beta^q\alpha$, together with the elliptic curve point $(x_0,y_0)$ and the root $\beta$. In terms of the coordinates $(a_q,a_1,a_0,x_E,y_E,r)$, such a point is given by $(-\alpha,-\beta^q,\beta^q\alpha,x_0,y_0,\beta)$.
It is non-singular if and only if the matrix
\begin{equation*}
\left(
\begin{matrix}
0				& 0								& \frac{\partial e}{\partial x_P}	& \frac{\partial e}{\partial y_P}			\\
v_1-r^q			& u_1-r							& \frac{\partial u_1}{\partial x_P}(v_1 + a_1) & \frac{\partial u_1}{\partial y_P}(v_1 + a_1) \\
v_2-r^q			& u_2-r							& \frac{\partial u_2}{\partial x_P}(v_2 + a_1) & \frac{\partial u_2}{\partial y_P}(v_2 + a_1) \\
v_3-r^q			& u_3-r							& \frac{\partial u_3}{\partial x_P}(v_3 + a_1) & \frac{\partial u_3}{\partial y_P}(v_3 + a_1) 			
\end{matrix}
\right)
\end{equation*}
has rank $4$ at this geometric point.

Let $\mathscr T_3^3$ be the subvariety of $\mathscr D_3$ of divisors $D$ for which this matrix is singular at all the corresponding $24$ exceptional points (or $12$ in characteristic $2$). 
Fix $P_0 \in E$, and let us show that
$\mathscr P_2(P_0) + [-P_0] \not\subset \mathscr T_3^3$.
Let $P \in E$, and
$$D = [D_1] + [D_2] + [D_3] = [-P_0] + [P_0 - 2P] + [2P] \in \mathscr P_2(P_0) + [-P_0].$$
The point $((x-\beta)^q(x-\alpha), P)$ is on the induced $X_0$ for 
$\alpha = x(P+D_1) = x(P+D_2)$ and $\beta^q = x(P^{(q)}+D_3 + Q)$.
At this exceptional point, the matrix simplifies to
\begin{equation*}
\left(
\begin{matrix}
0				& 0								& \frac{\partial e}{\partial x_P}	& \frac{\partial e}{\partial y_P}			\\
v_1-\beta^q			& u_1-\beta							& \frac{\partial u_1}{\partial x_P}(v_1 -\beta^q) &  \frac{\partial u_1}{\partial y_P}(v_1 -\beta^q) \\
v_2-\beta^q			& u_2-\beta							& \frac{\partial u_2}{\partial x_P}(v_2 -\beta^q) &  \frac{\partial u_2}{\partial y_P}(v_2 -\beta^q) \\
0			& u_3-\beta							& 0  & 0			
\end{matrix}
\right)
\end{equation*}
It is easy to see that $v_1-v_2$ and $u_3-\beta$ are non-zero rational functions of $P$. 
Then, at almost all~$P$, the matrix is singular if and only if the following matrix is singular\footnote{A word of caution: with $u_i = x(P+D_i)$, and substituting $D_1 = -P_0$ and $D_2 = P_0 - 2P$, we get $u_1 = u_2$ for every $P$, with this family of divisors~$D$. It could then seem that the second line of this $2\times 2$ matrix is zero. This is not the case, because $\frac{\partial u_i}{\partial x_P}$ is the derivative of $u_i = x(P+D_i)$ considered as a function of $P$, for \emph{constant} $D_i$.}:
\begin{equation*}
\left(
\begin{matrix}
 \frac{\partial e}{\partial x_P}	& \frac{\partial e}{\partial y_P}			\\
 \frac{\partial u_1}{\partial x_P} - \frac{\partial u_2}{\partial x_P} &  \frac{\partial u_1}{\partial y_P} - \frac{\partial u_2}{\partial y_P}
\end{matrix}
\right)
\end{equation*}
An explicit computation shows that for any $D_1$, the determinant of this matrix is a non-zero rational function of $P$.
Indeed, using the addition formula for the short Weierstrass equation $y^2 = x^3 + Ax + B$ in characteristic larger than $3$, we get that the numerator of this determinant divides a rational function in which the two leading terms are $$y(D_1)x_P^{30} - (A + 3x(D_1)^2)x_P^{28}y_P.$$
The explicit computations being cumbersome, we provide a Magma script\footnote{\url{https://github.com/Calodeon/dlp-proof/blob/master/3to2elimination.m}}.
This numerator is a non-zero rational function of $P$ if $y(D_1) \neq 0$ or $A \neq -3x(D_1)^2$. If $y(D_1) = 0$ and $A = -3x(D_1)^2$, then $B = 2x(D_1)^3$, and the discriminant of the short Weierstrass equation is zero, a contradiction.
The cases of characteristic $2$ and $3$ are similar (and the arguments are indeed simpler since the exhibited leading coefficients are $x^{27}$ and $x^{39}y$ respectively).
We deduce that for any $D_1$, all but finitely many points $P$ give rise to a non-singular point.
With $D_1 = -P_0$, we have shown that $\mathscr P_2(P_0) + [-P_0] \not\subset \mathscr T_3^3$.
\end{proof}

In the rest of this section, we fix the point $s$ from Proposition~\ref{prop:trapsT33}. With $s = ((x-\beta)^q(x-\alpha), P_0)$, let $a = ((x-\beta)^q(x-\alpha), P_0,\alpha)$ and $b = ((x-\beta)^q(x-\alpha), P_0,\beta)$, the two preimages of $s$ through~$\theta$. These points $s, a$ and $b$ are the ones that will allow us to apply Proposition~\ref{prop:structure-of-proof}. Since these points are smooth, they are analytically irreducible.

\begin{lem}\label{lem:2-3-unramif-at-a}
The morphism $X_1 \rightarrow X_0$ is unramified at $a$.
\end{lem}

\begin{proof}
In terms of the coordinates $(a_q,a_1,a_0,x_E,y_E,r)$, the point $a = ((x-\beta)^q(x-\alpha),P_0,\alpha) \in X_1$ is the tuple $(-\alpha,-\beta^q,\beta^q\alpha,x_0,y_0,\alpha)$. With a linear change of variables, send this point to the origin, and to avoid heavy notation, we still write $(a_q,a_1,a_0,x_E,y_E,r)$ for the translated variables.
Since $X_0$ is non-singular, it admits a local parameterisation $a_q,a_1,a_0,x,y \in k[[t]]$ at $s$.
Then, the curve $X_1$ is given analytically at $a$ by the equation
$$G = r^{q+1} + a_{q}r^q + (a_1 + \alpha^q - \beta^q)r + (a_q\alpha^q + a_1\alpha + a_0) \in k[[r,t]].$$
The induced morphism between the completions of the local rings is then given by
\begin{alignat*}{1}
f : k[[t]] &\longrightarrow k[[t,r]]/(G)\\
 t &\longmapsto t,
\end{alignat*}
Since $\alpha^q - \beta^q \neq 0$, the variable $t$ does not divide the linear term of $G$, and the morphism is therefore unramified at $a$.
\end{proof}

\subsection{Local analysis of $X_2$}
In this section, we show that the point $(b,b) \in X_2$ is analytically irreducible.

\begin{lem}\label{blowupanalytically}
Consider a morphism of smooth curves $\eta : Z \rightarrow Y$ over some field $k$. Suppose that at some point $z \in Z$, the induced morphism between the completions of the local rings is given by
\begin{alignat*}{1}
\eta^*_z : k[[t]] &\longrightarrow k[[t,r]]/(t-r^qB(t,r))\\
 t &\longmapsto t,
\end{alignat*}
where $B(0,0) \neq 0$ and $B(0,r)$ has a non-zero linear term. Then, the point $(z,z) \in (Z \times_Y Z) \setminus \Delta_Z$ is analytically irreducible.
\end{lem}
\begin{proof}
Up to isomorphism, $\eta^*_z$ can be written as
$$\eta^*_z : k[[t]] \rightarrow k[[r]] : t \mapsto r^qU(r),$$
where $U(0) \neq 0$ and $U(r)$ has a non-zero linear term.
Then, the completion of the local ring at $(z,z) \in (Z \times_Y Z) \setminus \Delta_Z$ is $k[[r,r']]/C(r,r')$ where
$$C(r,r') = \frac{r^qU(r) - r'^qU(r')}{r-r'}.$$
Unsurprisingly, $(z,z)$ is singular: this corresponds to the fact that $\eta$ is ramified at $z$, with ramification index $q > 2$. Let us blow up the equation $C(r,r')$ by introducing a variable $s$ and the equation $r' = rs$ (the case $r = r's$ is symmetric). Substituting in $C$, we obtain
$$C(r,rs) = \frac{r^q(U(r) - s^qU(rs))}{r(1-s)} = r^{q-1}\frac{U(r) - s^qU(rs)}{1-s}.$$
The equation of the blowup is $H(r,s) = C(r,rs)/r^{q-1}$. The only solution of $H(0,s) = 0$ is at $s = 1$, so there is only one point in the blowup, and it remains to see that it is smooth. Write $U(r) = u_0 + u_1r + r^2\tilde U(r)$, and $s' = s-1$. We have
\begin{alignat*}{1}
U(r) - s^qU(rs) &= U(r) - U(rs) - s'^qU(rs)\\
&= u_0 + u_1r + r^2\tilde U(r) - u_0 - u_1rs - r^2s^2\tilde U(rs) - s'^qU(rs)\\
&=  - u_1rs' + r^2\tilde U(r) - r^2s^2\tilde U(rs) - s'^qU(rs).
\end{alignat*}
Therefore, $u_1\neq 0$ is the coefficient of the monomial $r$ in $H(r,s)$, so $H(r,s)$ has a non-zero linear term, implying that the point at $r = 0$ and $s = 1$ is smooth.
\end{proof}

\begin{prop}\label{prop:analyticirredX2-3to2}
The point $(b,b) \in X_2$ is analytically irreducible. 
\end{prop}

\begin{proof}
Recall that $s\in X_0$ is the point from Proposition~\ref{prop:trapsT33}, and $a,b\in X_1$ are its two preimages through $\theta$.
We start as in the proof of Lemma~\ref{lem:2-3-unramif-at-a}.
In terms of the coordinates $(a_q,a_1,a_0,x_E,y_E,r)$, the point $b = ((x-\beta)^q(x-\alpha),P_0,\beta) \in X_1$ is the tuple $(-\alpha,-\beta^q,\beta^q\alpha,x_0,y_0,\beta)$, and we send this point to the origin via a linear change of variables (again, to avoid heavy notation, we still write $(a_q,a_1,a_0,x_E,y_E,r)$ for the translated variables).
Let $a_q,a_1,a_0,x,y \in k[[t]]$ be the local parameterisation of $X_0$ at $s$.
The curve $X_1$ is given analytically at $b$ by the equation
$$G = r^{q+1} + (a_{q} + \beta - \alpha)r^q + a_1r + (a_q\beta^q + a_1\beta + a_0) \in k[[r,t]].$$
Since $b$ is non-singular, $a_q\beta^q + a_1\beta + a_0 = tF(t)$ with $F(0) \neq 0$. Writing $a_1 = tH(t)$, we get
$$G = t(F(t) + H(t)r) + r^{q}(r + a_{q} + \beta - \alpha)  \in k[[r,t]].$$
Since $D \not\in \mathscr T_3^2$, we have $\alpha \neq \beta$, so up to multiplication by a unit, $G$ is of the form
$t - r^{q}B(t,r)$ for some $B(t,r)$ such that $B(0,0) \neq 0$.
We need to show that $B(0,r)$ has a non-zero linear term.
Write $v_3 = \tilde v_3 + \beta^q$. From equation $F_3$ (defined on page~\pageref{eq:equationF3}), we have
\begin{alignat*}{1}
0 &= v_3u_3 + (a_{q} -\alpha)v_3 + (a_1 - \beta^q)u_3 + a_0 + \beta^q\alpha\\
& = (\tilde v_3 + \beta^q)u_3 + (a_{q} -\alpha)(\tilde v_3+\beta^q) + (a_1 - \beta^q)u_3 + a_0 + \beta^q\alpha\\
& = a_1u_3 - \tilde v_3(\alpha-u_3- a_{q}) +a_{q} \beta^q  + a_0\\
& = a_1(u_3-\beta) - \tilde v_3(\alpha-u_3- a_{q}) +tF(t)
\end{alignat*}
Since $v_3$ is a $q$-th power, we get that $H(0) = \frac{a_1}{t}(0) = -\frac{F(0)}{u_3(0) - \beta}$.
We deduce that the linear term of $B(0,r)$ is $-F(0)^{-1}(1 - \frac{\alpha-\beta}{u_3(0) - \beta})$.
Since $D \not\in \mathscr T_3^0$, $u_3(0) \neq \alpha$, so $B(0,r)$ has a non-zero linear term.
We conclude with Lemma~\ref{blowupanalytically}.
\end{proof}

\subsection{Irreducibility of $X_3$}
We are finally ready to prove the main result of this section.

\begin{prop}\label{prop:3to2-x3irred}
For any divisor $D \in (\mathscr{D}_3 \setminus \mathscr{T}_3)(k)$, the curve $X_3$ contains an absolutely irreducible component defined over~$k$.
\end{prop}

\begin{proof}
We have shown that $\theta : X_1 \rightarrow X_0$ satisfies all the conditions of Proposition~\ref{prop:structure-of-proof}, so the result follows.
\end{proof}

\section{Degree $4$--to--$3$ elimination}\label{sec:4-to-3-elim}

\noindent
As for the degree $3$--to--$2$ elimination, we are going to apply Proposition~\ref{prop:structure-of-proof}.
Consider an extension $k/\F_{q}$ and a divisor $D \in \mathscr {D}_4(k)$.
Recall from Section~\ref{subsec:overview-4-3} that we work with the vector space $V = \mathrm{span}(x_i^qx_j \mid i,j \in \{ 0,1,2 \})$, the morphism $\psi : V \rightarrow \algclosure[E]$ which substitutes $x_0,x_1$ and $x_2$ with $1$, $x$, and $y$ respectively, and the morphism $\varphi: V \rightarrow L(3[-Q]+3[0_E])$ with $\varphi(x_i^qx_j) = (\psi(x_i)\circ \tau_Q) \cdot \psi(x_j)$.
Let $K$ be the preimage through $\varphi$ of $L(3[-Q]+3[0_E]-D)$.
We then have
$$X_0 = \overline{\{f \in \PGL_3 \star \mathfrak d \mid \varphi(f) \equiv 0 \mod D\}} = \overline{\PGL_3 \star \mathfrak d} \cap \P(K),$$
where $\mathfrak d = x_0^qx_1 - x_1^qx_0 \in V$.
The space $\P(K)$ is a hyperplane in $\P(V)$, which we denote by $H$.
We prove in Lemma~\ref{lem:equationX0} that $X_0$ is a curve.
Let us represent the elements of $V$ (or $\P(V)$) as column vectors
$$\sum_{ij}a_{ij}x_i^qx_j = \begin{pmatrix}
a_{00}	&a_{01}	&a_{02}	&a_{10}	&a_{11}	&a_{12}	&a_{20}	&a_{21}	&a_{22}
\end{pmatrix}^\transpose = a_\mathrm{vec}$$
The hyperplane $H$ is the kernel of the matrix
$$\begin{pmatrix}
H_{00}	&H_{01}	&H_{02}	&H_{10}	&H_{11}	&H_{12}	&H_{20}	&H_{21}	&H_{22}
\end{pmatrix},$$
where each $H_{ij}$ is a column vector of dimension $4$. 
With $\Lambda = \mathrm{span}(x_i \mid i \in \{ 0,1,2 \})$,
the curve $X_1$ is defined as 
$$X_1 = \{(f,u)  \mid f \in X_0 \text{ and }u\text{ is a factor of }f\} \subset \P(V)\times \P(\Lambda),$$
and we set $X_2 = X_1 \times_{X_0} X_1 \setminus \Delta_{X_1}$ and $X_3 = X_2 \times_{X_1} X_2 \setminus \Delta_{X_2}$ as in Section~\ref{subsec:roadmap}.

\subsection{Exceptional points of $X_0$}\label{subset:excep-points-x0}
\noindent
Let $D = \sum_{i = 1}^{4}[D_i] \in \mathscr {D}_4(k)$ be the divisor of $E$ to be eliminated, and
let $H$ be the induced hyperplane.
Suppose that $D$ is not divisible by a principal divisor of degree $3$ (being divisible by a principal divisor would correspond to the traps of type $\mathscr T_4^4$ defined in Section~\ref{subsec:summary-traps-4to3}).
Let $uv^q \in S \cap H$ be an exceptional point of $X_0$. We have $\varphi(uv^q) = u \cdot (v^{(q)}\circ \tau_Q)$ (where $v^{(q)}$ is $v$ with coefficients raised to the power $q$), which is of degree $6$ since $u$ and $v^{(q)}\circ \tau_Q$ are each of degree $3$. Since $D$ divides $\varphi(uv^q)$ and $D$ is not divisible by a principal divisor of degree $3$, we have a permutation $\sigma \in \mathfrak S_{4}$ such that
$$\mathrm{div}(u) = [D_{\sigma(1)}] + [D_{\sigma(2)}] + \left[-D_{\sigma(1)}-D_{\sigma(2)}\right] - 3[0_E],$$
and 
$$\mathrm{div}(v^{(q)}\circ \tau_Q) = [D_{\sigma(3)}] + [D_{\sigma(4)}] + \left[-D_{\sigma(3)}-D_{\sigma(4)} - 3Q\right] - 3[-Q].$$
The second equality implies 
$$\mathrm{div}(v^{(q)}) = [D_{\sigma(3)} + Q] + [D_{\sigma(4)} + Q] + \left[-D_{\sigma(3)}-D_{\sigma(4)} - 2Q\right] - 3[0_E].$$
Note that, as should be expected, the number of ways to split $D$ into two parts of two points is exactly the degree of $S$ (recall that $S$ is the image of the Segre embedding $\P(\Lambda)\times\P(\Lambda) \rightarrow \P(V)$). 
The above constitutes an exhaustive description of the set of exceptional points $S \cap H.$

There are $6$ exceptional points $U^i{V^i}^q$ in the intersection of $X_0$ with $S$. 
Up to reindexing, we necessarily have the following triples of aligned points:
$$(V^1,V^2,V^3), (V^1,V^4,V^5), (V^2,V^5,V^6), (V^3,V^4,V^6),$$
$$(U^4,U^5,U^6), (U^2,U^3,U^6), (U^1,U^3,U^4), (U^1,U^2,U^5).$$
They arise as follows. Consider the $6$ pairs of distinct points dividing $D = \sum_{i=1}^4 [D_i]$,
$$\{d_1,\dots,d_6\} = \{(D_3,D_4),(D_2,D_3),(D_2,D_4),(D_1,D_4),(D_1,D_3),(D_1,D_2)\}.$$
For each $i$, if $d_i = (D_j,D_k)$, then $U^i$ defines the line passing through $D_j$ and $D_k$, while ${V^i}^{(q)}$ defines the line passing through $D_m + Q$ and $D_n + Q$, where $\{j,k,m,n\} = \{1,2,3,4\}$. With this indexing, we can see that $U^4,U^5,$ and $U^6$ are aligned because they all have a root at $D_1$ (i.e., the corresponding lines intersect at the point $D_1$).
All the alignments listed above arise in this way.

As in the $3$--to--$2$ case, we follow the strategy outlined in Section~\ref{subsec:roadmap}, so we fix a point $s \in X_0$, say $s = (V^1)^qU^1$, and its two preimages $a = ((V^1)^qU^1, U_1)$ and $b = ((V^1)^qU^1, V^1)$ in $X_1$.

\subsection{Summary of the traps}\label{subsec:summary-traps-4to3}
As long as $D$ is not a trap, there should be no other alignment between the points $U^i$ and $V^i$ than the ones listed above. Hence we define the following varieties of traps, where $\ell(R,S)$ denotes the line passing through $R$ and $S$~:
\begin{alignat*}{3}
\mathscr T_4^0 &= \Bigg\{\sum_{k = 1}^4 [D_k] \ &\Bigg|&\ 
\begin{array}{l}
\ell(D_i,D_j) \cap \ell(D_m,D_n) \cap  \ell(D_r, D_s) \neq \emptyset, \\
\{i,j\},\{m,n\},\{r,s\}\text{ all distinct, and }\{i,j\}\cap\{m,n\}\cap\{r,s\} = \emptyset 
\end{array} &\Bigg\}&\\
\mathscr T_4^1 &= \Bigg\{\sum_{k = 1}^4 [D_k] \ &\Bigg|&\ 
\begin{array}{l}
\ell(D_i+Q,D_j+Q) \cap \ell(D_m+Q,D_n+Q) \cap  \ell(D_r+Q, D_s+Q) \neq \emptyset,\\
\{i,j\},\{m,n\},\{r,s\}\text{ all distinct, and }\{i,j\}\cap\{m,n\}\cap\{r,s\} = \emptyset 
\end{array} &\Bigg\}&\\
\mathscr T_4^2 &= \Bigg\{\sum_{k = 1}^4 [D_k] \ &\Bigg|&\ 
\begin{array}{l}
\ell(D_i,D_j)^{(q)} \cap \ell(D_m,D_n)^{(q)} \cap  \ell(D_r + Q, D_s + Q) \neq \emptyset, \\
 \text{and }\{i,j\}\neq\{m,n\}
 \end{array} &\Bigg\}&\\
\mathscr T_4^3 &= \Bigg\{\sum_{k = 1}^4 [D_k] \ &\Bigg|&\ 
\begin{array}{l} 
\ell(D_i+Q,D_j+Q) \cap \ell(D_m+Q,D_n+Q) \cap  \ell(D_r, D_s)^{(q)} \neq \emptyset, \\
\text{and }\{i,j\}\neq\{m,n\} 
\end{array} &\Bigg\}&\\
\mathscr T_4^4 &= \{F + [D_4] &\mid& \ \ F \in \mathscr P_3, D_4 \in E\}.&
\end{alignat*}
The conditional statements are to be understood as ``\emph{there exist indices $i,j,m,n,r,s$ such that $i\neq j$,  $m\neq n$,  $r\neq s$, and... }''. Let $\mathscr T_4' = \bigcup_{i=0}^4 \mathscr T_4^i$.
Note that the full variety of traps $\mathscr T_4$ (rather than $\mathscr T_4'$) requires an additional component, studied in Section~\ref{subsubsec:annoyingbadcasetrap}.

\begin{lem}\label{lem:notalltraps00} For any points $P_0,P_1 \in E$ such that either $P_0\neq P_1$ or $P_0^{(q)} \neq P_0 + 2Q$, we have ${\mathscr P_2(P_0) + \mathscr P_2(P_1) \not\subset \mathscr T_4'}$.
\end{lem}

\begin{proof}
Let $R,T\in E$, and $D = \sum_{i=1}^4[D_i] \in \mathscr P_2(P_0) + \mathscr P_2(P_1)$, where
$$(D_1,D_2,D_3,D_4) = (R, P_0 - R, T, P_1 - T).$$
We simply need to show that $D \not \in \mathscr T_4$ except for certain pairs $(R,T)$ that belong to some strict subvariety of $E^2$. There are many conditions to check in order to verify whether or not $D \in \mathscr T_4$; we use symmetries (exchanging $R$ and $T$, replacing $R$ with $P_0 - R$, or some permutations of the three sets $\{i,j\},\{m,n\},\{r,s\}$) to significantly reduce this number.

First, let us characterise the cases where $D \not \in \mathscr T_4^0$. Up to symmetry, we can assume that $1$ belongs to two of the pairs of indices, and even that $(i,j) = (1,3)$, $m = 1$ and $n\in\{2,4\}$.
As long as $D_2 = P_0 - R$ and $D_4 = P_1 - T$ are not on the line $\ell(R,T)$ (which corresponds to a strict subvariety of $E^2$), we have $\ell(D_i,D_j) \cap \ell(D_m,D_n) = \{R\}$.
The condition for $D \in \mathscr T_4^0$ is then
\begin{alignat*}{1}
R \in \ell(D_r, D_s),
\end{alignat*}
and for each allowable $(r,s)$, it corresponds to $(R,T)$ belonging to a strict subvariety of $E^2$. This implies that $\mathscr P_2(P_0) + \mathscr P_2(P_1) \not\subset \mathscr T_4^0$. The fact that $\mathscr P_2(P_0) + \mathscr P_2(P_1) \not\subset \mathscr T_4^1$ follows from the observation that $\mathscr T_4^1$ is a translation by $-Q$ of $\mathscr T_4^0$, and that $\mathscr P_2(P_0+2Q) + \mathscr P_2(P_0+2Q) \not\subset \mathscr T_4^0$.

The condition $D \not \in \mathscr T_4^2$ enjoys fewer symmetries and is therefore more cumbersome. First assume that $\{i,j\}\cap\{m,n\} \neq \emptyset$. Then, up to symmetry, we can assume $i = m = 1$, and apart from a strict subvariety of $E^2$, we have $\ell(D_i,D_j)^{(q)} \cap \ell(D_m,D_n)^{(q)} = \{R^{(q)}\}$. The conditions for $D \in \mathscr T_4^2$ become
$$R^{(q)} \in  \ell(D_r + Q, D_s + Q),$$
for any allowable $r\neq s$. None of them is satisfied as long as
$$R^{(q)} \not\in \{D_r + Q \mid r = 1,2,3,4\} \cup \{-(D_r + Q) - (D_s + Q) \mid r \neq s\},$$
which for any fixed $T$ corresponds to finitely many values of $R$ to be avoided.

It remains to consider the cases where $\{i,j\}\cap\{m,n\} = \emptyset$.
To continue the proof, let us work in $E^4$ instead of $\mathscr D_4$. More precisely, write
$$T_4^2(i,j,m,n,r,s) = \left\{(D_k)_{k = 1}^4 \ |\ 
\ell(D_i,D_j)^{(q)} \cap \ell(D_m,D_n)^{(q)} \cap  \ell(D_r + Q, D_s + Q) \neq \emptyset\right\},$$
such that $\mathscr T_4^2$ is the union of the varieties $\pi(T_4^2(i,j,m,n,r,s))$ for all allowable indices, where $\pi : E^4 \rightarrow \mathscr D_4$ is the natural projetion.
It is then sufficient to show that for each allowable $(i,j,m,n,r,s)$, we have $\pi^{-1}(\mathscr P_2(P_0) + \mathscr P_2(P_1)) \not \subset \pi^{-1}(\pi( T_4^2(i,j,m,n,r,s)))$.
This is equivalent to showing that
$\pi^{-1}(\mathscr P_2(P_0)) \times \pi^{-1}(\mathscr P_2(P_1)) \not \subset T_4^2(i,j,m,n,r,s)$ for any allowable indices; this follows from the facts that $\pi^{-1}(\mathscr P_2(P_0)) \times \pi^{-1}(\mathscr P_2(P_1))$ is absolutely irreducible, and that for any permutation $\sigma\in \mathfrak S_4$, we have $$(D_k)_{k = 1}^4 \in T_4^2(i,j,m,n,r,s) \Longleftrightarrow (D_{\sigma(k)})_{k = 1}^4 \in T_4^2(\sigma(i),\sigma(j),\sigma(m),\sigma(n),\sigma(r),\sigma(s)).$$
Up to symmetry, it is sufficient to consider $(i,j,m,n) = (1,2,3,4)$ or $(i,j,m,n) = (1,3,2,4)$.

First, suppose that $(i,j,m,n) = (1,2,3,4)$. Again up to symmetries, it is sufficient to consider $(r,s) = (2,3)$ or $(3,4)$.
Suppose $(r,s) = (3,4)$, and let 
$$(D_1,D_2,D_3,D_4) = (R, P_0 -R, T, P_1 - T) \in \pi^{-1}(\mathscr P_2(P_0)) \times \pi^{-1}(\mathscr P_2(P_1)).$$
First, if $P_0 = P_1$, then $\ell(D_1,D_2)^{(q)} \cap \ell(D_3,D_4)^{(q)} = -P_0^{(q)}$, and $\ell(D_3+Q,D_4+Q) \cap E = \{T+Q,P_0 - T+Q, -P_0 - 2Q\}$ does not contain $-P_0^{(q)}$ for almost all points $R$ and $T$ (as long as $P_0^{(q)} \neq P_0 + 2Q$). 
If $P_0 \neq P_1$, let $R = T$. Then, $\ell(D_1,D_2)^{(q)} \cap \ell(D_3,D_4)^{(q)} = R^{(q)}$. The condition becomes $R^{(q)} \in \ell(D_3+Q,D_4+Q)$. But $R^{(q)}$ is on $E$, and
$$\ell(D_3+Q,D_4+Q) \cap E = \{R+Q,P_1 - R+Q, -P_1 - 2Q\}.$$
For all but finitely many points $R$, we have that $R^{(q)}$ does not belong to this intersection.

The case $(r,s) = (2,3)$ is similar.
The same reasoning allows to conclude for the remaining cases $(i,j,m,n,r,s) \in \{(1,3,2,4,1,2), (1,3,2,4,1,3)\}$, at least for $P_0 \neq P_1$; for $P_0 = P_1$, one should choose $T$ to be one of the points such that $2T = P_0$ and $T^{(q)} \neq -P_0 - 2Q$, then observe that for almost all $R$, the condition for $\mathscr T_4^2$ is not satisfied. The proof for $\mathscr T^3_4$ is similar, and the proof for $\mathscr T^4_4$ is easy.
\end{proof}

\begin{lem}\label{lem:blowup-u-matrix}
Suppose $D$ is not a trap. Pick a matrix in $\PGL_3$ sending $U^1$ to $x_0$, $V^1$ to $x_1$, and $V^2$ to $x_2$, and let it act on $\P(V)$. In the matrix defining the transformation of $H$, the submatrices
$$\begin{pmatrix}H_{00}&H_{01}&H_{02}&H_{20}\end{pmatrix}, \begin{pmatrix}H_{00}&H_{01}&H_{11}&H_{12}\end{pmatrix}, \text{ and }\begin{pmatrix}H_{00}&H_{02}&H_{11}&H_{12}\end{pmatrix}.$$
each have full rank $4$.
\end{lem}

\begin{proof}
Using a matrix of $\PGL_3$ as described, we can suppose that $U^1 = x_0$, $V^1 = x_1$, and $V^2 = x_2$. Since $(V^1)^qU^1 = x_1^qx_0 \in H$, we have that $H_{10}$ is the zero vector.
Suppose by contradiction that $\mathrm{rank}\begin{pmatrix}H_{00}&H_{01}&H_{02}&H_{20}\end{pmatrix} \leq 3$.
Then, a non-trivial  linear-combination of the rows has the form
\begin{equation*}
\left(
\begin{matrix}
0		& 0		& 0			& 0			& D_{11}			& D_{12}			& 0			& D_{21}		& D_{22}	\\
\end{matrix}
\right)
\end{equation*}
Applied to $U^2$ and $V^2 = x_2$, we get that $U^2$ is on the line $(0:D_{21}:D_{22})$. But $U^1 = x_0$ also lies on this line, therefore so does $U^{5}$. The relation applied to $(U^5,V^5)$ implies that $U^5$ lies on the line $(0:D_{11}:D_{12})$ (unless $V^5$ is on the line $(0:1:0)$, which already contains $U^1 = x_0$ and $V^2 = x_2$, a contradiction). But $U^1$ also does, so $(0:D_{21}:D_{22}) = (0:D_{11}:D_{12})$.
The relation becomes
$$(\alpha v_1^q + \beta v_2^q)(D_{11}u_1 + D_{12}u_2),$$
where $(\alpha,\beta) \neq (0,0)$ are coefficient such that $\alpha(D_{21},D_{22}) = \beta (D_{11},D_{12})$.
We conclude that $(0:\alpha:\beta)$ is the line passing through $(V^3)^{(q)}$,$(V^4)^{(q)}$ and $(V^6)^{(q)}$, and it also contains $x_0 = (U^1)^{(q)}$, a contradiction.

Now, suppose by contradiction that $\mathrm{rank}\begin{pmatrix}H_{00}&H_{01}&H_{11}&H_{12}\end{pmatrix} \leq 3$.
We get a non-trivial linear combination of the rows of the form
\begin{equation*}
\left(
\begin{matrix}
0		& 0		& C_{02}	& 0			& 0				& 0			& C_{20}	& C_{21}	& C_{22}	\\
\end{matrix}
\right)
\end{equation*}
First, we cannot have $C_{02} = 0$, otherwise the above line gives the relation $v_{2}^q(C_{20}u_0 + C_{21}u_1 + C_{22}u_2)$: since no more than three of the $U^i$-points can be on the line $(C_{20}	: C_{21} : C_{22})$, at least three of the $V^i$-points must be on the line $(0:0:1)$, which also contains $U^1$, a contradiction.
We can therefore assume that $C_{02} = 1$.
We deduce that the line through $U^2$ and $U^3$ is $(C_{20}	: C_{21} : C_{22})$ (for $U^3$, we use the fact that $V^3_0 = 0$ because $V^3$ is aligned with $V^1$ and  $V^2$). This line contains $U^6$. We get that
$$0 = (V_0^6)^qU_2^6 + (V_2^6)^q(C_{20}U_0^6	 + C_{21}U_1^6 + C_{22}U_2^6) = (V_0^6)^qU_2^6.$$
We get that either $V_0^6 = 0$, implying that $V^6$ is aligned with $V^1,V^2,V^3$ (which are in the line $(1:0:0)$), or $U_2^6 = 0$, implying that $U^6$ is aligned with $U^1$ and $V^1$. Both cases are traps.
The same proof leads to $\mathrm{rank}\begin{pmatrix}H_{00}&H_{02}&H_{11}&H_{12}\end{pmatrix} = 4$. 
\end{proof}

\subsection{Irreducibility of $X_0$}\label{subsec:irreducibilityofX0}
In this section, we prove that $X_0$ has an absolutely irreducible component defined over $k$. To do so, we first find an equation for $X_0$ in the plane.
Recall that $\P(\Lambda)$ has coordinates $u_0,u_1$ and $u_2$, and each point $(u_0:u_1:u_2)$ represents the linear polynomial $u_0x_0 + u_1x_1 + u_2x_2$. 
Its dual space $\P(\Lambda)^\vee$ has coordinates $t_0,t_1$ and $t_2$, and any element $(t_0:t_1:t_2) \in \P(\Lambda)^\vee$ represents the line in $\P(\Lambda)$ with equation $u_0t_0 + u_1t_1 + u_2t_2 = 0$.
Define a subvariety $\mathscr O$ of $\P(V) \times \P(\Lambda)^\vee$ by the six polynomials $e_k = \sum_{j=0}^2a_{kj}t_j$ and $f_k = \sum_{i=0}^2a_{ik}t_i^q$ for $k = 0,1,2$, where $a_{ij}$ are the coordinates of $\P(V)$.

\begin{lem}
The variety $\mathscr O$ is the closure of the orbit $\PGL_3\star (\mathfrak d, (0:0:1))$. Furthermore, for any point $(f,\ell) \in \mathscr O$, any linear factor of $f$ is on the line $\ell \subseteq \P(\Lambda)$.
\end{lem}

\begin{proof}
Notice that $\PGL_3$ acts on both $\mathscr O$ and $\P(\Lambda)^\vee$, and the projection $\mathscr O \rightarrow \P(\Lambda)^\vee$ is $\PGL_3$-equivariant.
The group $\PGL_3$ acts faithfully on $\P(\Lambda)^\vee$, and the fibre of $(0:0:1)$ through $\mathscr O \rightarrow \P(\Lambda)^\vee$ is the subvariety $\P(V_2) \times\{(0:0:1)\} \subset \P(V) \times \P(\Lambda)^\vee$ where $V_2 = \mathrm{span}(x_i^qx_i \mid i,j \in \{ 0,1\})$.
Any $f \in \P(V_2)$ is a polynomial in $x_0$ and $x_1$, so its linear factors necessarily lie on the line in $\P(\Lambda)$ defined by $\ell = (0:0:1)$ (i.e., by the equation $u_2 = 0$), proving the second part of the lemma.
The group $\PGL_2$ acts on this fibre through the embedding into $\PGL_3$ as the $2\times 2$ upper-left minor. We conclude from~\cite[Lemma~2.2]{KW18}, which implies that the fibre $\P(V_2) \times\{(0:0:1)\}$ is the closure of the action of $\PGL_2$, proving the first part of the lemma.
\end{proof}

Let $h_1,\dots,h_4$ be the linearly independent linear polynomials in the $a_{ij}$-coordinates which define the hyperplane $H \subset \P(V)$ of codimension $4$, and let $C = \mathscr O\cap(H \times \P(\Lambda)^\vee)$, so that $X_0$ is a subvariety of $C$. Let $C'\subset \P(\Lambda)^\vee$ be the projection of $C$ to the second factor of $\P(V)\times \P(\Lambda)^\vee$. In the remainder of this section, we prove that all the absolutely irreducible components of $C'$ (and therefore also of $X_0$) are defined over $k$.

\begin{lem}
The projection $C \rightarrow C'$ is an injective map on the geometric points.
\end{lem}

\begin{proof}
We proceed by contradiction. Suppose there exist two distinct points $(f,\ell),(g,\ell) \in C$, where $f,g \in H \subset \P(V)$. Up to the action of $\PGL_3$, we can suppose $\ell = (0:0:1)$. Any linear combination of $f$ and $g$ is still in $H$, and still in $\P(V_2) \subset \mathscr O$, so still in $C$. Then, $C$ contains $\P(V_2) \cap H$. Since $S \cap \P(V_2)$ is a surface of degree $2$, either the intersection $\P(V_2) \cap H$ contains two distinct exceptional points $v^qu$ where $u,v \in \P(\Lambda_2)$, or it is tangent to $S$. Both cases are traps.
\end{proof}

\begin{lem}\label{lem:equationX0}
The plane curve $C'$ is defined by the polynomial (in projective coordinates $t_0,t_1,t_2$)
$$t_2^{-1}\cdot\det\begin{pmatrix}
t_0^q 	&   	&   & t_1^q &   &   & t_2^q &   &   \\
		& 	t_0^q &   &   & t_1^q &   &   & t_2^q &   \\
t_0   & t_1   & t_2 	&   	& &   & &   &    \\
&   &   & t_0   & t_1   & t_2 	&   	& & \\  
 	& &   & &   &   & t_0   & t_1   & t_2 	\\
H_{00}	&H_{01}	&H_{02}	&H_{10}	&H_{11}	&H_{12}	&H_{20}	&H_{21}	&H_{22}
\end{pmatrix},$$
where $H_{ij}$ denotes the $4$-dimensional vector whose entries are the coefficients of $a_{ij}$ in $h_1,\dots,h_4$.
The curve $C'$ has degree $2q + 2$.
\end{lem}

\begin{proof}
The six polynomials $e_k$, $f_k$, $k=0,1,2$, as well as $h_1,\ldots,h_4$
are linear polynomials in the $a_{ij}$-coordinates, thus they can be written as
$Ma_{vec}$ where $M$ is the $10 \times 9$-matrix of coefficients and
$$a_{vec} = \begin{pmatrix}a_{00}	&a_{01}	&a_{02}	&a_{10}	&a_{11}	&a_{12}	&a_{20}	&a_{21}	&a_{22}
\end{pmatrix}^\transpose.$$
Pick a row of $M$ corresponding to one of the six polynomials
$e_k$, $f_k$, $k=0,1,2$, add a column to $M$ containing zeros except
at the chosen row where the entry is $t_k^{-q}$ if $e_k$ was chosen
and $t_k^{-1}$ if $f_k$ was chosen, denote the resulting
$10 \times 10$-matrix by $M'$ and let $f=\det(M')$.
If two rows (of the six above) are chosen and the corresponding entries
are set in the adjoined column (with a possible sign change of one of them),
it follows from the relation
$\sum_{k=0}^2 t_k^qe_k - \sum_{k=0}^2 t_kf_k=0$ that the determinant of the matrix
is zero.
Then Laplace expansion with respect to the adjoined column shows that
the definition of $f$ is independent (up to sign) of the choice of
one of the six rows.
By choosing $f_2$, deleting the adjoined column as well as the row
corresponding to $f_2$ and setting $t_2=0$ it follows from
$\sum_{k=0}^1 t_k^qe_k - \sum_{k=0}^1 t_kf_k=0$ that
the determinant of the resulting matrix is zero which implies that
$f$ is a polynomial.
Therefore $f$ defines the variety $C'$.

It remains to show that $f$ has degree $2q + 2$.
Choose an arbitrary point $P \in C \cap (S \times \P(\Lambda)^\vee)$.
There is an element $g \in \PGL_3$ which maps $P$ to
$(x_1^qx_0, (0:0:1))$. Since this element is a linear transformation of $\P(\Lambda)^\vee$, it does not change the degree of the curve, and we can simply assume that $(x_1^qx_0, (0:0:1)) \in C$. This implies that in each equation $h_i$, the coefficient of $a_{10}$ is zero. Now, an simple computation shows that the coefficient of the monomial $t_2^{2q}t_0t_2$ is $\det\begin{pmatrix}H_{00}&H_{01}&H_{11}&H_{12}\end{pmatrix}$. 
From Lemma~\ref{lem:blowup-u-matrix}, it is not zero, so the degree of the equation is $2q+2$.
\end{proof}

\begin{lem}\label{lem:jacobianX0}
Let $N$ be the matrix from Lemma~\ref{lem:equationX0}, such that $f = t_2^{-1}\det(N)$ is an equation defining $C'$. We have
$$\frac{\partial f}{\partial t_0} = t_2^{-1}(m_{31}+m_{44}+m_{57}), \text{ and }\frac{\partial f}{\partial t_1} = t_2^{-1}(m_{32}+m_{45}+m_{58}),$$
where $m_{ij}$ is the $(i,j)$-minor of $N$.
\end{lem}

\begin{proof}
This is an elementary application of Jacobi's formula $d\det(N) = \tr(\mathrm{adj}(N)dN)$, where $\mathrm{adj}(N)$ is the adjoint matrix, and $dN$ is the differential of $N$.
\end{proof}

\begin{cor}
The image in $C'$ of any point in $C \cap S$ is smooth.
\end{cor}

\begin{proof}
Let $P \in C \cap S$.
There is an element $g \in \PGL_3$ which maps $P$ to
$(x_1^qx_0, (0:0:1))$. 
This transformation will map the $4$-codimensional hyperplane
$\tilde{H} \subset \P(V)$ used in the definition of $X_0$ to a
$4$-codimensional hyperplane $H$ with defining polynomials $h_1,\ldots,h_4$, for each of which the coefficient of $a_{10}$ is zero. From Lemma~\ref{lem:jacobianX0}, and the fact that $m_{31}(0:0:1) = m_{57}(0:0:1) = 0$, we get that $$\frac{\partial f}{\partial t_0}(0:0:1) = m_{44}(0:0:1) = \pm\det\begin{pmatrix}H_{00}	&H_{01}	&H_{11}	&H_{12}
\end{pmatrix}.$$
From Lemma~\ref{lem:blowup-u-matrix}, the latter determinant is non-zero, hence the image of $P$ on $C'$ is smooth.
\end{proof}

Let us now study the singularities of $C'$ away from $S$. As above, up to a transformation by a matrix $g \in \PGL_3$, it is sufficient to study the point $(x_1^qx_0 - x_0x_1^q, (0:0:1)) \in \mathscr O \cap H$ with $H_{01} = H_{10}$.
Note that with this transformation, the $H_{ij}$-columns change, and they do not have, for instance, the properties of Lemma~\ref{lem:blowup-u-matrix}.
We now have that the minors $m_{31},m_{57},m_{45},m_{58}$ are all zero at $(0:0:1)$, and
\begin{align*}
\frac{\partial f}{\partial t_0}(0:0:1) &= m_{44}(0:0:1) = \pm\det\begin{pmatrix}H_{00}	&H_{01}	&H_{11}	&H_{12}
\end{pmatrix}, \text{ and }\\
\frac{\partial f}{\partial t_1}(0:0:1) &= m_{32}(0:0:1) = \pm\det\begin{pmatrix}H_{00}	&H_{02}	&H_{10}	&H_{11}
\end{pmatrix}.\end{align*}
The point is singular if and only if both determinants are zero, i.e.,
$$\mathrm{rank}\begin{pmatrix}H_{00}	&H_{01} &H_{02}	&H_{11}	&H_{12}
\end{pmatrix} \leq 3.$$
From now on, suppose that the point is indeed singular, so  there is a linear combination of the linear equations $h_i$ that has the form $\alpha a_{20} + \beta a_{21} + \gamma a_{22}$, corresponding to the row vector
$$\begin{pmatrix}0	&0 &0	&0	&0 & 0 & \alpha & \beta & \gamma
\end{pmatrix}.$$
One must have $(\alpha,\beta) \neq (0,0)$, otherwise all the points in $X_0\cap S$ satisfy the equation $v_2^qu_2 = 0$, meaning that among all the linear functions $U^{i}$ and $V^i$, at least six of them are on the line ${(0:0:1)} \in \P(V)^\vee$, a trap.
Let us show that the singularity has multiplicity $q$. From the curve equation given in Lemma~\ref{lem:equationX0}, we derive that the quadratic terms at our point $(x_1^qx_0 - x_0x_1^q, {(0:0:1)})$ are
\begin{align*}
&\pm\det \begin{pmatrix}H_{01}	&H_{02}&H_{11}&H_{12}\end{pmatrix}t_0^2,\\
&\pm\det \begin{pmatrix}H_{00}	&H_{02}&H_{11}&H_{12}\end{pmatrix}t_0t_1,\\
&\pm\det \begin{pmatrix}H_{00}	&H_{02}&H_{10}&H_{12}\end{pmatrix}t_1^2,
\end{align*}
which are all zero since $\mathrm{rank}\begin{pmatrix}H_{00}	&H_{01} &H_{02}	&H_{11}	&H_{12}
\end{pmatrix} \leq 3.$ Now, the terms of degree $q$ are
\begin{align*}
&\pm\det \begin{pmatrix}H_{00}	&H_{10}&H_{11}&H_{21}\end{pmatrix}t_0^q,\\
&\pm\det \begin{pmatrix}H_{00}	&H_{01}&H_{11}&H_{20}\end{pmatrix}t_1^q,
\end{align*}
which are not both zero, as that would imply $(\alpha,\beta) = (0,0)$. So the multiplicity is $q$.\\

We now show that the blowup of this singularity is either a single smooth point or a node.
Without loss of generality, assume $h_1 = \alpha a_{20} + \beta a_{21} + \gamma a_{22}$, and denote by $\tilde H_{ij}$ the $3$-dimensional vector whose entries are the coefficients of $a_{ij}$ in $h_2,h_3$ and $h_4$. Then, restricting to the affine plane $\A^2 \subset \P(\Lambda)^\vee$ defined by $t_2 = 1$, and considering one affine chart of the blowup at $P_0 = (0,0)$ obtained by setting $t_1 = st_0$, one obtains the equation
$$\det\begin{pmatrix}
t_0^q 	&   	&   & s^qt_0^q &   &   & 1 &   &   \\
		& 	t_0^q &   &   & s^qt_0^q &   &   & 1 &   \\
t_0   & st_0   & 1 	&   	& &   & &   &    \\
&   &   & t_0   & st_0   & 1 	&   	& & \\  
 	& &   & &   &   & t_0   & st_0   & 1 	\\
	 	& &   & &   &   & \alpha   & \beta   & \gamma\\
\tilde H_{00}	&\tilde H_{01}	&\tilde H_{02}	&\tilde H_{10}	&\tilde H_{11}	&\tilde H_{12}	&\tilde H_{20}	&\tilde H_{21}	&\tilde H_{22}
\end{pmatrix} = t_0^q \det(M_2),$$
where
$$M_2 = \begin{pmatrix}
1 	&   	&   & s^q &   &   & 1 &   &   \\
		& 	1 &   &   & s^q &   &   & 1 &   \\
t_0   & st_0   & 1 	&   	& &   & &   &    \\
&   &   & t_0   & st_0   & 1 	&   	& & \\  
 	& &   & &   &   & t_0   & st_0   & 1 	\\
	 	& &   & &   &   & \alpha   & \beta   & \gamma\\
\tilde H_{00}	&\tilde H_{01}	&\tilde H_{02}	&\tilde H_{10}	&\tilde H_{11}	&\tilde H_{12}	& t_0^q\tilde H_{20}	& t_0^q\tilde H_{21}	& t_0^q\tilde H_{22}
\end{pmatrix}$$
for the pre-image (the equality follows by multiplying the last three columns by $t_0^q$ as well as multiplying the first, second, fifth and
sixth row by $t_0^{-q}$). At $t_0 = 0$, the determinant of $M_2$ becomes, up to sign,
$$\det\begin{pmatrix}
1 	&   	& s^q &   &  1 &   \\
		& 	1 &   & s^q &    & 1 \\
	 	& & &   & \alpha   & \beta   \\
\tilde H_{00}	&\tilde H_{01}		&\tilde H_{10}	&\tilde H_{11}	&& 	& 
\end{pmatrix} = \pm (\alpha s^q - \beta) \det\begin{pmatrix}
\tilde H_{00}	&\tilde H_{10}	&\tilde H_{11}
\end{pmatrix}.$$
So the preimage $P_1$ of $P_0$ in the blowup of $C'$ is the single point at $t_0 = 0$ and $s = (\beta/\alpha)^{1/q}$. Note that if $\alpha = 0$, then $\beta \neq 0$ and one can simply consider another affine patch of the blowup.

Let $\delta = (\beta/\alpha)^{1/q}$, and set $v = s - \delta$ so that $P_1$ is given by $t_0 = v = 0$. We now show that either $P_1$ is non-singular, or it is a singular point of multiplicity $2$ with two branches with distinct tangents. To do so, we compute the linear and quadratic terms of $\det(M_2)$. It is sufficient to compute $\det(M_2)$ modulo the ideal $(t_0^q,v^q)$ in $k[t_0,v]$. Up to sign, it is equal to the determinant of $M_3$ with
$$M_3 = \begin{pmatrix}
1 	&   	&   & \delta^q &   &   & 1 &   &   \\
		& 	1 &   &   & \delta^q &   &   & 1 &   \\
t_0   & \delta t_0 + v t_0   & 1 	&   	& &   & &   &    \\
&   &   & t_0   & \delta t_0 + v t_0   & 1 	&   	& & \\  
 	& &   & &   &   & t_0   & \delta t_0 + v t_0   & 1 	\\
	 	& &   & &   &   & \alpha   & \beta   & \gamma\\
\tilde H_{00}	&\tilde H_{01}	&\tilde H_{02}	&\tilde H_{10}	&\tilde H_{11}	&\tilde H_{12}	& 	& 	& 
\end{pmatrix},$$
and by subtracting the second column from the fourth one,
subtracting $\delta^q$ times the seventh column from the fourth one
as well as adding the eighth column to the fourth one, it follows that
$\det(M_3)=t_0\det(M_4)$ with
$$M_4 = \begin{pmatrix}
1 		&   	&   &   &   &   & 1 &   &   \\
		& 	1 	&   &   & \delta^q &   &   & 1 &   \\
t_0   & \delta t_0 + v t_0   & 1 	&   -\delta  - v 	& &   & &   &    \\
		&   &   & 1   & \delta t_0 + v t_0   & 1 	&   	& & \\  
 		& 	&   &  - \delta^q  + \delta  + v  &   &   & t_0   & \delta t_0 + v t_0   & 1 	\\
	 	& 	&   &  &   &   & \alpha   & \beta   & \gamma\\
\tilde H_{00}	&\tilde H_{01}	&\tilde H_{02}	&  &\tilde H_{11}	&\tilde H_{12}	& 	& 	& 
\end{pmatrix}.$$
From $\det(M_3)=t_0\det(M_4)$, we deduce that the constant term of $\det(M_4)$ gives the linear term of the equation, and the linear term in $v$ for $\det(M_4)$ gives us the quadratic term in $t_0v$ in $\det(M_3)$. In order to find these two terms, we can set $t_0 = 0$ in $M_4$. Subtracting the eighth column from the second, subtracting $\delta^q$ times the eighth from the fifth one, and removing the second row as well as the eighth one, one obtains that $\det(M_4) = \pm \det(M_5)$ with
$$M_5 = \begin{pmatrix}
1 		&   	&   &   &   &   & 1 &   \\
   &   & 1 	&   -\delta  - v 	& &   &  &    \\
		&   &   & 1   &    & 1 	&   	& \\  
 		& 	&   &  - \delta^q  + \delta  + v  &   &   &    & 1 	\\
	 	& -\beta	&   &  &  -\delta^q\beta &   & \alpha      & \gamma\\
\tilde H_{00}	&\tilde H_{01}	&\tilde H_{02}	&  &\tilde H_{11}	&\tilde H_{12}	& 	& 
\end{pmatrix}.$$
By subtracting the seventh column from the first one and removing the first row as well as
the seventh column, one obtains $\det(M_5) = \pm \det(M_6)$ with
$$M_6 = \begin{pmatrix}
   &   & 1 	&   -\delta  - v 	& & &    \\
		&   &   & 1   &    & 1 	& \\  
 		& 	&   &  - \delta^q  + \delta  + v  &   &   & 1 	\\
-\alpha	 	& -\beta	&   &  &  -\delta^q\beta &   &  \gamma\\
\tilde H_{00}	&\tilde H_{01}	&\tilde H_{02}	&  &\tilde H_{11}	&\tilde H_{12}	& 
\end{pmatrix}.$$
By subtracting the fourth column from the sixth one and removing the second row as well as the fourth column, one obtains $\det(M_6) = \pm \det(M_7)$ with
$$M_7 = \begin{pmatrix}
   	&   & 1 	 	& & \delta + v&    \\
 		& 	&     &   & \delta^q  - \delta  - v  & 1 	\\
-\alpha	 & -\beta	&   &  -\delta^q\beta &   &  \gamma\\
\tilde H_{00}	&\tilde H_{01}	&\tilde H_{02}	 &\tilde H_{11}	&\tilde H_{12}	& 
\end{pmatrix},$$
and by subtracting $\delta+v$ times the third column from the fifth one, subtracting $\delta^q  - \delta  - v$ times the sixth column from the fifth one and removing the first two rows as well as the third and sixth column, one obtains $\det(M_7) = \pm \det(M_8)$ with
$$M_8 = \begin{pmatrix}
\alpha	 & \beta	 &  \delta^q\beta &  \gamma(-\delta^q  + \delta  + v) \\
\tilde H_{00}	&\tilde H_{01}		 &\tilde H_{11} 	&(\delta+v)\tilde H_{02}-\tilde H_{12}
\end{pmatrix},$$
Finally, we have $\det(M_8) = \delta\det(M_9) - \det(M_{10}) + v \det(M_9)$ with
$$   
M_9=
\begin{pmatrix}
\alpha & \beta & \delta^q\beta & \gamma \\
\tilde{H}_{00} &
\tilde{H}_{01} & \tilde{H}_{11} &
\tilde{H}_{02} \\
\end{pmatrix},\text{ and } M_{10} = 
\begin{pmatrix}
\alpha & \beta & \delta^q\beta & \delta^q\gamma \\
\tilde{H}_{00} &
\tilde{H}_{01} & \tilde{H}_{11} &
\tilde{H}_{12} \\
\end{pmatrix}.
$$
Therefore the linear terms of the equation of the blowup
of $C'$ consist only of $\pm (\delta\det(M_9) - \det(M_{10})) t_0$,
implying that $P_1$ is non-singular if $\delta\det(M_9) \neq \det(M_{10})$.

If $\delta\det(M_9) = \det(M_{10})$, then $P_1$ is singular, and the quadratic term of the equation is of the form $t_0(\epsilon t_0 + \det(M_9) v)$ for some coefficient $\epsilon$. As long as $\det(M_9) \neq 0$, $P_1$ is a singularity of multiplicity $2$, with $2$ distinct tangents.

Suppose that $\det(M_9) = \det(M_{10}) = 0$. We get that
$$\mathrm{rank}
\begin{pmatrix}
\alpha & \beta & \gamma& \delta^q\alpha & \delta^q\beta  & \delta^q\gamma \\
\tilde{H}_{00} &
\tilde{H}_{01}&
\tilde{H}_{02}&
\tilde{H}_{10} & \tilde{H}_{11} &
\tilde{H}_{12} \\
\end{pmatrix} \leq 3.$$
Therefore, there is a linear combination of the equations defining the hyperplane that has the form
$$\begin{pmatrix}\alpha & \beta & \gamma& \delta^q\alpha & \delta^q\beta  & \delta^q\gamma & a & b & c\end{pmatrix}.$$
Consider the points $(v_0x_0 + v_1x_1 + v_2x_2)^q(u_0x_0 + u_1x_1 + u_2x_2) \in H \cap S$. First, they must satisfy the equation
$$0 = \alpha v_2u_0 + \beta v_2u_1 + \gamma v_2u_2 = v_2(\alpha u_0 + \beta u_1 + \gamma u_2),$$
so either $v_2 = 0$ or $\alpha u_0 + \beta u_1 + \gamma u_2 = 0$. Second, they must satisfy the equation
\begin{align*}
0 &= \alpha v_0^qu_0+  \beta  v_0^qu_1 +  \gamma  v_0^qu_2 +  \delta^q\alpha v_1^qu_0  +  \delta^q\beta  v_1^qu_1 +  \delta^q\gamma  v_1^qu_2 +  a  v_2^qu_0 +  b  v_2^qu_1  +  c v_2^qu_2\\
&= (v_0 +  \delta v_1)^q(\alpha u_0 +  \beta  u_1 +  \gamma  u_2) +  v_2^q(a  u_0 +  b  u_1 +  c u_2).
\end{align*}
For the points such that $\alpha u_0 + \beta u_1 + \gamma u_2 \neq 0$, one must have $v_2 = 0$, and thereby $v_0 +  \delta v_1 = 0$. The only possibility is $(v_0:v_1:v_2) = (-\delta : 1 : 0)$. Therefore, there can only be one point such that $\alpha u_0 + \beta u_1 + \gamma u_2 \neq 0$ (two such points would share the factor $-\delta^q x_0^q + x_1^q$, a contradiction). So all the points of $H\cap S$ satisfy $\alpha u_0 + \beta u_1 + \gamma u_2 = 0$, also a contradiction (among the exceptional points, at most $3$ can lead to $(u_0 : u_1 : u_2)$ being on a given line).

\begin{cor}
The curve $C'$ has four singularities. Each of them has multiplicity $q$, and is either analytically irreducible, or one blowup results in a node (the intersection of two smooth branches with distinct tangents).
\end{cor}

\begin{proof}
We have just shown the last part of the statement: any singular point has multiplicity $q$, and is either analytically irreducible, or one blowup results in a node.
Let us show that there are four singularities.
A point is singular if and only if, after the transformation sending its preimage in $C$ to $(x_1^qx_0 - x_0x_1^q, (0:0:1))$, we have a relation of the form
$$\begin{pmatrix}0	&0 &0	&0	&0 & 0 & \alpha & \beta & \gamma
\end{pmatrix}.$$
Applied to the points $v^qu = (v_0x_0 + v_1x_1 + v_2x_2)^q(u_0x_0 + u_1x_1 + u_2x_2) \in H \cap S$, we get
$$v_2^q(\alpha u_0 +\beta u_1 + \gamma u_2) = 0.$$
For each $v^qu \in H \cap S$, either $v^{(q)}$ is on the line $(0:0:1)$, or $u$ is on the line $(\alpha : \beta : \gamma)$. There are six such $v^qu$, but no more than three values $u$ or $v$ can lie on any given line.
Therefore, for three of the $v^qu \in H \cap S$, the $v^{(q)}$-values are aligned on $(0:0:1)$, and for the other three $v^qu \in H \cap S$, the $u$-values are aligned. Therefore, a point on $C' \subset \P(V)^\vee$ is singular if and only if it corresponds to one of the lines in $\P(V)$ that contain three $v^{(q)}$-values; since the divisor is not a trap, there are exactly four such lines.
\end{proof}

\begin{lem}\label{lemma:fulton-bound}
Let $P$ be a point at the intersection of two components $A$ and $B$ of the curve $C'$.
The intersection number $I(P, A \cap B)$ is at most $q^2/4+1$.
\end{lem}

\begin{proof}
The point $P$ is a singularity of $C'$, with multiplicity $q$. Its blowup is a node, so it belongs to at most two components of $C'$, which in the present case are $A$ and $B$. 
Let $f : Z \rightarrow \P(V)^\vee$ be the blowup at $P$, let $\tilde A$, $\tilde B$ and $\tilde C'$ be the corresponding blowups of $A$, $B$ and $C'$ respectively, let $E$ be the exceptional divisor, and let $\tilde P$ the unique preimage of $P$ in $C'$ (it is a node, at the intersection of $\tilde A$ and $\tilde B$). Applying the formula~\cite[Corollary~6.7.1]{Fulton}, we get
\begin{align*}
I(P, A\cap B) & = I(\tilde P, f^*A\cap f^*B) \\
&= I(\tilde P, (\tilde A + e_{P}(A) E)\cap (\tilde B + e_{P}(B) E)) \\
&=  I(\tilde P, \tilde A \cap \tilde B) + e_{P}(A)e_{P}(B),
\end{align*}
where $e_{P}(A)$ and $e_{P}(B)$ are the multiplicities of $P$ on $A$ and $B$. Since $e_{P}(A) + e_{P}(B) = q$, we have $e_{P}(A)e_{P}(B) \leq (q/2)^2$. Since $\tilde P$ is a node at the intersection of $\tilde A$ and $\tilde B$, we have $I(\tilde P, \tilde A \cap \tilde B) = 1$. Therefore, $I(P, A \cap B) \leq (q/2)^2 + 1$.
\end{proof}

\begin{prop}\label{prop:patak}
All the absolutely irreducible component of the curve $X_0$ are defined over $k$.
\end{prop}

\begin{proof}
Since it is a plane curve, any two components of $C'$ must intersect, and they can only do so at the singular points $T \subset C'$. Also, a given singular point can be contained in at most $2$ of the components (at most one if it is analytically irreducible, and at most $2$ if its blowup is a node).

First observe that the number of pairs of irreducible components is at most the number of singularities, so there are at most $3$ components.
Second, observe that each component is defined over an extension $K/k$ of degree at most $2$. Indeed, degree $4$ or more would contradict the previous observation. For the degree $3$ case, since there are at most three components, there must be exactly~$3$ Galois-conjugate components. 
None of the four singularities can be fixed by the Galois action (such a singularity would appear in all three components). Yet, the number of singularities is not divisible by $3$, a contradiction.

If there is only one absolutely irreducible component, it is $X_0$ itself, which is defined over $k$, and we are done.

If there are two components, either they are both defined over $k$ and we are done, or $C' = A \cup B$, where $A$ and $B$ are two Galois-conjugate absolutely irreducible plane curves. We now deal with the latter case. Since $C'$ has degree $2q+2$, the components each have degree $q+1$, so by B\'ezout's theorem,
$$(q+1)^2 = A\cdot B = \sum_{P \in A \cap B} I(P, A \cap B) = \sum_{P \in T} I(P, A \cap B).$$
From Lemma~\ref{lemma:fulton-bound}, for any $P \in T$ we have $I(P, A \cap B) \leq q^2/4 + 1$, so $(q+1)^2 \leq q^2+4$, a contradiction. 

Finally, it remains to deal with the case were there are $3$ components.
If they are all defined over $k$, we are done, so let us suppose that  $C' = A \cup A^\sigma \cup B$, where $A$ is defined over a quadratic extension of $k$ and $\sigma$ is the corresponding conjugation.
Let $a=\deg(A)$ and $b = \deg(B)$. The only possible configurations of the singular points $T = \{P_1,\dots,P_4\}$ are
\begin{enumerate}
\item \label{eq:firstbadconfiguration}$A \cap A^\sigma = \{P_1\} = \{P_1^\sigma\}, A \cap B = \{P_2\}, A^\sigma \cap B = \{P_3\} = \{P_2^\sigma\},\text{ and }P_4 = P_4^\sigma \in B,$ or
\item \label{eq:secondbadconfiguration}$A \cap A^\sigma = \{P_1,P_4\} = \{P_1^\sigma, P_4^\sigma\}, A \cap B = \{P_2\}, A^\sigma \cap B = \{P_3\} = \{P_2^\sigma\}.$
\end{enumerate}
Write $e_i(Z) =e_{P_i}(Z)$ for the multiplicity of $P_i$ on any component $Z$. 
Observe that $a \geq q/2$ (in the case~\eqref{eq:firstbadconfiguration}, it follows from $e_1(A) = e_1(A^\sigma)$, and $e_1(A) + e_1(A^\sigma) = q$; in the case~\eqref{eq:secondbadconfiguration}, it follows from $a \geq \max\{e_1(A),e_4(A)\}$, $e_i(A)+e_i(A^\sigma) = q$ for $i\in \{1,4\}$ and $\{P_1,P_4\} = \{P_1^\sigma, P_4^\sigma\})$.
Secondly, from B\'ezout's theorem, we have $$a^2 = \sum_{P \in A \cap A^\sigma}I(P, A \cap A^\sigma) \leq |A \cap A^\sigma|\cdot(q^2/4 + 1) \leq q^2/2 + 2.$$
Therefore $a \leq \sqrt{q^2/2 + 2} \leq (q+1)/\sqrt{2}$.
On one hand, since $2a+b = 2q+2$, we get
$$ab \geq \frac{qb}{2} =  \frac{q(2q+2 - 2a)}{2} \geq \frac{q(2q+2 - \sqrt{2}(q+1))}{2} = \frac{q(q+1)(2-\sqrt{2})}{2} \geq 0.29 q(q+1).$$
On the other hand, we have
$$ab = I(P_2, A \cap B) \leq 1+q^2/4.$$
This implies $0.29 q(q+1) \leq 1+q^2/4$, a contradiction.
\end{proof}

\subsection{Defining equations for $X_1$}\label{sec:defeqforx1}
\noindent
Consider the action of $\PGL_3$ on $\P(V) \times \P(\Lambda)$, and let $W$ be the closure of the orbit of $(\mfd,x_0)$.
Suppose $u = u_0x_0 + u_1x_1 + u_2x_2 \in \P(\Lambda)$. We focus on the affine patch $u_0 = 1$, since the rest of the proof is a study of local properties of points on this patch. Consider the matrix
$$m = \left(\begin{matrix}
1 & -u_1  & -u_{2} \\
0 & 1  & 0\\ 
0 & 0 &   1
\end{matrix}\right).$$
We have $m^\transpose u = x_0$, so $u$ divides $f$ if and only if $x_0$ divides $m\star f$, i.e., the coefficient of $x_i^{q}x_j$ in $m\star f$ is zero for any $i,j \neq 0$. Write $f = \sum_{i,j}a_{ij}x_i^{q}x_j.$ Then,
\begin{alignat*}{1}
m \star f & = 
a_{00}x_0^qx_0 + \sum_{i,j \neq 0}(a_{ij} + a_{00}u_i^qu_j - a_{i0}u_j - a_{0j}u_i^q)x_i^{q}x_j\\
&\ \ \ + \sum_{i \neq 0}(a_{i0} - a_{00}u_i^q)x_i^{q}x_0 + \sum_{j \neq 0}(a_{0j} - a_{00}u_j)x_0^{q}x_j.
\end{alignat*}
We deduce that the equations corresponding to the condition that $u$ divides $f$ are
$$E_{ij}: a_{ij} + a_{00}u_i^qu_j - a_{i0}u_j - a_{0j}u_i^q = 0$$
for any indices $i,j\neq 0$.
Assuming these hold, we have
\begin{alignat*}{1}
m \star f = a_{00}x_0^qx_0 + \sum_{i \neq 0}(a_{i0} - a_{00}u_i^q)x_i^{q}x_0 + \sum_{j \neq 0}(a_{0j} - a_{00}u_j)x_0^{q}x_j.
\end{alignat*}
Now, since $m\star f$ is divisible by $x_0$, it is in $\PGL_3 \star \mfd$ if and only if $m\star f = \tilde m \star \mfd$ for some matrix $\tilde m$ such that $\tilde m^\transpose x_0 = x_0$.
Writing $\tilde m^\transpose x_1 = \sum_{i}b_ix_i$, we get
$$\tilde m \star \mfd = x_0^q\left(\sum_{i=0}^2b_ix_i\right) - x_0\left(\sum_{i=0}^2b_ix_i\right)^q = (b_0-b_0^q)x_0^{q}x_0 - 
\sum_{i\neq 0} b_i^qx_i^qx_0 + 
\sum_{j \neq 0} b_jx_0^qx_j.$$
Therefore, we obtain $W$ by adding the equation
$$F_{12}: 
(a_{10} - a_{00}u_1^q)(a_{02} - a_{00}u_2)^q - (a_{20} - a_{00}u_2^q)(a_{01} - a_{00}u_1)^q = 0.$$
We obtain that in the affine patch $u_0 = 1$, the curve $X_1$ is defined by these equations for $W$ and the equations defining the hyperplane $H \subset \P(V)$.

\subsection{Desingularisation at $a \in X_1$}
Recall that we have fixed an exceptional point $s = (V^1)^qU^1 \in X_0$, and its two preimages $a = ((V^1)^qU^1, U_1)$ and $b = ((V^1)^qU^1, V^1)$ in $X_1$.
We need to prove that the conditions of Proposition~\ref{prop:structure-of-proof} are satisfied, starting with the analytic irreducibility of $a$.

\begin{lem}\label{lem:aIsAnaIrred}
The point $a \in X_1$ is analytically irreducible.
\end{lem}

\begin{proof}
Take a matrix in $\PGL_3$ sending $U^1$ to $x_0$, $V^1$ to $x_1$, and $V^2$ to $x_2$ as in Lemma~\ref{lem:blowup-u-matrix}.
Consider the variety $\tilde X_1 \subset \P(V)\times \P(\Lambda) \times \P(\Lambda)^\vee$ with coordinates
$a_{ij}, u_i, t_i$ (and made affine by $a_{10}=u_0=t_2=1$) defined by the equations of the hyperplane $H \subset \P(V)$, as well as $E_{ij}, F_{12}$ and the equations $e_k$ and $f_k$ from Section~\ref{subsec:irreducibilityofX0}. The equations $e_0, e_1$ and $f_0$ are
\begin{align*}
a_{00} t_0 + a_{01} t_1 + a_{02} &= 0,\\
    t_0 + a_{11} t_1 + a_{12} &= 0,\\
a_{00} t_0^q  + t_1^q + a_{20} &= 0.
\end{align*}
If $a_{00} a_{11}^q \neq 1$, the last two equations determine $t_0$ and $t_1$ uniquely
for any given $a_{ij}$, so in a neighbourhood of the point $a$ the projection $\tilde X_1 \rightarrow X_1$ is one-to-one. Computing the Jacobian matrix, we see that the point
is smooth if the matrix $\begin{pmatrix}H_{00} &H_{01}&H_{11}&H_{12}\end{pmatrix}$ has rank $4$, which is the case, as proved in Lemma~\ref{lem:blowup-u-matrix}. Therefore, $\tilde X_1 \rightarrow X_1$ is a desingularisation of $X_1$ at $a$. Since $a$ has a single preimage, it is analytically irreducible.
\end{proof}

\subsection{Blowing up $b \in X_1$}
Recall that $b = ((V^1)^qU^1, V^1)$. Also, we have the $6$ points $(V^i)^qU^i$ in the intersection with $S$, and the $3$ points $U^1,U^3,U^4$ are aligned, as well as $V^2,V^5,V^6$. Apply the action of a matrix in $\PGL_3$ sending $U^1$ to $x_1$, and $V^1$ to $x_0$.
With this transformation, $b$ belongs to the affine patch $u_0 = 1$ of $\P(V)\times\P(\Lambda)$, so we can study it locally through the equations of $X_1$ derived in Section~\ref{sec:defeqforx1}.

\begin{lem}\label{lem:blowup-v-matrix}
When $D$ is not a trap, the matrix defining $H$ can be written as
\begin{equation*}
\left(
\begin{matrix}
1		& 0		& 0		 	& 0			& *		& *		& 0		& *		& *	\\
0		& 0		& 1			& 0			& *		& *		& 0		& *		& *	\\
0		& 0		& 0			& 1			& C_{11}		& *		& 0		& C_{21}		& *	\\
0	 	& 0		& 0			& 0			& D_{11}		& *		& 1		& D_{21}		& *	\\
\end{matrix}
\right)
\end{equation*}
with $C_{11}D_{21}-C_{21}D_{11}	\neq 0$ and $D_{11} \neq 0$, and the fifth and sixth columns are linearly independent.
\end{lem}

\begin{proof}
The matrix can be written in the form
\begin{equation*}
\left(
\begin{matrix}
1		& 0		& 0		 	& *&*&*&*&*		&*	\\
0		& 0		& 1			& *&*&*&*&*		&*	\\
0		& 0		& 0			& C_{10}	& C_{11}		& C_{12}		& C_{20}		& C_{21}		& C_{22}	\\
0	 	& 0		& 0			& 0			& D_{11}		& D_{12}		& D_{20}		& D_{21}		& D_{22}	\\
\end{matrix}
\right).
\end{equation*}
If $C_{10} D_{20}\neq 0$, the matrix can then be written as in the lemma. By contradiction, suppose that $C_{10} D_{20} = 0$; we deduce that there is a relation of the form
\begin{equation}\label{eqref:specialrelationZ}
\left(
\begin{matrix}
0		& 0		& 0			& 0			& Z_{11}		& Z_{12}		& 0			& Z_{21}		& Z_{22}	\\
\end{matrix}
\right).
\end{equation}
We first show that $(Z_{11},Z_{21})$ and $(Z_{12},Z_{22})$ are linearly independent.
By contradiction, suppose there exists $(\alpha,\beta) \neq (0,0)$ such that $\alpha(Z_{12},Z_{22}) = \beta(Z_{11},Z_{21})$.
The relation becomes
$$(\alpha u_1 + \beta u_2)(Z_{11}v_1^q + Z_{21}v_2^q).$$
At most $3$ of the $(V^i)^{(q)}$-values are on the line $(0:Z_{11}:Z_{21})$, so at least $3$ of the $U^i$-values are on the line $(0:\alpha:\beta)$, which also contains $V^1$, a contradiction. So $(Z_{11},Z_{21})$ and $(Z_{12},Z_{22})$ are linearly independent.

Applying Relation~\eqref{eqref:specialrelationZ} to $U^4$ and $V^4$, we get that $(V^4)^{(q)}$ is on the line 
$$L^4 = (0 : Z_{11}U^4_1 + Z_{12}U^4_2 : Z_{21}U^4_1 + Z_{22}U^4_2).$$
But $(V^1)^{(q)} = x_0$ also lies on this line, therefore so does $(V^{5})^{(q)}$. Similarly, the relation applied to $U^5,V^5$ implies that $(V^5)^{(q)}$ lies on the line
$$L^5 = (0 : Z_{11}U^5_1 + Z_{12}U^5_2 : Z_{21}U^5_1 + Z_{22}U^5_2).$$
Since $L^5$ also contains $(V^1)^{(q)}$, we have $L^4 = L^5$.
Note that $(U^4_1,U^4_2)$ and $(U^5_1,U^5_2)$ are linearly independent (otherwise $U^4,U^5$ and $V^1$ would be aligned). Therefore, the equality $L^4 = L^5$ implies that $(Z_{11},Z_{21})$ and $(Z_{12},Z_{22})$ are linearly dependent, a contradiction. 

The matrix can be written as stated in the lemma, and it remains to prove the additional properties. A proof similar to the above shows that $C_{11}D_{21}-C_{21}D_{11}	= 0$ implies that $V^1,V^3,$ and $V^4$ are aligned, another contradiction. 

If the fifth column $H_{11}$ and sixth column $H_{12}$ are linearly dependent, there exists a non-zero pair $(\alpha,\beta)$ such that $\alpha H_{11} = \beta H_{12}$. Then, the polynomial $x_1^q(\alpha x_1 - \beta x_2)$ is an exceptional point of the curve, so $x_1$ is one of the $V^i$-points, yet $x_1 = U^1$, a contradiction.

Finally, suppose by contradiction that $D_{11} = 0$.
The case $D_{12} = 0$ easily leads to a contradiction, so assume $D_{12} \neq 0$.
For any $i\neq 1$, we have $U_2^i \neq 0$ and $$(V^i)^{(q)} \in (0: U_2^iD_{12} : U_0^i + U_1^iD_{21} + U_2^iD_{22}).$$ Since $V^1= x_0$ belongs to all of these lines, and $V^1,V^2,V^3$ are aligned, we get
$$(0: U_2^2D_{12} : U_0^2 + U_1^2D_{21} + U_2^2D_{22}) = (0: U_2^3D_{12} : U_0^3 + U_1^3D_{21} + U_2^3D_{22})$$
Let $\alpha = (U_0^2 + U_1^2D_{21} + U_2^2D_{22})/U_2^2 = (U_0^3 + U_1^3D_{21} + U_2^3D_{22})/U_2^3$. We have $U^2, U^3 \in (1:D_{21}:D_{22} - \alpha)$, and therefore $U^6 \in (1:D_{21}:D_{22} - \alpha)$. We deduce
\begin{align*}
0 &= D_{12}(V_1^6)^{q}U_2^6 + (V_2^6)^{q}(U_0^6	+ D_{21}U_1^6		+ D_{22}U_2^6)\\
&= U_2^6(D_{12}(V_1^6)^{q} + \alpha(V_2^6)^{q}).
\end{align*}
Therefore, either $(V^6)^{(q)} \in (0:D_{12}:\alpha)$ so $V^6$ is aligned with $V^1, V^2, V^3$, or $U^6 \in (0:0:1)$, so $U^6$ is aligned with $U^1$ and $V^1$, each being a contradiction.
\end{proof}

\begin{lem}\label{lem:bIsAnaIrred}
If $D_{11}D_{12}^q \neq D_{11}^qD_{21}$, the point $b \in X_1$ is analytically irreducible.
\end{lem}
\begin{rem}
We deal with the case $D_{11}D_{12}^q = D_{11}^qD_{21}$ in Section~\ref{subsubsec:annoyingbadcasetrap}.
\end{rem}

\begin{proof}
From Lemma~\ref{lem:blowup-v-matrix}, we can rewrite the matrix defining $H$ as
\begin{equation*}
\left(
\begin{matrix}
1		& 0		& 0		 	& 0			& A_{11}		& A_{12}		& 0		& A_{21}		& A_{22}	\\
0		& 0		& 1			& 0			& B_{11}		& B_{12}		& 0		& B_{21}		& B_{22}	\\
0		& 0		& 0			& 1			& C_{11}		& C_{12}		& 0		& C_{21}		& C_{22}	\\
0	 	& 0		& 0			& 0			& D_{11}		& D_{12}		& 1		& D_{21}		& D_{22}	\\
\end{matrix}
\right).
\end{equation*}
Let us blow up via $u_2t_1 = u_1t_2$, and focus on the affine patch $t_1 = 1$. 
We get the equations
\begin{alignat*}{1}
E_{11}&= a_{11} + u_1(a_{00}u_1^q - a_{10} - a_{01}u_1^{q-1}),\\
E_{12}&= a_{12} + u_1(a_{00}u_1^qt_2 - a_{10}t_2 - a_{02}u_1^{q-1}),\\
E_{21}&= a_{21} + u_1(a_{00}u_1^qt_2^q - a_{20} - a_{01}u_1^{q-1}t_2^q),\\
E_{22}&= a_{22} + u_1(a_{00}u_1^qt_2^qt_2 - a_{20}t_2 - a_{02}u_1^{q-1}t_2^q),\\
F_{12}&= (a_{10} - a_{00}u_1^q)(a_{02} - a_{00}u_1t_2)^q - (a_{20} - a_{00}u_1^qt_2^q)(a_{01} - a_{00}u_1)^q.
\end{alignat*}
For any $Z \in \{A,B,C,D\}$, write $Z_{*i} = Z_{1i} + Z_{2i}t_2^q$, and $Z_{i*} = Z_{i1} + Z_{i2}t_2$. We get the relations
\begin{equation*}
\left(
\begin{matrix}
1-u_1^{q+1}(A_{1*} + A_{2*}t_2^q)	& u_1^qA_{*2}	& u_1A_{1*}		& u_1A_{2*}		& u_1^qA_{*1}\\
-u_1^{q+1}(B_{1*} + B_{2*}t_2^q)		& 1+u_1^qB_{*2}	& u_1B_{1*}		& u_1B_{2*}		& u_1^qB_{*1}\\
-u_1^{q+1}(C_{1*} + C_{2*}t_2^q)		& u_1^qC_{*2}	& 1+u_1C_{1*}	& u_1C_{2*}		& u_1^qC_{*1}\\
-u_1^{q+1}(D_{1*} + D_{2*}t_2^q)		& u_1^qD_{*2}	& u_1D_{1*}		& 1+u_1D_{2*}	& u_1^qD_{*1}\\
\end{matrix}
\right)
\left(
\begin{matrix}
a_{00}	\\
a_{02}		\\
a_{10}	\\
a_{20}		\\
a_{01}		\\
\end{matrix}
\right) = 0.
\end{equation*}
Eliminating $a_{00}, a_{02}, a_{10}$ and $a_{20}$ in the ring of formal power series $k[[u_1,t_2]]$ yields
\begin{alignat*}{1}
a_{00} &= u_1^q(-A_{*1} + u_1c_{00}) = u_1^qb_{00},\\
a_{02} &= u_1^q(-B_{*1} + u_1c_{02}) = u_1^qb_{02},\\
a_{10} &= u_1^q(-C_{*1} + u_1c_{10}) = u_1^qb_{10},\\
a_{20} &= u_1^q(-D_{*1} + u_1c_{20}) = u_1^qb_{20},
\end{alignat*}
for some $b_{ij}$ and $c_{ij}$ in $k[[u_1,t_2]]$.
In the affine patch $a_{01} = 1$, we get the equation
$$u_1^{q}u_1^{q^2}(b_{10} - b_{00}u_1^q)(b_{02} - b_{00}u_1t_2)^q - u_1^q(-D_{*1} + u_1c_{20} - b_{00}u_1^qt_2^q)(1 - a_{00}u_1)^q,$$
and removing the factor $u_1^q$,
$$u_1^{q^2}(b_{10} - b_{00}u_1^q)(b_{02} - b_{00}u_1t_2)^q - (-D_{*1} + u_1c_{20} - b_{00}u_1^qt_2^q)(1 - a_{00}u_1)^q.$$
Lemma~\ref{lem:blowup-v-matrix} implies that $(D_{11},D_{21}) \neq(0,0)$, therefore $D_{*1}(t_2) = 0$ has a unique solution (possibly at infinity), with multiplicity $q$.
We have
$$c_{20} = D_{*1}D_{2*} + D_{1*}C_{*1}+ u_1(\dots).$$
Since $(C_{11},C_{21})$ and $(D_{11},D_{21})$ are not collinear (Lemma~\ref{lem:blowup-v-matrix}), the polynomials $C_{*1}$ and $D_{*1}$ do not share a root. We deduce that when $D_{11}D_{12}^q \neq D_{11}^qD_{21}$, the power series $c_{20}$ is a unit, and we are done.
\end{proof}
\begin{rem}\label{rem:c20nonzeroterms}
Note for later that, in the variables $t = (D_{*1}(t_2))^{1/q}$ and $u_1$, the power series $c_{ij}$ reduced modulo $(u_1,t^2)$ are linear polynomials which do not all have the same root (because the fifth and sixth columns in Lemma~\ref{lem:blowup-v-matrix} are linearly independent).
\end{rem}
\subsection{Ramification}
The goal of this section is to show that the normalisation of the cover $X_1 \rightarrow X_0$ is unramified at $a$ (condition~\ref{item:3:prop:structure-of-proof} in Proposition~\ref{prop:structure-of-proof}).

To do so, we work with the desingularisations $C \subset \P(V)\times \P(\Lambda)^\vee$ and $\tilde X_1 \subset \P(V)\times \P(\Lambda)\times \P(\Lambda)^\vee$ at $s$ and $a$, introduced in Section~\ref{subsec:irreducibilityofX0} and Lemma~\ref{lem:aIsAnaIrred} respectively.
As previously, we take a matrix in $\PGL_3$ sending $U^1$ to $x_0$, $V^1$ to $x_1$, and $V^2$ to $x_2$,
and we work with the coordinates
$a_{ij}, u_i, t_i$ for $\P(V)\times \P(\Lambda)\times \P(\Lambda)^\vee$ (made affine by $a_{10}=u_0=t_2=1$). Recall that $C$ is defined by the equations of the hyperplane $H \subset \P(V)$, and the equations $e_k$ and $f_k$ from Section~\ref{subsec:irreducibilityofX0}, and $\tilde X_1$ is defined by the same equations together with $E_{ij}$ and $F_{12}$. The cover $\tilde X_1 \rightarrow C$ is a restriction and corestriction of the projection $\P(V)\times \P(\Lambda)\times \P(\Lambda)^\vee \rightarrow \P(V)\times \P(\Lambda)^\vee$. Therefore, if one of $a_{ij}$ or $t_i$ is a uniformizing parameter for $a$ on $\tilde X_1$, then the cover is unramified at $a$. Let us show that it is the case.
From $E_{11}$, if $u_1$ is a uniformizing parameter at $a$, then so is $a_{11}$. From $E_{12}$, the same holds for $u_2$ and $a_{12}$. Therefore, either
\begin{enumerate}
\item $a_{11}$ or $a_{12}$ is a local parameter at $a$ on $\tilde X_1$, or
\item neither $u_1$ nor $u_2$ is a local parameter at $a$ on $\tilde X_1$, so one of $a_{ij}$ or $t_i$ must be (because $\tilde X_1$ is smooth at $a$).
\end{enumerate}
In either case, we deduce that the cover is unramified at $a$.

\subsection{Blowing up $(b,b) \in X_2$}
\begin{lem}
If $D_{11}D_{12}^q \neq D_{11}^qD_{21}$, the point $(b,b)$ is analytically irreducible.
\end{lem}
\begin{proof}
Analytically at the point $b$, the desingularised equation derived in Lemma~\ref{lem:bIsAnaIrred} is of the form (up to multiplication by a unit in $k[[u,t]]$)
$$u - t^qB(u,t),$$
where $B(u,t)$ is a unit (and more precisely, $B(u,t) \equiv c_{20}(u,t)^{-1} \equiv \gamma_0 - \gamma_1t \mod (u,t^2)$, where $c_{20}(u,t) \equiv \gamma_0 + \gamma_1t\mod (u,t^2)$ and $\gamma_0\neq 0$).
The fibre product at this point with respect to the projection to $X_0$ is given by the equations in $k[[u,t,v,s]]$
\begin{alignat*}{1}
&u - t^qB(u,t) = 0,\\
&v - s^qB(v,s) = 0,\\
&u^qb_{ij}(u,t) - v^qb_{ij}(v,s) = 0,
\end{alignat*}
for all pairs $i,j$. There is an automorphism of $k[[u,t]]$ sending $u - t^qB(u,t)$ to $u$ while fixing $t$. It sends $u$ to some $F(u,t) = u + t^qG(u,t)$ where $G$ is another unit (which also satisfies $G(u,t) \equiv \gamma_0 - \gamma_1t \mod (u,t^2)$). The same applied to $k[[v,s]]$ sends $v - s^qB(v,s)$ to $v$, and $v$ to $F(v,s)$, and fixes $s$. Therefore the curve is isomorphic to the curve given by the equations in $k[[t,s]]$
$$F(0,t)^qb_{ij}(F(0,t),t) - F(0,s)^qb_{ij}(F(0,s),s) = 0.$$
For simplicity, we just write
$$F(t)^q b_{ij}(F(t),t) - F(s)^q b_{ij}(F(s),s) = 0.$$
Write $F(t) = t^qG(t)$ where $G(t) = G(0,t)$ is a unit (with $G(t) \equiv \gamma_0 - \gamma_1t \mod t^2$).
The equations above are divisible by $t-s$ (which corresponds to the diagonal component of the fibre product).
By blowing up with $t = st'$, we get
$$\frac{s^{q^2}t'^{q^2}G(st')^q b_{ij}(F(st'),st') - s^{q^2}G(s)^q b_{ij}(F(s),s)}{s(t'-1)} = 0.$$
The numerator has a factor $s^{q^2+q}$, the exceptional divisor. The remaining factor has a unique solution at $s = 0$ given by $t' = 1$. In terms of the variables $(t'-1)$ and $s$, its smallest degree term is $\gamma_0^q(\gamma_0\delta_1
 - \gamma_1\delta_0)s(t'-1)$ where $c_{ij}(u,t) \equiv \delta_0 + \delta_1t\mod (u,t^2)$. From Remark~\ref{rem:c20nonzeroterms}, there are indices $i,j$ such that $\gamma_0\delta_1
 - \gamma_1\delta_0 \neq 0$, therefore the blowup is non-singular.
\end{proof}

\subsection{The case $D_{11}D_{12}^q = D_{11}^qD_{21}$}  \label{subsubsec:annoyingbadcasetrap}
It only remains to show that the case $D_{11}D_{12}^q = D_{11}^qD_{21}$ can be avoided: it corresponds to $D$ being some kind of trap.

\begin{lem}\label{lem:trapsDDq}
One can choose $s \in X_0\cap S$ such that $D_{11}D_{12}^q \neq D_{11}^qD_{21}$, unless $D$ belongs to a strict closed subvariety $\mathscr T_4^5$ of $\mathscr D_4$.
For any $P_0,P_1 \in E$, we have $\mathscr P_2(P_0) + \mathscr P_2(P_1) \not\subset \mathscr T_4^5$.
\end{lem}

\begin{proof}
Let $\mathscr T_4^5$ be the subvariety of $\mathscr D_4$ such that $D_{11}D_{12}^q = D_{11}^qD_{21}$ for all the corresponding exceptional points. 
We need to show that for any $P_0,P_1 \in E$, we have $\mathscr P_2(P_0) + \mathscr P_2(P_1) \not\subset \mathscr T_4^5$.
Consider points $R, T \in E$, and the divisor $D = \sum_{i=1}^4[D_i]\in \mathscr P_2(P_0) + \mathscr P_2(P_1)$ where
$$D_1 = R, D_2 = P_0 - R, D_3 = T,\text{ and } D_4 = P_1 - T.$$
We will assume that $(R,T)$ does not fall in certain strict subvarieties of $E^2$.
With $R\neq T$ and $P_0 - R \neq P_1 - T$, let $u^1 = \ell(D_1,D_3)$, and $(v^1)^{(q)} = \ell(D_2 + Q, D_4 + Q)$.
More explicitly, we have
\begin{align*}
u_0^1 &= \det\begin{pmatrix}
x(D_1) & y(D_1)\\
x(D_3) & y(D_3)\\
\end{pmatrix},&
(v_0^1)^q &= \det\begin{pmatrix}
x(D_2 + Q) & y(D_2 + Q)\\
x(D_4 + Q) & y(D_4 + Q)\\
\end{pmatrix},\\
u_1^1 &= \det\begin{pmatrix}
y(D_1) & 1\\
y(D_3) & 1\\
\end{pmatrix},&
(v_1^1)^q &= \det\begin{pmatrix}
y(D_2 + Q) & 1\\
y(D_4 + Q) & 1\\
\end{pmatrix},\\
u_2^1 &= \det\begin{pmatrix}
1 & x(D_1)\\
1 & x(D_3)\\
\end{pmatrix},&
(v_2^1)^q &= \det\begin{pmatrix}
1 & x(D_2 + Q)\\
1 & x(D_4 + Q)\\
\end{pmatrix}.
\end{align*}
The following computation shows that $D_{11}D_{12}^q - D_{11}^qD_{21}$ evaluated at the exceptional point $(v^1)^{q}u^1$ is a non-zero rational function of $R$ and $T$.
Note that the $v^1_i$-values are not rational functions of $R$ and $T$, but the $(v^1_i)^q$-values are.
Consider the matrix
$$m = \begin{pmatrix}
v_0^1 & v_1^1 & v_2^1\\
u_0^1 & u_1^1 & u_2^1\\
0	  & 0	  & 1	
\end{pmatrix},$$
which sends the line $u^1$ to the line $(0:1:0)$ and $v^1$ to the line $(1:0:0)$. 
Observe that this matrix is non-singular away from a strict subvariety of pairs $(R,T) \in E^2$. Indeed,
we have that $(v_0^1u_1^1 - v_1^1u_0^1)^q$ is a non-zero rational function of $R$ and $T$: up to a linear transformation, we can assume $(0,0) \in E$, and choose $T = P_1 - (0,0)$, so that $(v_0^1)^q = 0$ and $(v_1^1)^q = y(D_2 + Q)$; we get $(v_0^1u_1^1 - v_1^1u_0^1)^q = -(y(D_2 + Q)u_0^1)^q$, and neither $y(D_2 + Q)$ nor $u_0^1$ is the zero function of $R$.
For any $j$, we have $m(D_j) = (a_0^j : a_1^j : a_2^j)$ and $m^{(q)}(D_j+Q) = (b_0^j : b_1^j : b_2^j)$ where
\begin{align*}
a_0^j &= v_0^1 + v_1^1x(D_j) + v_2^1y(D_j), & b_0^j &= (v_0^1)^{q} + (v_1^1)^{q} x(D_j + Q) + (v_2^1)^{q} y(D_j + Q),\\
a_1^j &= u_0^1 + u_1^1x(D_j) + u_2^1y(D_j), & b_1^j &= (u_0^1)^{q} + (u_1^1)^{q} x(D_j + Q) + (u_2^1)^{q} y(D_j + Q),\\
a_2^j &= y(D_j), & b_2^j &= y(D_j + Q).
\end{align*}
Note that $a_1^1 = a_1^3 = 0$ and $b_0^2 = b_0^4 = 0$.
The matrix defining $H$ (after applying the action of $m$) is
$$\begin{pmatrix}
b_0^1a_0^1 	& 	0		 & b_0^1a_2^1 & b_1^1a_0^1 & 0			& b_1^1a_2^1 & b_2^1a_0^1 & 0			& b_2^1a_2^1 \\
0			&	0		 & 	0		  & b_1^2a_0^2 & b_1^2a_1^2 & b_1^2a_2^2 & b_2^2a_0^2 & b_2^2a_1^2 & b_2^2a_2^2 \\
b_0^3a_0^3 & 	0		 & b_0^3a_2^3 & b_1^3a_0^3 & 0			& b_1^3a_2^3 & b_2^3a_0^3 & 0			 & b_2^3a_2^3 \\
0			&	0		 & 	0		  & b_1^4a_0^4 & b_1^4a_1^4 & b_1^4a_2^4 & b_2^4a_0^4 & b_2^4a_1^4 & b_2^4a_2^4
\end{pmatrix}.$$
From Lemma~\ref{lem:blowup-v-matrix}, apart from a strict subvariety of $(P_0,P_1) \in E^2$, one must have $$\det \begin{pmatrix}
b_0^1a_0^1 	& 	b_0^1a_2^1\\
b_0^3a_0^3 & 	b_0^3a_2^3
\end{pmatrix} \neq 0,$$
We deduce
\begin{align*}
D_{11} &= b_1^2a_0^2 \times b_1^4a_1^4 - b_1^4a_0^4 \times b_1^2a_1^2 = b_1^2b_1^4(a_0^2  a_1^4 - a_0^4  a_1^2),\\
D_{12} &= b_1^2a_0^2 \times b_1^4a_2^4 - b_1^4a_0^4 \times b_1^2a_2^2 = b_1^2b_1^4(a_0^2  a_2^4 - a_0^4  a_2^2),\\
D_{21} &= b_1^2a_0^2 \times b_2^4a_1^4 - b_1^4a_0^4 \times b_2^2a_1^2.
\end{align*}
Suppose by contradiction that $D_{11}D_{12}^q - D_{11}^qD_{21} = 0$.
We have $D_{11}^qD_{12}^{q^2} - D_{11}^{q^2}D_{21}^q \in k(E)$.
Consider a universal $k$-derivation $\derivation$ of $k(E)$.
We have already proved in Lemma~\ref{lem:blowup-v-matrix} that when $D$ is not in $\cup_{i=1}^4\mathscr T_4^i$, then $D_{11} \neq 0$.
We can therefore apply the derivation to the equality $\frac{D_{21}^q}{D_{11}^q} = \frac{D_{12}^{q^2}}{D_{11}^{q^2}}$, and get $\derivation\left(\frac{D_{21}^q}{D_{11}^q}\right) = 0$, so $\derivation (D_{11}^q)D_{21}^q - D_{11}^q\derivation(D_{21}^q) = 0$. 
All the factors $(a_i^j)^q$ and $(b_i^j)^q$ are in $k(E)^q$ and are therefore annihilated by the derivation, except possibly the factors $(a_0^j)^q$ since $(v^1)^q = \ell\left(D_2 + Q, D_4 + Q\right) \not\in k(E)^q$.
Writing $A_i^j = (a_i^j)^q$ and $B_i^j = (b_i^j)^q$, we get
\begin{align*}
\derivation (D_{11}^q) &= (B_1^2B_1^4)\left(\derivation (A_0^2)  A_1^4 - \derivation (A_0^4)  A_1^2\right),\\
\derivation (D_{21}^q) &= \derivation (A_0^2)B_1^2B_2^4A_1^4 - \derivation (A_0^4)B_1^4  B_2^2A_1^2.
\end{align*}
Then,
\begin{align*}
0 = \derivation (D_{11}^q)D_{21}^q - D_{11}^q\derivation(D_{21}^q) &= B_1^2B_1^4A_1^2A_1^4(\derivation (A_0^4)A_0^2- \derivation (A_0^2)A_0^4)   (B_1^4B_2^2-B_1^2B_2^4).
\end{align*}
It remains to prove that each factor on the right-hand side is a non-zero rational function. It is easy to see that each occurring $A_i^j$ and $B_i^j$ is non-zero.
Let us prove that $\derivation (A_0^4)A_0^2- \derivation (A_0^2)A_0^4 \neq 0$.
Writing $V_i^j = (v_i^j)^q$, we have
\begin{align*}
\derivation (A_0^4)A_0^2 - A_0^4\derivation(A_0^2) = &\  
(\derivation(V_1^1) V_2^1 - V_1^1 \derivation(V_2^1))(x(D_4)y(D_2)-y(D_4)x(D_2))^q\\
&-(\derivation(V_2^1)V_0^1  -  V_2^1\derivation(V_0^1))(y(D_2)-y(D_4))^q\\
&+(\derivation(V_0^1) V_1^1 - V_0^1 \derivation(V_1^1))(x(D_2)-x(D_4))^q.
\end{align*}
Fixing $D_4$, its term of highest pole at $D_2 = 0$ is
$$((\derivation(V_1^1) V_2^1 - V_1^1 \derivation(V_2^1))x(D_4)^q - (\derivation(V_2^1)V_0^1  -  V_2^1\derivation(V_0^1)))y(D_2)^q,$$
of order $3q + O(1)$, unless $((\derivation(V_1^1) V_2^1 - V_1^1 \derivation(V_2^1))x(D_4) - (\derivation(V_2^1)V_0^1  -  V_2^1\derivation(V_0^1)))$ is the zero function, which happens for finitely many $D_4$ because is has a pole of order $2q + O(1)$ at $D_4 = 0$. For the latter point, we are using the fact that $\derivation(V_1^1) V_2^1 - V_1^1 \derivation(V_2^1)$ is itself non-zero; indeed choosing as above the point $T$ such that $D_4 + Q = (0,0)$, it is easy to see that there is a derivation $\partial$ such that
$$\partial(V_1^1) V_2^1 - V_1^1 \partial(V_2^1) =  y(D_2 + Q) \partial(x(D_2 + Q)) -\partial(y(D_2 + Q)) x(D_2 + Q) \neq 0.$$
Finally, we prove in a similar way that $B_1^4B_2^2-B_1^2B_2^4 \neq 0$, choosing $T$ such that $D_4 + Q = (0,0)$ and observing that $B_1^4B_2^2-B_1^2B_2^4 = (u_0^1)^{q^2}y(D_2 + Q)^q$ is a non-zero function of $R$.
\end{proof}

\subsection{Irreducibility of $X_3$}
We can now prove the main result of this section.

\begin{prop}\label{prop:4to3irred}
For any divisor $D \in (\mathscr{D}_4 \setminus \mathscr{T}_4)(k)$, the curve $X_3$ contains an absolutely irreducible component defined over~$k$.
\end{prop}

\begin{proof}
We have shown that $\theta : X_1 \rightarrow X_0$ satisfies all the conditions of Proposition~\ref{prop:structure-of-proof}, so the result follows.
\end{proof}

\section{Avoiding traps}\label{sec:traps}
\noindent
Recall from Section~\ref{subsec:elimandzigzag} that the zigzag descent consists in applying the $4$--to--$3$ and $3$--to--$2$ eliminations recursively until all remaining divisors are in the factor base
$$\widetilde{\mathfrak F} = \{N_{\F_{q^{2^c}}/\F_q}(D) \mid D \in \Div_{k}(E,\mathscr I), D > 0, \deg(D) \leq 2\}.$$
Here, $c$ is the smallest integer so that the eliminations work for any extension $k = q^{2^i}$ where $i \geq c$.
We show in this section that $c = O(1)$ is an absolute constant.\\

Now, to show that eliminations work all the way down to the factor base, we need to show that traps can be avoided. Paradoxically, to avoid traps, we need to add more traps. Originally, a divisor is a trap if it cannot be eliminated into smaller degree divisors. These are traps of level $0$. Now, we want to call a divisor a trap also if it can be eliminated, but only into divisors that are themselves traps. We call these traps of level $1$, and so on.
For a rigorous definition, let $x_0 = 1$, $x_1 = x$ and $x_2 = y$ in $k[E]$, and for $n = 2$ or $3$
let $V_n = \mathrm{span}(x_i^qx_j \mid i,j < n)$ and $\Lambda_n = \mathrm{span}(x_i \mid i < n)$. 
Recall that the $4$--to--$3$ elimination arises from the relation
$\varphi(f) \equiv \psi(f)\mod \mathscr I$ for any $f \in V_3$. Indeed, when $f$ splits as a product of linear factors $f = \prod_{i = 1}^{q+1}L_i$, and applying the norm and the logarithm maps, we deduce
\begin{alignat*}{1}
\sum_{i = 1}^{q+1}\log(N_{k/\F_q}(L_i)) = \Log(N_{k/\F_q}(D)) + \Log(N_{k/\F_q}(D')) - 3\cdot[k:\F_q]\cdot\Log([Q]),
\end{alignat*}
for some divisor $D'$ of degree $2$. The sum on the left is referred to as the \emph{left-hand side} of the elimination, and the terms on the right are the \emph{right-hand side} of the elimination.
Similarly, for the $3$--to--$2$ elimination, we get relations of the form
\begin{alignat*}{1}
 \sum_{i = 1}^{q+1}\log(N_{k/\F_q}(L_i \circ \tau_P)) = &\ \Log(N_{k/\F_q}(D)) + \Log(N_{k/\F_q}([P']))\\
&\  - 2\cdot\Log(N_{k/\F_q}([-P])) - 2\cdot\Log(N_{k/\F_q}([-Q-P^{(q)}])),
\end{alignat*}
where the sum on the left is the \emph{left-hand side} of the elimination, and the terms on the right are the \emph{right-hand side} of the elimination.\\

Consider the morphisms
\begin{alignat*}{3}
\delta' &: \P(V_3) \times \P(\Lambda_3) \longrightarrow \mathscr D_3 : (f,u) \longmapsto \div(u) + 3[0_E],&&\\
\delta_i &: \P(V_2) \times \P(\Lambda_2) \times E \longrightarrow \mathscr D_4 : (f,u,P) \longmapsto \div(u\circ \tau_P) & \ +\ &\div\left((u\circ \tau_P)^{\left(q^{2^{i-1}}\right)}\right) \\
& &\ +\ &2[-P] + 2\left[-P^{\left(q^{2^{i-1}}\right)}\right].
\end{alignat*}
The intuition behind these morphisms is the following. Given any degree $4$ divisor $D$, the corresponding $X_1$ is a curve in $\P(V_3) \times \P(\Lambda_3)$. Suppose $f \in X_0$ splits as a product of linear polynomials $f = \prod_{i = 1}^{q+1}L_i$. For any such $f$, the preimages of $f$ in $X_1$ are the points $(f,L_i)$, and we have $\delta'(f,L_i) = \div(L_i) + 3[0_E]$. Therefore, $\delta'(X_1)$ contains all the degree $3$ divisors susceptible to appear on the left-hand side of the elimination. In particular, we wish to show that $\delta'(X_1)$ does not consist only of traps. Similarly, $\delta_i$ allows to capture the divisors susceptible to appear on the left-hand side of the $3$--to--$2$ elimination. Note that $\delta_i$ captures only the `positive' part of $\div(u\circ \tau_P)$; since the terms $\Log(N_{k/\F_q}([-P]))$ also appear on the right-hand side, we will account for them as terms of the right-hand side.\\

Consider the natural morphisms $\pi_3 : E^{3} \rightarrow \mathscr D_3 $ and $\pi_4 : E^4 \rightarrow \mathscr D_4$.
For any $i\geq 0$, let $T_3(i,0) = \pi_3^{-1}(\mathscr T_3)$ and $T_4(i,0) = \pi_4^{-1}(\mathscr T_4)$.
For any $i>0$, let $T_3(i,1)$ be the set of pairs $(P_1,P_2) \in E^2$ such that
$$\left(P_1,P_2,P_1^{\left(q^{2^{i-1}}\right)},P_2^{\left(q^{2^{i-1}}\right)}\right) \in T_4(i-1,0) \subset E^4.$$
For any $i>0$, let $T_4(i,1)$ be the set of pairs $(P_1,P_2) \in E^2$ such that
$$(P_1,P_2,-P_1-P_2) \in T_3(i,0) \subset E^3.$$
For any $1< j \leq 2i-1$ define $T_3(i,j) = T_4(i-1,j-1) \subset E^2$, and for any  $1< j \leq 2i$ define $T_4(i,j) = T_3(i,j-1) \subset E^2.$
Now, for every $i$, let

\begin{alignat*}{1}
T_3(i) &= \bigcup_{j = 1}^{2(i-c+1)-1} T_3(i,j),\text{ and } T_4(i) = \bigcup_{j = 1}^{2(i-c+1)} T_4(i,j).
\end{alignat*}
Finally, we can define traps at level $i$ as
\begin{alignat*}{1}
\mathscr T_3(i) &= \mathscr T_3 \cup \left\{\sum_{k=1}^3[P_k]\ \middle|\ \forall k\neq \ell, (P_k,P_\ell) \in T_3(i)\right\},\\
\mathscr T_4(i) &= \mathscr T_4 \cup \left\{\sum_{k=1}^4[P_k]\ \middle|\ \forall k\neq \ell, (P_k,P_\ell) \in T_4(i)\right\}.
\end{alignat*}
The following proposition suggests this is the correct notion of traps: if a divisor is not a trap at a certain level, then its eliminations do not all lead to traps at the level below (at least on the left-hand side).
Given a divisor $D$ of degree $3$ (respectively, of degree $4$), we write $X_1(D)$ for the corresponding curve $X_1$ as defined in Section~\ref{subsec:roadmap} (respectively, in Section~\ref{sec:4-to-3-elim}).

\begin{prop}\label{prop:nottrapimpliesnottraps}
For any $i \geq c$, 
\begin{enumerate}
\item if $D \in \mathscr D_4$ and $D \not\in \mathscr T_4(i)$, then $\delta'(X_1(D)) \not \subset \mathscr T_3(i)$, and
\item if $D \in \mathscr D_3$ and $D \not\in \mathscr T_3(i+1)$, then $\delta_{i+1}(X_1(D)) \not \subset \mathscr T_4(i)$.
\end{enumerate}
\end{prop}

\begin{proof}
Suppose $D \not\in \mathscr T_4(i)$. 
So there is a pair $(P_1,P_2) \not\in T_4(i)$ such that $[P_1] + [P_2]$ divides $D$. In particular, $(P_1,P_2) \not\in T_4(i,1)$ so
$[P_1]+[P_2]+[-P_1-P_2] \not\in \mathscr T_3.$
Also, for any $1 < j \leq 2(i-c+1)$, we have $(P_1,P_2) \not\in T_4(i,j) = T_3(i,j-1)$ so
$(P_1,P_2) \not\in T_3(i).$
From Section~\ref{subset:excep-points-x0}, $[P_1] + [P_2] + [-P_1-P_2] - 3[0_E]$ is the divisor of a linear factor of an exceptional point of $X_1(D))$.
Therefore, $$[P_1] + [P_2] + [-P_1-P_2] \in \delta'(X_1(D)) \setminus \mathscr T_3(i),$$
proving that $\delta'(X_1(D)) \not \subset \mathscr T_3(i)$.
The second point is proved in the same way.
\end{proof}

\subsection{Degree of trap subvarieties} Let $n \in \{3,4\}$. Embedding $E$ in $\P^2$, we can naturally see $E^n$ as a variety in $\left(\P^{2}\right)^n$. When referring to the degree of a subvariety of $E^n$, we refer to its degree through the Segre embedding. Alternatively, we could consider its degree in the projectivization of the affine patch $\A^{2n}$, and as long as the variety properly intersects the hyperplane at infinity, these two notions of degree differ by a factor $O(1)$.

The variety $\mathscr D_n$ can be seen as a subvariety of $\P\left({S}^n(\A^3)\right) \cong \P^{\binom{n+2}{2}-1}$, where ${S}^n(\A^3)$ is the $n$-th symmetric power of the vector space $\A^3$ (i.e., the image of $(\A^3)^{\otimes n}$ in the symmetric algebra $S(\A^3)$).
Each morphism $\pi_n : E^n \rightarrow \mathscr D_n$ is the restriction of the natural morphism $\left(\P^2\right)^n \rightarrow \P\left({S}^n(\A^3)\right)$.
We refer to the embedding $\mathscr D_n \subset \P^{\binom{n+2}{2}-1}$ when discussing the degree of a pure dimensional subvariety of $\mathscr D_n$ (it is pure dimensional of all the components have the same dimension).
An important observation is that
for any variety $\mathscr A \subset E^n$, the degree of $\mathscr A$ differs from the degree of $\pi_n(\mathscr A)\subset \mathscr D_n$ by a factor $O(1)$.
If a variety $\mathscr A$ has components of different dimension, we let $\mathrm{dim}(\mathscr A)$ be the dimension of the highest dimensional component, and $\deg(\mathscr A)$ be the smallest degree of a variety of pure dimension $\mathrm{dim}(\mathscr A)$ containing $\mathscr A$ (note that this notion of degree will only be used to obtain upper bounds on the number of rational points of a variety).
It is easy to see that with this notion of degree, we have $\deg(\mathscr T_3) = q^{O(1)}$ and $\deg(\mathscr T_4) = q^{O(1)}$.

\begin{lem}\label{lem:rec-traps}
For $i > 0$ and any $j > 0$,
\begin{alignat*}{1}
T_3(i,2j) &= T_4(i-j,1),\\
T_4(i,2j) &= T_3(i-j+1,1),\\
T_3(i,2j-1) &= T_3(i-j+1,1),\\
T_4(i,2j-1) &= T_4(i-j+1,1).
\end{alignat*}
\end{lem}

\begin{proof}
These identities easily follow from the recursive definitions of $T_3(i,j)$ and $T_4(i,j)$.
\end{proof}

\begin{lem}\label{lem:deg-traps}
There exists $c = O(1)$ such that for any $i \geq c$,
we have 
$\deg(T_4(i,1)) = q^{O(1)}$ and $\deg(T_3(i,1)) = q^{2^{i-1} + O(1)}$.
\end{lem}
\begin{proof}
The fact that $\deg(T_4(i,1)) = q^{O(1)}$ follows from $\deg(\mathscr T_3) = q^{O(1)}$.
For the case of $T_3(i,1)$, let $f_j(P_1,P_2,P_3,P_4)$ be the equations defining $\pi^{-1}(\mathscr T_4) \subset E^4$. By construction of $\mathscr T_4$, each of them has degree $q^{O(1)}$. Choosing any of these equations (at least a non-trivial one), say $f_1$, we have,
$$T_3(i,1) \subset \left\{(P_1,P_2) \ \middle|\ f_1\left(P_1,P_2,P_1^{\left(q^{2^{i-1}}\right)},P_2^{\left(q^{2^{i-1}}\right)}\right) = 0\right\} \subset E^2.$$
There is a constant $c$ such that for all $i>c$, the equation $$f_1\left(P_1,P_2,P_1^{\left(q^{2^{i-1}}\right)},P_2^{\left(q^{2^{i-1}}\right)}\right) = 0$$ is non-trivial. This equation has degree $q^{2^{i-1} + O(1)}$, and so does $T_3(i,1)$. 
\end{proof}

\begin{cor}
There exists $c = O(1)$ such that for any $i \geq c$,, we have $\deg(\mathscr T_3(i)) = q^{2^{i-1} + O(1)}$ and $\deg(\mathscr T_4(i)) = q^{2^{i-1} + O(1)}$.
\end{cor}
\begin{proof}
From Lemmata~\ref{lem:rec-traps} and~\ref{lem:deg-traps}, we deduce that
for any $i > 0$ and  $j > 0$,
\begin{alignat*}{1}
\deg(T_3(i,2j)) &= q^{O(1)},\\
\deg(T_4(i,2j)) &= q^{2^{i-j} + O(1)},\\
\deg(T_3(i,2j-1)) &= q^{2^{i-j} + O(1)},\\
\deg(T_4(i,2j-1)) &= q^{O(1)}.
\end{alignat*}
The result follows from the definitions of $\mathscr T_3(i)$ and $\mathscr T_4(i)$.
\end{proof}

\subsection{Degree $3$--to--$2$ elimination}
The following proposition allows to avoid traps appearing on the left-hand side during the $3$--to--$2$ elimination.
\begin{prop}\label{prop:notraps_lhs_32}
For any $i > 0$, if $D \in \mathscr D_3$ and $D \not\in \mathscr T_3(i)$, then $$|(\delta_i(X_1(D))\cap \mathscr T_4(i-1))(\F_{q^{2^{i-1}}})| \leq q^{\frac{3}{2} \cdot 2^{i-1} + O(1)}.$$
\end{prop}

\begin{proof}
Since $D \not\in \mathscr T_3(i)$, Proposition~\ref{prop:nottrapimpliesnottraps} implies that $\delta_i(X_1(D)) \not \subset \mathscr T_4(i-1)$.
Since $\delta_i(X_1(D))$ is absolutely irreducible, and $\mathscr T_4(i-1)$ is closed, they intersect properly.
Also, $\deg(\delta_i(X_1(D))) = q^{2^{i-1} + O(1)}$ and $\deg(\mathscr T_4(i-1)) = q^{2^{i-2} + O(1)}$. Applying B\'ezout's theorem,
$$|\delta_i(X_1(D))\cap \mathscr T_4(i-1)| \leq \deg(\delta_i(X_1(D))) \cdot \deg(\mathscr T_4(i-1)) = q^{2^{i-1} + 2^{i-2} + O(1)},$$
which proves the proposition.
\end{proof}

\begin{prop}[Degree $3$--to--$2$ elimination]\label{prop:3to2}
Consider the field $k = \F_{q^{2^{i}}}$ and a divisor $D \in (\mathscr {D}_3 \setminus \mathscr T_3(i)) (k)$. 
There exists $c = O(1)$ such that for any $i \geq c$, there is a probabilistic algorithm that finds a list $(D_j)_{j = 1}^{q+1}$ of effective divisors of degree $2$ over $k$, three divisors $D_1', D_2', D_3'$ of degree $1$ over $k$ and, integers $\alpha_1,\alpha_2,\alpha_3$ such that
$$\Log(N_{k/\F_q}(D)) = \sum_{j = 1}^{q+1}\log(N_{k/\F_q}(D_j)) + \sum_{i=1}^3\alpha_i\cdot \log(N_{\F_{q^{2^{i-1}}}/\F_q}(D_j')),$$
in expected time polynomial in $q$ and $2^{i}$. Furthermore, it ensures that
\begin{enumerate}
\item for any $D_j$, we have $N_{k/\F_{q^{2^{i-1}}}}(D_j) \not\in \mathscr T_4(i-1)$, and
\item for any $D_j'$, we have $N_{k/\F_{q^{2^{i-2}}}}(D'_j) \not\in \mathscr T_4(i-2)$.
\end{enumerate}
\end{prop}

\begin{proof}
Consider an affine patch $A$ of the ambient space (which intersects all the components of $X_3$), and the corresponding restriction $\widetilde X_3 \subset A$.
We have $\deg(\widetilde X_3) = q^{O(1)}$. 
From Proposition~\ref{prop:3to2-x3irred} and~\cite[Theorem~3.1]{Ba96}, we have $$\left|\widetilde X_3(k)\right| \geq  q^{2^{i}} -  q^{2^{i-1} + O(1)}.$$
The algorithm simply consists in generating random points of $\widetilde X_3(k)$, which can be done in polynomial time since the degree of the curve is polynomial in $q$. Each $(f,P,u_1,u_2,u_3) \in \widetilde X_3(k)$ gives a possible elimination, as described in Section~\ref{subsec3to2}. It only remains to prove that with high probability, no trap appears in the elimination.

Fix a linear factor $u$, and consider the subvariety $H_u$ of $A$ parameterising polynomials of which $u$ is a factor.
Let us show that $\widetilde X_3 \cap H_u$ contains at most $(q+1)\deg(\widetilde X_3) = q^{O(1)}$ points.
First, one cannot have $\widetilde X_3\subset H_u$ (or the exceptional points would form a subvariety of dimension $1$, so $D$ would be a trap), so $\widetilde X_3\cap H_u$ has dimension $0$.
Let $H_u' \subset H_u$ be the (degree 1) subspace of $A$ where $u = u_1$; the intersection $\widetilde X_3 \cap H_u'$ has dimension $0$ and contains at most $\deg(\widetilde X_3)$ points. If $(f,P) \in X_0 \setminus S$ and $u$ is a factor of $f$, there are $q^3-q$ points in $X_3$ projecting to $(f,P)$, and $q^2-q$ of them are in $H'_u$. Therefore, there are at most $(q+1)\deg(\widetilde X_3)$ points in $\widetilde X_3 \cap H_u$.
Similarly, we can bound the number of divisors coprime to $D$ occuring in the functions $\varphi_P(f)$, for $(f,P,u,v,w) \in X_3$, by looking at the hyperplanes $H_P = \P(V) \times \{P\} \times (\P^1)^3$ for $P \in E$: we have as previously that for any $P$, the intersection $\widetilde X_3 \cap H_P$ contains at most $\deg(\widetilde X_3)$ points, and the three points coprime to $D$ appearing in the divisor of $\varphi_P(f)$ are $P_0 = -2P - 2 Q - 2P^q -\sigma D$, $P_1 = -P$, and $P_2 = -Q-P^{(q)}$ (where $\sigma D \in E$ is the sum of the points of $D$). Since values of $P$ are in $O(q)$-to-$1$ correspondence with values of $P_\ell$ (for each $\ell \in \{0,1,2\}$), we deduce that any divisor coprime to $D$ appears in at most $O(q\deg(\widetilde X_3)) = q^{O(1)}$ of the functions $\varphi_P(f)$, for $(f,P,u,v,w) \in X_3$.

Each element in $X_3(k)$ gives a relation where the right-hand side is a divisor of the form $D + [P_0] - 2[P_1] - 2[P_2]$, and $P_\ell \in E(k)$.
Let $\ell \in \{0,1,2\}$. Ranging over all rational points $X_3(k)$, the point $P_\ell$ takes $q^{2^i + O(1)}$ distinct values. Any such point can be descended to $N_{k/\F_{q^{2^{i-2}}}}([P_\ell]) \in \mathscr D_4(\F_{q^{2^{i-2}}})$.
Applying~\cite[Theorem~3.1]{Ba96}, there are only $q^{3\cdot 2^{i-2} + 2^{i-3} + O(1)}$ such divisors that are traps.

Now, let us look at traps that could appear on the left-hand side.
The degree $4$ divisors that can appear on the left-hand side are $\delta_i(X_1(D))$.
Since $D \not\in \mathscr T_3(i)$, Proposition~\ref{prop:notraps_lhs_32} implies that 
$|(\delta_i(X_1(D))\cap \mathscr T_4(i-1))(\F_{q^{2^{i-1}}})| \leq q^{3 \cdot 2^{i-2} + O(1)}$. 
Therefore, at most $q^{3 \cdot 2^{i-2} + O(1)}$ points of $X_3(k)$ give rise to a trap on the left-hand side. 

Finally, if $G \subset \widetilde X_3(k)$ is the subset of points giving an elimination that does not involve traps on either side, we get
$$\left|\widetilde X_3(k) \setminus G\right| \leq q^{3\cdot 2^{i-2} + 2^{i-3} + O(1)} + q^{3 \cdot 2^{i-2} + O(1)} =q^{\frac 7 8 \cdot 2^{i} + O(1)}.$$
Therefore, for $i$ larger than some absolute constant, more than half the points of $X_3(k)$ are in $G$, so choosing uniformly random points in $X_3(k)$, the elimination succeeds in expected polynomial time in $q$ and $2^{i}$.
\end{proof}

\subsection{Degree $4$--to--$3$ elimination}
The following proposition allows to avoid traps appearing on the left-hand side during the $4$--to--$3$ elimination.
\begin{prop}\label{prop:notraps_lhs_43}
For any $i \geq c$, if $D \not\in \mathscr T_4(i)$, then 
$$|(\delta'(X_1(D))\cap \mathscr T_3(i))(\F_{q^{2^i}})| \leq q^{2^{i-1} + O(1)}.$$
\end{prop}

\begin{proof}
Since $D \not\in \mathscr T_4(i)$ and $i \geq c$, we have $\delta'(X_1(D)) \not \subset \mathscr T_3(i)$. Since $\delta'(X_1(D))$ is absolutely irreducible and $\mathscr T_3(i)$ is closed, they intersect properly so $\dim(\delta'(X_1(D))\cap \mathscr T_3(i)) = 0$.
Now, $\deg(\delta'(X_1(D))) = q^{O(1)}$ and
$\deg(\mathscr T_3(i)) = q^{2^{i-1} + O(1)}$.
Applying B{\'e}zout's theorem, 
$$|(\delta'(X_1(D))\cap \mathscr T_3(i))(\F_{q^{2^i}})| \leq \deg(\delta'(X_1(D)))\deg(\mathscr T_3(i)) = q^{2^{i-1} + O(1)}.$$
\end{proof}
The following results allow to avoid traps on the right-hand side during the $4$--to--$3$ elimination.

\begin{lem}\label{lem:notalltraps0}
For any $S,R\in E$ such that $S^{(q)} \not \in \{S-Q,S+2Q\}$, we have $\mathscr P_2(S) + \mathscr P_2(R) \not\subset \mathscr T_4$ and $\mathscr P_2(S) + [-S] \not\subset \mathscr T_3$. 
\end{lem}

\begin{proof}
This is simply a summary of Lemmata~\ref{lem:notalltraps02to3},~\ref{lem:notalltraps00} and~\ref{lem:trapsDDq}.
\end{proof}

\begin{lem}
As long as the order of $Q$ is not a power of two (which can be enforced), there is no $S \in E(\F_{q^{2^{i}}})$ such that $S^{(q)} \in \{S-Q, S+2Q\}$.
\end{lem}
\begin{proof}
Suppose $S^{(q)} = S+jQ$ for $j \in \{-1,2\}$. Then, for any integer $r$, $S^{(q^r)} = S+rjQ$. The smallest $r$ such that $S^{(q^r)} = S$ is the order of $jQ$, which is not a power of two. So $S$ cannot be defined over a power-of-two degree extension of $\F_{q}$, i.e., it cannot be defined over a field $\F_{q^{2^{i}}}$.
\end{proof}

For any positive integer $i$ and any $P \in E$, let $\mathscr B_i(S) = \left\{F + F^{\left(q^{2^i}\right)} \ \middle|\ F \in \mathscr P_2(S)\right\}$.

\begin{lem}\label{lem:notalltraps0frob}
Let $S \in E(\F_{q^{2^{i}}})$.
There exists $c = O(1)$ such that for any $i \geq c$, we have that $\mathscr B_{i-1}(S) \not \subset \mathscr T_4$.
\end{lem}

\begin{proof}
We show that there exists $c = O(1)$ such that for any $i \geq c$, there exists $D \in \mathscr  P_2(S)$ such that $D + D^{\left(q^{2^{i-1}}\right)} \not \in \mathscr T_4$.
Let $\mathscr A = \left(\mathscr P_2(S) + \mathscr P_2\left(S^{\left(q^{2^{i-1}}\right)}\right)\right) \cap \mathscr T_4$.
Since $\mathscr P_2(S) + \mathscr P_2\left(S^{\left(q^{2^{i-1}}\right)}\right)$ is an absolutely irreducible surface and is not contained in $\mathscr T_4$ (which is closed), the intersection $\mathscr A$ is a curve.
We have
$$|\mathscr A(\F_{q^{2^{i-1}}})| \leq c(\mathscr A) (q^{2^{i-1}} + 1 + \deg(\mathscr A)^2q^{2^{i-2}}) \leq q^{2^{i-1} + O(1)},$$
where $c(\mathscr A)$ is the number of absolutely irreducible components of $\mathscr A$.
On the other hand, observe that through the morphism $\mathscr P_2(S) \rightarrow \mathscr B_{i-1}(S)$, each point has at most $4$ preimages, so
$$|\mathscr B_{i-1}(S)(\F_{q^{2^{i-1}}})| \geq |\mathscr P_2(S)(\F_{q^{2^{i}}})|/4 = q^{2^{i} + O(1)}.$$
Therefore there exists $c = O(1)$ such that for any $i \geq c$, $|\mathscr A(\F_{q^{2^{i-1}}})| < |\mathscr B_{i-1}(S)(\F_{q^{2^{i-1}}})|$, hence 
 $\mathscr B_{i-1}(S) \not\subset \mathscr A$. Since $\mathscr B_{i-1}(S) \cap \mathscr T_4 \subset \mathscr A$, we deduce that for $i \geq c$, we have $\mathscr B_{i-1}(S) \not\subset \mathscr T_4$.
\end{proof}

\begin{lem}\label{lem:notalltrapsij}
Let $S \in E(\F_{q^{2^i}})$, and let $A(S) = \{(P,S-P) \mid P \in E\} \subset E^2$.
There exists $c = O(1)$ such that for any $i$ and $j$ such that $i-\lfloor j/2\rfloor +1 \geq d$, we have $A(S) \not \subset  T_4(i,j)$.
\end{lem}

\begin{proof}
From Lemma~\ref{lem:rec-traps}, it suffices to prove that there exists $c = O(1)$ such that $A(S) \not \subset T_4(i,1)$ and $A(S) \not \subset T_3(i,1)$ for any $i \geq  c$.
Since $\mathscr P_2(S) + [-S] \not\subset \mathscr T_3$, there is a divisor $[P] + [S-P] \in \mathscr P_2(S)$ such that $[P] + [S-P] + [-S] \not\in \mathscr T_3$ which by definition implies that $(P,S-P) \not \in T_4(i,1)$.
Also, from Lemma~\ref{lem:notalltraps0frob}, we can choose $c = O(1)$ which ensures that $\mathscr B_{i-1}(S) \not \subset \mathscr T_4$, so there exists $F = [P] + [S-P] \in \mathscr  P_2(S)$ such that $F + F^{\left(q^{2^{i-1}}\right)} \not \in \mathscr T_4$, which implies that $(P,S-P) \not \in T_3(i,1)$.
\end{proof}

\begin{lem}\label{lem:notalltrapsleveli}
Let $S \in E(\F_{q^{2^{i}}})$.
There exists $c = O(1)$ such that for any $i \geq c$, we have that $\mathscr B_{i-1}(S)\not \subset \mathscr T_4(i-1).$
\end{lem}

\begin{proof}
Let $c$ be the maximum between the constants $c$ of Lemmata~\ref{lem:notalltraps0frob} and~\ref{lem:notalltrapsij}.
Recall that $$\mathscr T_4(i-1) = \mathscr T_4 \cup \left\{\sum_{k=1}^4[P_k]\ \middle|\ \forall k\neq \ell, (P_k,P_\ell) \in T_4(i-1)\right\}.$$
From Lemma~\ref{lem:notalltraps0frob}, $\mathscr B_{i-1}(S) \not \subset \mathscr T_4$.
From Lemma~\ref{lem:notalltrapsij}, there exists a divisor $F = [P] + [S-P] \in \mathscr P_2(S)$ such that $(P,S-P) \not \in T_4(i-1)$, so   $$F + F^{\left(q^{2^{i-1}}\right)} \in \mathscr B_{i-1}(S) \setminus \left\{\sum_{k=1}^4[P_k]\ \middle|\ \forall k\neq \ell, (P_k,P_\ell) \in T_4(i-1)\right\}.$$
The absolute irreducibility of $\mathscr B_{i-1}(S)$ (it is an image of $\mathscr P_2(S)$) allows to conclude.
\end{proof}

\begin{prop}\label{prop:avoidtraps_rhs}
Let $S \in E(\F_{q^{2^{i}}})$.
There exists $c = O(1)$ such that for any $i \geq c$, $|\mathscr B_{i-1}(S) \cap \mathscr T_4(i-1)| \leq q^{\frac{3}{2}2^{i-1} + O(1)}$.
\end{prop}

\begin{proof}
From Lemma~\ref{lem:notalltrapsleveli}, we have $\mathscr B_{i-1}(S) \not \subset \mathscr T_4(i-1)$. Since $\mathscr B_{i-1}(S)$ is absolutely irreducible and $\mathscr T_4(i-1)$ is closed, we have $\dim(\mathscr B_{i-1}(S) \cap \mathscr T_4(i-1)) < \dim(\mathscr B_{i-1}(S)) = 1$. Therefore, from B{\'e}zout's theorem,
$$|\mathscr B_{i-1}(S) \cap \mathscr T_4(i-1)| \leq \deg(\mathscr B_{i-1}(S))\deg(\mathscr T_4(i-1)) = q^{2^{i-1} + 2^{i-2} + O(1)} = q^{\frac{3}{2}2^{i-1} + O(1)}.$$
\end{proof}

\begin{prop}[Degree $4$--to--$3$ elimination]\label{prop:4to3}
Consider the field $k = \F_{q^{2^{i}}}$ and a divisor $D \in (\mathscr {D}_4 \setminus \mathscr T_4(i)) (k)$. 
There exists $c = O(1)$ such that for any $i \geq c$, there is a probabilistic algorithm that finds a list $(D_j)_{j = 1}^{q+1}$ of effective divisors of degree $3$ over $k$, and one effective divisor $D'$ of degree $2$ over $k$ such that
$$\Log(N_{k/\F_q}(D))
 = \sum_{j = 1}^{q+1}\Log(N_{k/\F_q}(D_j)) - \Log(N_{k/\F_q}(D')) + 3\cdot2^{i}\cdot\Log([Q]),$$
and runs in expected time polynomial in $q$ and $2^{i}$. Furthermore, it ensures that $D_i \not\in  \mathscr T_3(i)$ for each index $i$, and $N_{k/\F_{q^{2^{i-1}}}}(D') \not\in \mathscr T_4(i-1)$.
\end{prop}

\begin{proof}
This proof is similar to the proof of Proposition~\ref{prop:3to2}.
We consider an affine patch $A$ of the ambient space, and the corresponding $\widetilde X_3$, and we have $\deg(\widetilde X_3) = q^{O(1)}$.
From Proposition~\ref{prop:4to3irred} and~\cite[Theorem~3.1]{Ba96}, we have $$\left|\widetilde X_3(k)\right| \geq  q^{2^{i}} -  q^{2^{i-1} + O(1)}.$$
As in the $3$--to--$2$ case, the algorithm consists in generating random points of $\widetilde X_3(k)$. Each $(f,u_1,u_2,u_3) \in \widetilde X_3(k)$ gives a possible elimination, as described in Section~\ref{subsec:overview-4-3}, and it remains to prove that with high probability, no trap appears in the elimination.

Fix a linear factor $u$, and consider the subvariety $H_u$ of $A$ parameterising polynomials of which it is a factor. One can prove in the same way as in Proposition~\ref{prop:3to2} that $\widetilde X_3 \cap H_u$ contains at most $(q+1)\deg(\widetilde X_3)$ points. We now prove that any divisor coprime to $D$, $[0_E]$ and $[-Q]$ appears in at most $\deg(\widetilde X_3)$ of the functions $\varphi(f)$, for $(f,u_1,u_2,u_3) \in \widetilde X_3$.
Indeed, the divisor of $\varphi(f)$ is of the form $D + D' - 3[0_E] - 3[-Q]$, with $D' \in \mathscr P_2(-\sigma D-3Q)$ where $\sigma D$ is the sum of the points of $D$. Therefore, a divisor coprime to $D$, $[0_E]$ and $[-Q]$ appears for two polynomials $f$ and $g$ if and only if $\varphi(f)$ and $\varphi(g)$ only differ by a scalar factor. The subvariety $H_{D'}$ of $A$ of points $(f,u_1,u_2,u_3)$ where $\div(\varphi(f)) = D + D' - 3[0_E] - 3[-Q]$ is of degree~$1$, and does not contain $\widetilde X_3$, so the intersection $\widetilde X_3\cap H_{D'}$ contains at most $\deg(\widetilde X_3)$ elements.

Proposition~\ref{prop:avoidtraps_rhs} implies that at most $q^{\frac{3}{2}2^{i-1} + O(1)}$ divisors $D' \in \mathscr P_2(-\sigma D-3Q)$ give rise to a trap at level $i-1$ on the right-hand side. So at most $\deg(\widetilde X_3)q^{\frac{3}{2}2^{i-1} + O(1)} = q^{\frac{3}{2}2^{i-1} + O(1)}$ points of $X_3$ give rise to a trap on the right-hand side.

Now, let us look at traps that could appear on the left-hand side.
The degree $3$ divisors that can appear on the left-hand side are $\delta'(X_1(D))$.
Since $D = F + F^{\left(q^{2^{i}}\right)}$ for some $F \in \mathscr D_2$ and $D \not\in \mathscr T_4(i)$, Proposition~\ref{prop:notraps_lhs_43} implies that 
$|(\delta'(X_1(D))\cap \mathscr T_3(i))(\F_{q^{2^i}})| \leq q^{2^{i-1} + O(1)}$. Therefore, at most $q^{2^{i-1} + O(1)}$ points of $X_3$ give rise to a trap on the left-hand side.

Finally, if $G \subset \widetilde X_3(k)$ is the subset of points giving an elimination that does not involve traps on either side, we get
$$\left|\widetilde X_3(k) \setminus G\right| \leq q^{3 \cdot 2^{i-2} + O(1)} + q^{ 2^{i-1} + O(1)} =q^{\frac 3 4 \cdot 2^{i} + O(1)}.$$
Therefore, for $i$ larger than some absolute constant, more than half the points of $X_3(k)$ are in $G$, so choosing uniformly random points in $X_3(k)$, the elimination succeeds in expected polynomial time in $q$ and $2^{i}$.

\end{proof}

\section{Proof of the main theorem}
\begin{lem}\label{lem:chebo} Given a polynomial $F \in \F_{q}[E]$, there is a probabilistic polynomial-time algorithm that finds an irreducible polynomial $G \in \F_{q}[E]$ of degree $2^{e+2}$ such that $G \equiv F \mod \mathscr I$, for some integer $e = \log_2(n) + O(1)$. Furthermore, $G = N_{\F_{q^{2^{e}}}/\F_{q}}(D)$
for some irreducible  divisor $D \in (\mathscr D_4 \setminus \mathscr T_4(e))(\F_{q^{2^{e}}})$.
\end{lem}

\begin{proof}
This is an application of the Chebotarev density theorem for function fields.
Let $H(\mathscr I)$ be the ray class field modulo $\mathscr I$ of $\F_{q}(E)$, and $\varphi : \Cl_{\mathscr I} \rightarrow \Gal(H(\mathscr I)/\F_{q}(E))$ the Artin map from the ray class group. Recall that $\Cl_{\mathscr I} = D(\mathscr I)/ P(\mathscr I)$ where $D(\mathscr I)$ is the group of fractional ideals of $\F_{q}[E]$ coprime to $\mathscr I$ and $P(\mathscr I)$ is the subgroup of principal ideals generated by elements $f \in \F_{q}[E]$ such that $f \equiv 1 \mod \mathscr I$. From \cite[p.~520]{Sal06}, $\varphi$ is an isomorphism. 

Let $e > \log_2(n)-1$ be an integer, and
pick a uniformly random function $f \in \F_{q}[E]$ of degree $2^{e+2}$ such that $f \equiv 0 \mod \mathscr I$. Let $G = F + f$.
Then, $G \equiv F \mod \mathscr I$, and $G$ is uniformly distributed among the functions of degree $2^{e+2}$ in the $\mathscr I$-ray class of $F$.
Recall that $n = \deg(\mathscr I)$ and $N = \#E(\F_q)$. Let $S_{\F_q}(E,\mathscr I)$ be the set of irreducible divisors of $E$ other than $\mathscr I$, defined over $\F_q$. Applying the Chebotarev density theorem~\cite[Theorem~9.13B]{Rosen} to $H(\mathscr I)/\F_{q}(E)$, we get that for any $d > 0$,
\begin{align*}
\# \{P \in S_{\F_q}(E,\mathscr I) \mid \deg(P) = d, [P]_{\mathscr I} = [F]_{\mathscr I}\}
& = \frac{1}{\#\Cl_{\mathscr I}}\frac{q^d}{d} + O\left(\frac{q^{d/2}}{d}\right),
\end{align*}
where $[-]_{\mathscr I}$ denotes the $\mathscr I$-ray class.
Let $d = 2^{e+2}$. Since $\#\Cl_{\mathscr I} = N(q^n-1)/(q-1)$, we get
\begin{align*}
\# \{P \in S_{\F_q}(E,\mathscr I) \mid \deg(P) = d, [P]_{\mathscr I} = [F]_{\mathscr I}\}
& = \frac{q^{2^{e+2} - n + O(1)}}{2^{e+2}}.
\end{align*}
On the other hand, applying~\cite[Theorem~3.1]{Ba96}, we have $|\mathscr T_4(e)(\F_{q^{2^{e}}})| = q^{3\cdot2^{e} + O(1)}$. So for $e = \log_2(n) + O(1)$, the random prime divisor $G$ is not a trap with overwhelming probability.
\end{proof}

Let $c = O(1)$ be the smallest integer such that both degree $4$--to--$3$ and $3$--to--$2$ eliminations from Propositions~\ref{prop:4to3} and~\ref{prop:3to2} are guaranteed to work for $i \geq c$. Let
$$\widetilde{\mathfrak F} = \{N_{\F_{q^{2^c}}/\F_q}(D) \mid D \in \Div_{k}(E,\mathscr I), D > 0, \deg(D) \leq 2\}.$$
The factor base for the descent algorithm is defined as
$$\mathfrak F = \{f \in \F_q[E] \mid \exists D \in \widetilde{\mathfrak F} \text{ such that } \mathrm{div}(f) = ND - \deg(f)[0_E]\}.$$

\begin{prop}[Zigzag descent]\label{prop:zigzag} Given a polynomial $F \in \F_{q}[E]$, there is a probabilistic algorithm that finds integers $(\alpha_f)_{f \in \mathfrak F}$ such that
$$\log(F) = \sum_{f \in \mathfrak F} \alpha_f \cdot \log(f),$$
and that runs in expected time $q^{2\log_2(n) + O(1)}$.
\end{prop}

\begin{proof}
First apply Lemma~\ref{lem:chebo} to find an irreducible polynomial $G$ in $\F_{q}[E]$ of degree $2^{e+2}$ such that $G \equiv F \mod \mathscr I$, and such that
$\log(G) = \Log(N_{\F_{q^{2^{e}}}/\F_{q}}(D))$
for some irreducible  divisor $D \in (\mathscr D_4 \setminus \mathscr T_4(e))(\F_{q^{2^{e}}})$. Applying the degree $4$--to--$3$ elimination (Proposition~\ref{prop:4to3}), there is a list $(D_i)_{i = 1}^{q+1}$ of effective divisors of degree $3$ over $\F_{q^{2^{e}}}$ and an effective divisor $D'$ of degree $2$ over $\F_{q^{2^{e}}}$ such that
$$\Log(N_{\F_{q^{2^{e}}}/\F_q}(D))
 = \sum_{i = 1}^{q+1}\Log(N_{\F_{q^{2^{e}}}/\F_q}(D_i)) - \Log(N_{\F_{q^{2^{e}}}/\F_q}(D')) + 3\cdot2^{e}\cdot\Log([Q]).$$
Since $D_i \in \mathscr {D}_3 \setminus \mathscr T_3(e)$, one can apply the degree $3$--to--$2$ elimination (Proposition~\ref{prop:3to2}), rewriting each of them as combinations of smaller degree polynomials. At this stage, the quantity $\Log(N_{\F_{q^{2^{e}}}/\F_q}(D))$ is expressed as a product of $O(q^2)$ terms involving divisors of degree $1$ or $2$ over $\F_{q^{2^{e}}}$.
They give irreducible divisors of degree $4$ by considering the norm to $\F_{q^{2^{e-1}}}$ or $\F_{q^{2^{e-2}}}$ (and these divisors do not belong to $\mathscr T_4(e-1)$ or $\mathscr T_4(e-2)$ respectively), hence one can recursively apply the degree $4$--to--$3$ and $3$--to--$2$ eliminations, until all the resulting divisors are in the set $\widetilde {\mathfrak F}$.
We obtain a linear combination of logarithms of factor base elements via the fact that for any $D \in \widetilde {\mathfrak F}$, we have $\Log(D) = \log(f)/N$, where $f$ is any function such that $\mathrm{div}(f) = ND$.
\end{proof}

\subsection{Proof of the main theorem} Theorem~\ref{thm:main} follows immediately from Theorem~\ref{thm:descentsufficient}, Theorem~\ref{thm:good-rep}, and Proposition~\ref{prop:zigzag}.

\section{Acknowledgements}
The authors wish to thank Zsolt Patakfalvi for discussions that led to the proof of Proposition~\ref{prop:patak}, and Arjen K. Lenstra for valuable comments that helped improve the quality of this manuscript.
Part of this work was supported by the Swiss National Science Foundation under grant number 200021-156420, and by the ERC Advanced Investigator Grant 740972 (ALGSTRONGCRYPTO).

\bibliographystyle{alpha}
\bibliography{bib}

\begin{thebibliography}{BGJT14}

\bibitem[Bac96]{Ba96}
Eric Bach.
\newblock Weil bounds for singular curves.
\newblock {\em Applicable Algebra in Engineering, Communication and Computing},
  7(4):289--298, 1996.

\bibitem[BGJT14]{BGJT14}
Razvan Barbulescu, Pierrick Gaudry, Antoine Joux, and Emmanuel Thom\'e.
\newblock A heuristic quasi-polynomial algorithm for discrete logarithm in
  finite fields of small characteristic.
\newblock In Phong~Q. Nguyen and Elisabeth Oswald, editors, {\em Advances in
  Cryptology -- {EUROCRYPT} 2014}, volume 8441 of {\em Lecture Notes in
  Computer Science}, pages 1--16. Springer, 2014.

\bibitem[Die11]{diem_2011}
Claus Diem.
\newblock On the discrete logarithm problem in elliptic curves.
\newblock {\em Compositio Mathematica}, 147(1):75--104, 2011.

\bibitem[EG02]{EG02}
Andreas Enge and Pierrick Gaudry.
\newblock {A general framework for subexponential discrete logarithm
  algorithms}.
\newblock {\em {Acta Arithmetica}}, 102:83--103, 2002.

\bibitem[Ful13]{Fulton}
William Fulton.
\newblock {\em Intersection theory}, volume~2.
\newblock Springer Science \& Business Media, 2013.

\bibitem[GKZ18]{GKZ}
Robert Granger, Thorsten Kleinjung, and Jens Zumbr{\"a}gel.
\newblock On the discrete logarithm problem in finite fields of fixed
  characteristic.
\newblock {\em Transactions of the American Mathematical Society},
  270(5):3129--3145, 2018.

\bibitem[KW18]{KW18}
Thorsten Kleinjung and Benjamin Wesolowski.
\newblock A new perspective on the powers of two descent for discrete
  logarithms in finite fields.
\newblock In {\em Thirteenth Algorithmic Number Theory Symposium -- ANTS-XIII},
  2018.
\newblock proceedings to appear in the Open Book Series, Mathematical Sciences
  Publishers.

\bibitem[Mic19]{Mic19}
Giacomo Micheli.
\newblock On the selection of polynomials for the {DLP} quasi-polynomial time
  algorithm for finite fields of small characteristic.
\newblock {\em SIAM Journal on Applied Algebra and Geometry}, 3(2):256--265,
  2019.

\bibitem[Pom87]{Pom87}
Carl Pomerance.
\newblock Fast, rigorous factorization and discrete logarithm algorithms.
\newblock In David~S. Johnson, Takao Nishizeki, Akihiro Nozaki, and Herbert~S.
  Wilf, editors, {\em Discrete Algorithms and Complexity}, pages 119--143.
  Academic Press, 1987.

\bibitem[Ros13]{Rosen}
Michael Rosen.
\newblock {\em Number theory in function fields}, volume 210.
\newblock Springer Science \& Business Media, 2013.

\bibitem[Sal06]{Sal06}
Gabriel Daniel~Villa Salvador.
\newblock {\em Topics in the theory of algebraic function fields}.
\newblock Springer Science \& Business Media, 2006.

\bibitem[Sil86]{Silverman-Arithmetic}
Joseph~H. Silverman.
\newblock {\em The Arithmetic of Elliptic Curves}, volume 106 of {\em Gradute
  Texts in Mathematics}.
\newblock Springer-Verlag, 1986.

\bibitem[Wat69]{Waterhouse1969}
William~C. Waterhouse.
\newblock Abelian varieties over finite fields.
\newblock {\em Annales scientifiques de l'\'Ecole Normale Sup\'erieure},
  2(4):521--560, 1969.

\bibitem[Wes18]{WesThesis}
Benjamin Wesolowski.
\newblock {\em Arithmetic and geometric structures in cryptography}.
\newblock PhD thesis, EPFL, 2018.

\end{thebibliography}

\end{document}